\documentclass[11pt]{amsart}
\usepackage{euscript}
\usepackage{amssymb}
\usepackage{amsmath}
\usepackage{epic}
\usepackage[dvips]{graphicx}
\usepackage{epsfig}
\usepackage{color}
\usepackage{amscd}
\usepackage{xspace}
\usepackage{latexsym}

\usepackage[matrix,arrow,curve]{xy}

\numberwithin{equation}{section}

\newtheoremstyle{my}{1.5em}{0.5em}{\em}{}{\sc}{.}{0.5em}{}

\theoremstyle{my}

\theoremstyle{my}
\newtheorem{thm}{Theorem}[section]
\newtheorem{Theorem}[thm]{Theorem}
\newtheorem*{Theorem*}{Theorem}
\newtheorem{Corollary}[thm]{Corollary}

\newtheorem*{corollary*}{Corollary}
\newtheorem{Lemma}[thm]{Lemma}
\newtheorem{lem}[thm]{Lemma}

\newtheorem{Proposition}[thm]{Proposition}

\newtheorem*{conjecture*}{Conjecture}

\newtheorem*{question*}{Question}
\newtheorem{defn}[thm]{Definition}
\newtheorem{Definition}[thm]{Definition}

\newtheorem{Hypothesis}[thm]{Hypothesis}
\newtheorem*{definitions*}{Definitions}
\newtheorem{rem}[thm]{Remark}
\newtheorem*{rem*}{Remark}
\newtheorem{Remark}[thm]{Remark}

\newtheorem*{remark*}{Remark}

\newtheorem*{remarks*}{Remarks}
\newtheorem*{example*}{Example}

\newtheorem*{examples*}{Examples}

\newtheorem*{convention*}{Convention}
\newtheorem*{conventions*}{Conventions}
\newtheorem*{Note*}{Note}
\newtheorem*{exercise*}{Exercise}
\newtheorem*{bibliographical-note*}{Bibliographical note}

\newcommand{\Acknowledgements}{{\em Acknowledgements.} }

\def\co{\colon\thinspace}
\newcommand{\nc}{{\bf nc}}
\newcommand{\bR}{\mathbb{R}}
\newcommand{\bZ}{\mathbb{Z}}
\newcommand{\bQ}{\mathbb{Q}}
\newcommand{\bC}{\mathbb{C}}

\newcommand{\bP}{\mathbb{P}}

\newcommand{\bk}{\mathbf{k}}

\newcommand{\iso}{\cong}           

\newcommand{\cdbar}{\mathrm{\overline{\partial}}}
\newcommand{\Sym}{\mathrm{Sym}}
\newcommand{\Symp}{\mathrm{Symp}}

\newcommand{\id}{\mathrm{id}}

\newcommand{\im}{\mathrm{im}}
\renewcommand{\ker}{\mathrm{ker}}

\newcommand{\Slice}{\mathcal{S}}
\newcommand{\Hom}{\mathrm{Hom}}
\newcommand{\Aut}{\mathrm{Aut}}

\newcommand{\Kh}{\mathrm{Kh}}
\newcommand{\Br}{\mathrm{Br}}

\newcommand{\Id}{\mathrm{Id}}
\newcommand{\wt}{\mathrm{wt}}
\newcommand{\Ob}{\mathrm{Ob} \,}

\newcommand{\Diff}{\mathrm{Diff}}

\newcommand{\Hilb}{\mathrm{Hilb}}
\newcommand{\Conf}{\textrm{Conf}}

\newcommand{\scrA}{\EuScript{A}}
\newcommand{\scrB}{\EuScript{B}}
\newcommand{\scrY}{\EuScript{Y}}

\newcommand{\scrR}{\EuScript{R}}

\newcommand{\scrX}{\EuScript{X}}

\newcommand{\scrM}{\EuScript{M}}

\newcommand{\scrF}{\EuScript{F}}

\newcommand{\wpb}{\wp_{\bullet}}

\numberwithin{equation}{section}
\newcommand{\superscript}[1]{\ensuremath{^{\textrm{#1}}} }

\renewcommand{\th}[0]{\superscript{th}}
\newcommand{\st}[0]{\superscript{st}}

\newcommand{\comment}[1]{}

\title{The symplectic arc algebra is formal}
\author{Mohammed Abouzaid}
\thanks{This research was partially carried out during the time M.A. served as a Clay Research Fellow. M.A. was also partially supported by NSF grant DMS-1308179}
\address{Mohammed Abouzaid, Columbia University, 2960 Broadway Ave, New York, NY 10027,  U.S.A.}
\author{Ivan Smith}
\thanks{I.S. was partially supported by grant ERC-2007-StG-205349 from the European Research Council.}
\address{Ivan Smith, Centre for Mathematical Sciences, University of Cambridge, Wilberforce Road, CB3 0WB, England.}

\date{v1 - November 2013.  v2 - April 2015.  v3 - June 2015. v4 - October 2015.}

\begin{document}

\begin{abstract}

We prove a formality theorem for the Fukaya categories of the symplectic manifolds   underlying symplectic Khovanov cohomology, over fields of characteristic zero. The key ingredient is the construction of a degree one Hochschild cohomology class on a Floer $A_{\infty}$-algebra associated to the $(k,k)$-nilpotent slice $\scrY_k$, obtained by counting holomorphic discs which satisfy a suitable conormal  condition at infinity in a partial compactification $\bar{\scrY}_k$. The space $\bar{\scrY}_k$ is obtained as the Hilbert scheme of a partial compactification of the $A_{2k-1}$-Milnor fibre.  A sequel to this paper will prove formality of the symplectic cup and cap bimodules, and infer that  symplectic Khovanov cohomology and Khovanov cohomology have the same total rank over characteristic zero fields. 
\end{abstract}
\maketitle

\begin{small}
\tableofcontents
\end{small}
\parindent0em
\parskip1em
\section{Introduction}\label{Sec:Introduction}

Khovanov cohomology associates to an oriented link  $\kappa \subset S^3$ a bigraded group $\Kh^{\ast,\ast}(\kappa)$, whose Euler characteristic $\sum_{i,j} (-1)^i q^j rk_{\bQ} \Kh^{i,j}(\kappa)$, suitably normalised,  is the Jones polynomial of $\kappa$.  The invariant, which is defined via diagrammatic combinatorics \cite{Khovanov} or representation theory \cite{Khovanov:functor, Stroppel}, is interesting for several reasons. First, it is effective, in particular it distinguishes the unknot \cite{KrMr}; second, it is functorial under surface cobordisms in $\bR^4$, which gives rise to applications to four-dimensional topology \cite{Rasmussen}; and third, it has a number of formal features in common with gauge theoretic and Floer theoretic invariants in low-dimensional topology, which lead  to comparison spectral sequences relating Khovanov cohomology and Heegaard Floer theory or instanton Floer theory \cite{OzSz, KrMr}.  In these papers Khovanov cohomology appears as an algebraic approximation of a geometric invariant, rather than being given a geometric interpretation in its own right. 

Nonetheless, several geometric models of Khovanov cohomology have been proposed, and in some cases proved \cite{CK}.  An early such -- symplectic Khovanov cohomology, denoted henceforth $\Kh^{\ast}_{symp}(\kappa)$ -- gave rise to a singly graded link invariant defined using  the symplectic topology of certain spaces of matrices arising  in Lie theory \cite{SS}.  The purpose of this paper and its sequel \cite{AbSm:Kh=Floer} is to revisit that construction, and to establish a proof over $\bQ$ of the conjectured relationship \cite[Conjecture 2]{SS} of $\Kh^{\ast}_{symp}(\kappa)$ to Khovanov cohomology.

To state the results, we introduce some notation.  Let $\scrY_k$ denote a transverse slice to the adjoint quotient map $\chi: \frak{sl}_{2k}(\bC) \rightarrow \bC^{2k-1}$ at a nilpotent matrix with two equal Jordan blocks. Typical such slices are provided by the Jacobson-Morozov theorem; an explicit slice better tailored to our needs, following \cite{SS}, is given in Equation \eqref{eq:y-matrix} below.  Being an affine variety, $\scrY_k$ inherits an exact K\"ahler structure.  Any crossingless matching $\wp$ of $2k$ points  defines a Lagrangian submanifold $(S^2)^k \cong L_{\wp} \subset \scrY_k$, depending up to Hamiltonian isotopy only on the isotopy class of the matching,  by an iterated vanishing cycle construction recalled in Section \ref{Sec:Slices}.  Considering the finitely many isotopy types of crossingless matchings contained in the upper half-plane, one obtains a distinguished finite collection $\bigcup_{\wp} L_{\wp}$ of Lagrangian submanifolds of $\scrY_k$.  

Each $L_{\wp}$ admits a $Spin$ structure and  grading in the sense of \cite{Seidel:graded}, hence defines an object in the Fukaya category of $\scrY_k$ (whose objects are compact exact Lagrangian submanifolds equipped with suitable brane data).   Let $\scrF(\scrY_k)$ denote the subcategory of the Fukaya category  of $\scrY_k$ with objects $\{L_{\wp}\}$. Essentially by definition, this is quasi-equivalent to the $A_{\infty}$-algebra $\oplus_{\wp,\wp'} HF^*(L_\wp, L_{\wp'})$, which we call the \emph{symplectic arc algebra}.

\begin{Theorem} \label{thm:Formal}
Fix a coefficient field $\bk$ of characteristic zero.  The $A_{\infty}$-category $\scrF(\scrY_k)$ defined over $\bk$ is formal, i.e. it is quasi-equivalent to its cohomological category equipped with the $A_{\infty}$-structure in which all $\{\mu^j\}_{j \neq 2}$ vanish identically. 
\end{Theorem}

The categories $\scrF(\scrY_k)$ for different $k$ are related by various canonical bimodules $\cup_i$ and $\cap_i$, $1\leq i\leq 2k-1$, defined by symplectic analogues of the cup and cap bimodules of \cite{Khovanov:functor}.  More precisely, there are  Lagrangian correspondences $\Gamma_i \subset \scrY_k \times \scrY_{k-1}$, the graphs of co-isotropic vanishing cycles, which define bimodules via the quilt formalism of Mau, Wehrheim and Woodward \cite{WW,MWW}. Such correspondences play an implicit role in the construction of the link invariant $\Kh_{symp}(\kappa)$ in \cite{SS}, and were further considered in the work of Rezazadegan  \cite{Reza}.  The sequel to this paper \cite{AbSm:Kh=Floer} proves that the bimodules $\cup_i$ and $\cap_i$ are themselves formal over any $\bk$ of characteristic zero, and proves that the symplectic arc algebra is isomorphic over $\bZ$ to Khovanov's diagrammatically defined arc algebra \cite{Khovanov} (the latter isomorphism over $\bZ_2$ is the main result of \cite{Reza2}).  These results together imply that symplectic and combinatorial Khovanov cohomologies have the same total rank over characteristic zero fields. 

Theorem \ref{thm:Formal} is inferred from a general formality criterion for $A_{\infty}$-categories, Theorem \ref{thm:Pure},  which we learned from Paul Seidel. The input for that criterion is a degree one Hochschild cohomology class $b \in HH^1(\scrF(\scrY_k), \scrF(\scrY_k))$ satisfying a certain purity condition; this class plays the role in our story of a ``dilation" in symplectic cohomology, as studied by Seidel and Solomon in \cite{Seidel-Solomon}.  The construction of such a  Hochschild class, which arises by counting holomorphic curves in a partial compactification of $\scrY_k$ satisfying a conormal-type condition at infinity, is first carried out in some generality, see Section \ref{Sec:Generalities}, although the crucial purity condition seems hard to establish without appeal to specific geometric features of our situation.

\noindent \textbf{Outline of the paper.}  Section \ref{Sec:Formality} contains  algebraic background and the abstract formality results.  Section \ref{Sec:Generalities} explains a general mechanism for building degree one Hochschild cocycles from partial compactifications, and Section \ref{Sec:Milnor} applies this machinery in the model case of the Milnor fibre of the $A_n$-singularity. Sections \ref{Sec:Slice} recalls the spaces $\scrY_k$ and their embeddings into Hilbert schemes, introduces the Lagrangians $L_\wp$, and makes a preliminary investigation of the Floer product in the symplectic arc algebra. Section \ref{Sec:HilbScheme} studies the holomorphic curve theory of the Hilbert scheme of the Milnor fibre, constructs the required Hochschild cocycle, and establishes Theorem \ref{thm:Formal}. 

\subsection*{Conventions}
 When discussing categorical constructions (for instance the Fukaya category, its Hochschild cohomology, etc) we work over a coefficient field $\bk$. At certain points, it will be essential to specialise to the case in which $\bk$ has characteristic zero, but for clarity we impose that hypothesis only when required.

\Acknowledgements We are grateful to Paul Seidel for helpful suggestions concerning this project, over many years and related to numerous different strategies.   Conversations with Sabin Cautis, Ciprian Manolescu, and Catharina Stroppel were also influential. We would finally like to thank the referee for their useful comments.   


\section{Formality results} \label{Sec:Formality}

Let $\scrA$ be a $\bZ$-graded $A_{\infty}$-algebra over $\bk$, equipped with $A_{\infty}$ products
\begin{equation}
  \mu_{\scrA}^{d} \co \scrA^{\otimes d} \to \scrA, \, \, 1 \leq d
\end{equation}
of degree $2-d$. The first two operations satisfy the Leibniz equation\footnote{Our sign conventions follow those of Seidel in \cite{FCPLT}: elements of $\scrA$ are equipped with the reduced degree $||a|| = |a| -1$, and operators act on the right.}:
\begin{equation} \label{eq:Leibnitz}
  \mu_{\scrA}^{1}(  \mu_{\scrA}^{2}(a_2,a_1)) +   \mu_{\scrA}^{2}( a_2,\mu_{\scrA}^{1}( a_1)) + (-1)^{|a_1|-1} \mu_{\scrA}^{2}( \mu_{\scrA}^{1}(a_2), a_1)  = 0.
\end{equation}

The cohomology groups with respect to $\mu_{\scrA}^{1}$, denoted $A = H(\scrA)$, naturally form an $A_{\infty}$-algebra for which all operations vanish except the product, which is induced by $\mu_{\scrA}^{2} $. 
\begin{defn}
$\scrA$ is formal if it is quasi-isomorphic to $A$.
\end{defn}

This section formulates and proves a necessary and sufficient condition for the formality of an $A_{\infty}$-algebra, due to Paul Seidel, in terms of the existence of a particular kind of degree one Hochschild cohomology class.

\begin{rem}
The quintessential result about formality is \cite{DGMS} which proves that the de Rham cochains of a K\"ahler manifold are formal as a \emph{commutative dg-algebra}. While we use a more abstract language, the notion of bigrading plays an essential role both in the formality criterion that we prove and in the classical result of \cite{DGMS}.
\end{rem}

\subsection{Formality for algebras}\label{Sec:FormalityAlgebras}

Recall the Hochschild cochain complex has chain groups
\begin{equation} \label{eq:Hoch_cochain_complex}
CC^d(\scrA,\scrA) =  \prod_{s \geq 0} \Hom_{d}( \scrA[1]^{\otimes s}, \scrA), 
\end{equation}
where $\scrA[1]$ is the graded vector space obtained by shifting the degree of all elements of $\scrA$ \emph{down} by $1$, i.e. equipping them with the reduced degree, and $\Hom_{d}$ is the space of $\bk$-linear maps of degree $d$. 

There is a convolution operation
\begin{equation}
(\sigma \circ \tau)^d(a_d,\ldots,a_1) = \sum_{i,j} (-1)^{(|\tau|-1)\dagger_i} \sigma^{d-j+1}(a_d,\ldots, \tau^j(a_{i+j},\ldots, a_{i+1}), a_i,\ldots, a_1)
\end{equation}
where we use the standing notation
\begin{equation}
  \dagger_i = \sum_{k=1}^i (|a_k|-1).
\end{equation}
 The $A_{\infty}$-structure operations $\mu_{\scrA} = \{\mu_{\scrA}^k\}_{k \geq 1}$  define an element
\begin{equation}
\mu_{\scrA} \in CC^2(\scrA,\scrA),
\end{equation}
and the $A_{\infty}$ equation which specialises to Equation \eqref{eq:Leibnitz} when all higher products $\{\mu_{\scrA}^k\}_{k \geq 3} $  vanish is
\begin{equation}
  \mu_{\scrA} \circ \mu_{\scrA} = 0.
\end{equation}
The \emph{Hochschild cohomology} of $\scrA$, denoted $HH^{*}(\scrA, \scrA)$ is the cohomology of the Hochschild cochain complex \eqref{eq:Hoch_cochain_complex} with respect to the differential:

\begin{align} \label{eq:CC_differential}
\delta: CC^{d-1}(\scrA,\scrA) & \rightarrow CC^{d}(\scrA,\scrA) \\
(\delta \sigma)^d(a_d,\dots,a_1) & = \sum_{i,j} (-1)^{(|\sigma|-1)\dagger_{i}}
\mu^{d-j+1}_\scrA(a_d,\dots,\sigma^j(a_{i+j},\dots,a_{i+1}),\dots,a_1) \\ \notag
& + \sum_{i,j} (-1)^{|\sigma|+ \dagger_{i}} \sigma^{d-j+1}(a_d,\dots,\mu_\scrA^j(a_{i+j},\dots,a_{i+1}),\dots,a_1). 
\end{align}

Specialising further to the case when $\scrA$ is a graded algebra, i.e. when all operations vanish except for $\mu_{\scrA}^{2}$, and denoting this product by concatenation, we obtain, up to a change in sign conventions,  the usual definition of the Hochschild differential for graded algebras:
\begin{equation}\label{Eqn:DiffInAlgCase}
\begin{aligned}
 (\delta\phi)^d(a_d,\dots,a_1) &  =  a_d \phi^{d-1}(a_{d-1},\dots,a_1) + (-1)^{(|\phi|-1)\dagger_{1}} \phi^{d-1}(a_j,\dots,a_2)a_1 \\
& \quad + \sum_i (-1)^{|\phi| + \dagger_{i}} \phi^{d-1}(a_d,\dots,a_{i+2}a_{i+1},\dots,a_1).
\end{aligned}
\end{equation}

\begin{Definition} An \nc-vector field  is a cocycle $b\in CC^1(\scrA,\scrA)$.
\end{Definition}
\begin{rem} \label{Rem:justify-nc}
 In the definition, \nc \   stands for non-commutative. The terminology is motivated by the following example: if $\scrA$ is an $A_{\infty}$-refinement of the category of coherent sheaves on a smooth algebraic variety, then algebraic vector fields on the underlying space (i.e. sections of the tangent bundle) give rise to elements of $ HH^1(\scrA,\scrA) $. 
\end{rem}

On a graded algebra, we have a canonical \nc-vector field called the \emph{Euler vector field}, which multiplies the graded piece $A^i \subset A$ of $A$ by $i$:
\begin{equation}
\label{Eqn:EulerField}
e: A^{i} \rightarrow A^{i}, \quad a \mapsto i \cdot a.
\end{equation}
The fact that multiplication preserves the grading
\begin{equation}
|a_2 a_1| = |a_2| + |a_1|
\end{equation}
implies via Equation \ref{Eqn:DiffInAlgCase} that $e \in CC^1(A,A)$ is a cocycle, hence defines a class in $HH^1(A,A)$ (which has no constant or higher order\footnote{We use ``order" to refer to the arity, i.e. number of inputs, to a multilinear map which is part of a Hochschild cochain.} terms).  

We shall presently see that the presence of an \nc-vector field that induces the Euler vector field on cohomology characterises formal algebras. To state the result precisely, note that there is a natural projection of cochain complexes
\begin{align} \label{eq:image_b_0}
 CC^*(\scrA,\scrA) & \rightarrow \scrA \\
b & \mapsto b^{0}
\end{align}
induced by taking the order-$0$ part of a Hochschild cochain.  Given an element of the kernel of this map, the first order part 
\begin{equation}
  b^{1} \co \scrA \to \scrA
\end{equation}
is a chain map, and hence defines an endomorphism of $A$.

\begin{Definition} \label{def:purity_algebra}
An \nc-vector field $b \in CC^1(\scrA,\scrA)$ is \emph{pure} if $b^0 = 0$, and the induced endomorphism of $A$ agrees with the Euler vector field. \end{Definition}

 If $\scrA$ admits a pure vector field, in a minor abuse of notation we say that $\scrA$ itself is pure.  We learned the following from Paul Seidel; this is the key result which requires that the field have characteristic zero.

\begin{Theorem}[Seidel] \label{thm:Pure} 
Suppose $\bk$ has characteristic zero. 
An $A_{\infty}$-algebra $\scrA$ over $\bk$ is pure if and only if it is formal.
\end{Theorem}
One direction holds trivially, since the Euler vector field itself defines a pure vector field on an ordinary algebra. To show that purity implies formality, we begin by noting that both properties are invariant under quasi-isomorphisms. In particular, it suffices to prove the result for a minimal $A_{\infty}$-algebra, i.e. one for which $\mu^{1}_{\scrA} $ vanishes. In this case, we have an isomorphism $\scrA \simeq A$ of graded vector spaces, but we view $\scrA$ as carrying its full $A_{\infty}$-structure and $A$ as carrying only the multiplication $\mu^2_{\scrA}$.

It is useful at this stage to recall that a formal diffeomorphism is a (collection of) map(s)
  \begin{equation}
\Phi = \{\Phi^d\} \co  \bigoplus_{1 \leq d} A^{d} \to A
  \end{equation}
  which is arbitrary subject to the constraint that $\Phi^1: A \to A$ be an isomorphism. 
As discussed in \cite[Section 1c]{FCPLT}, there is a unique $A_{\infty}$-structure $\scrA_{\Phi}$ on the vector space $A$ such that $\Phi$ defines an $A_{\infty}$-homomorphism from $\scrA$ to $\scrA_{\Phi} $. The higher products which comprise $\scrA_{\Phi}$  are obtained by recursively solving the $A_{\infty}$-equation for a functor.  For a minimal $A_{\infty}$ algebra, formality is equivalent to the existence of a formal diffeomorphism, whose linear term is the identity, such that all higher products on $\scrA_{\Phi}$ vanish.

The construction of the required formal diffeomorphism will be done by induction on the order of vanishing of the higher products on $\scrA$.  To this end, we introduce the notion of a minimal algebra which is formal to order $k$, i.e. such that
\begin{equation}
  \mu^d_{\scrA} = 0 \textrm{ for } 2 <  d \leq k,
\end{equation}
and a pure \nc-vector field $b \in CC^1(\scrA, \scrA)$ which is linear to order $k -1$, meaning that
\begin{equation}
b^i = 0  \textrm{ for } 1 < i \leq k-1.
\end{equation}
\begin{Lemma} \label{lem:formality_induction_step_algebra} 
Suppose $\mathrm{char}(\bk)=0$. 
If $(\scrA,b)$ is a pair consisting of a minimal algebra and a pure \nc-vector field which are respectively  formal to order $k$ and linear to order $k-1$,  there is a formal diffeomorphism $\Phi$ which agrees with the identity to order $k-1$ such that  $\scrA_{\Phi} $  is formal to order $k+1$.

\end{Lemma}
\begin{proof} We define $\Phi$ by the formula
\begin{equation}\label{eq:construct_formal_diff}
\Phi^1 = \Id; \quad \Phi^k = \frac{b^k}{1-k}; \quad\Phi^d = 0 \ \textrm{for} \ d \neq 1, k.
\end{equation}
We shall now prove that this choice ensures that $\scrA_{\Phi} $ is formal to order $k+1$.  Consider the equation $\delta b = 0 \in CC^2(\scrA, \scrA)$, and recall that the cocycle $b$ has graded degree $|b|=1$.  The condition that $(\delta b)^{k+1}(a_{k+1},\ldots, a_1) = 0$ reads as follows (for simplicity we write multiplication by $\cdot$): 
\begin{multline}\label{db=0}
a_{k+1} \cdot b^k(a_k, \ldots, a_1) + b^k(a_{k+1},\ldots,a_2)\cdot a_1 + \sum_i \mu_{\scrA}^{k+1} (a_{k+1},\ldots, b^1(a_i),a_{i-1},\ldots,a_1)  \\
-\sum (-1)^{\dagger_i} b^k(a_{k+1},\ldots, a_{i+2}\cdot a_{i+1},\ldots,a_1) - b^1(\mu^{k+1}_{\scrA}(a_{k+1},\ldots,a_1)) = 0.
\end{multline}
Moreover, we know that $b^1$ co-incides with the Euler vector field, and 
\begin{equation}
|\mu^{k+1}_{\scrA}(a_{k+1},\ldots,a_1)| \ = \  1- k +\sum |a_i| 
\end{equation}  
Therefore, the total coefficient of $\mu^{k+1}_{\scrA}(a_{k+1},\ldots,a_1)$ in Equation \eqref{db=0} is equal to $k-1$, and we conclude that
\begin{multline} \label{eq:vanishing_expression_db_k}
  a_{k+1} \cdot b^k(a_k, \ldots, a_1) + b^k(a_{k+1},\ldots,a_2)\cdot a_1  \\
-\sum (-1)^{\dagger_i} b^k(a_{k+1},\ldots, a_{i+2}\cdot a_{i+1}, \ldots,a_1) + (k-1) \mu_{\scrA}^{k+1} (a_{k+1},\ldots, a_i,a_{i-1},\ldots,a_1)   = 0.
\end{multline}

We shall use this relation to prove that the $(k+1)$\st higher product  on $\scrA_{\Phi} $ vanishes. The lowest order non-trivial condition that $\Phi$ defines an $A_{\infty}$-functor $\scrA \rightarrow \scrA_{\Phi}$ gives the equation
\begin{equation}\label{functorconstraint}
\begin{aligned}
a_{k+1} \cdot \Phi^k(a_k, \ldots, a_1) + \Phi^k(a_{k+1},\ldots, a_2)\cdot a_1 + \mu_{\scrA_{\Phi}}^{k+1}(a_{k+1},\ldots, a_1)\  \\
 -  \sum (-1)^{\dagger_i} \Phi^k (a_{k+1},\ldots, a_{i+2}\cdot a_{i+1}, \ldots, a_1) - \mu^{k+1}_{\scrA}(a_{k+1},\ldots,a_1) =0 .
\end{aligned}
\end{equation}
We  have used the fact that the multiplication is the same $\cdot$ in both $\scrA$ and $\scrA_{\Phi}$ and that $\Phi^1 = \id$.  In particular,  $\mu^{k+1}_{\scrA_{\Phi}}(a_{k+1},\ldots,a_1)$ vanishes if and only if
\begin{equation}
\begin{aligned}
a_{k+1} \cdot \Phi^k(a_k, \ldots, a_1) + \Phi^k(a_{k+1},\ldots, a_2)\cdot a_1  \\
 -  \sum (-1)^{\dagger_i} \Phi^k (a_{k+1},\ldots, a_{i+1}\cdot a_i, \ldots, a_1) - \mu^{k+1}_{\scrA}(a_{k+1},\ldots,a_1) =0 .
\end{aligned}
\end{equation}
The choice $\Phi^k = b^k/(1-k)$ exactly guarantees that this is true, as a consequence of Equation \eqref{eq:vanishing_expression_db_k}.

Finally, note that the choice of $\Phi^i$ for $i\leq k$ in \eqref{eq:construct_formal_diff}  ensures that the $A_{\infty}$-structure $\scrA_{\Phi}$ agrees with the given structure $\scrA$ up to order $k$.  We deduce that the products $\mu^j_{\scrA_{\Phi}}$ vanish for $j< k+1$ as well as for $j=k+1$, hence $\scrA_{\Phi}$ is indeed formal to order $k+1$. 
\end{proof}

The next step in the induction procedure is to show that $\scrA_{\Phi} $ is naturally equipped with an \nc-vector field which is linear to order $k$. To do this, we use the fact that an $A_{\infty}$-quasi-isomorphism is always invertible \cite[Corollary 1.14]{FCPLT}.  Applying this to $\Phi: \scrA \rightarrow \scrA_{\Phi}$ gives a functor
\begin{equation}
\Psi: \scrA_{\Phi} \rightarrow \scrA
\end{equation} 
which for general reasons has the feature that
\begin{equation}
\Psi^1 = \Id; \quad \Psi^i = 0 \ \textrm{for} \ 2 \leq i < k; \quad \Psi^k = - \Phi^k.
\end{equation}
There are now maps
\begin{equation}
CC^*(\scrA,\scrA) \stackrel{\Phi}{\longrightarrow} CC^*(\scrA,\scrA_{\Phi}) \stackrel{\Psi}{\longrightarrow} CC^*(\scrA_{\Phi},\scrA_{\Phi}),
\end{equation}
where the middle term is the Hochschild complex of $\scrA$ with coefficients in the bimodule $\scrA_{\Phi}$ induced by the functor $\Phi$ (see \cite[Section 2.9]{Ganatra} for background on the Hochschild complex with coefficients in a bimodule).  Explicitly, the maps are given by
\begin{align*}
(\Phi(\phi))^k(a_k,\ldots,a_1) & = \sum_{i,b} (-1)^{(|\phi|-1) \cdot \dagger_i} \Phi^{k-i}(a_k,\ldots, a_{k+b+1},\phi^i(a_{b+i},\ldots, a_b),a_{b-1},\ldots, a_1) \\  
(\Psi(\psi))^k(a_k,\ldots,a_1) & = \sum_{i_1 + \cdots + i_s = k} \psi^s(\Psi^{i_1}(a_k,\ldots, a_{k-i_1}), \ldots,\Psi^{i_s}(a_{i_s},\ldots,a_1))
\end{align*}

 We define $b_{\Phi} \in CC^1(\scrA_{\Phi},\scrA_{\Phi})$ to be the image of $b$.   The proof that this satisfies the required conditions, i.e. that it agrees with the Euler field to order $k$, is an easy explicit computation left to the reader; the non-trivial fact $b_{\Phi}^k = 0$ follows from $\Phi^k = - \Psi^k$.
\begin{Lemma} \label{lem:new_b_linear_order_k+1}
  If $\scrA $, $b$, and $\scrA_{\Phi} $ are as in Lemma \ref{lem:formality_induction_step_algebra}, $b_{\Phi}$ is a pure \nc-vector field which is linear to order $k$. \qed
\end{Lemma}

To apply these results, we recall the definition of a composition of formal diffeomorphisms
\begin{equation} \label{Eqn:Compose-formal_differo}
(\Phi \circ \Psi)^d(a_d,\ldots, a_1) = \sum_{i,\,  j_1+\cdots+j_i = d} \Phi^{i}(\Psi^{j_1}(a_d,\ldots,a_{d-j_1+1}), \ldots, \Psi^{j_i}(a_{j_i}, \ldots, a_1)).
\end{equation}

We can now give the proof of Seidel's formality criterion.

\begin{proof}[Proof of Theorem \ref{thm:Pure}]
By assumption, we are given an algebra which is formal to order $2$, and an \nc-vector field which agrees with the Euler vector field to order $1$.  This corresponds to the base case $k=2$ of our inductive procedure.  Lemmata \ref{lem:formality_induction_step_algebra} and \ref{lem:new_b_linear_order_k+1} provide us with sequences $(\scrA_k, b_k)$ of algebras formal to order $k$ and vector fields linear to order $k-1$, together with formal diffeomorphisms $\Phi_{k+1}$ on $A$ such that $\scrA_{k+1}$ is obtained by applying $\Phi_{k+1}$ to $\scrA_{k}$. 

We shall define a formal diffeomorphism as an infinite composition of diffeomorphisms $\Phi_{k}$. To see that this is well defined, consider $(a_k,\ldots,a_1) \in A^{\otimes k}$. Since  $\Phi_{k}$ agrees with the identity to order $k$, we find that
\begin{equation} \label{eq:check_composition_defined}
   \left(\Phi_j \circ \cdots \circ \Phi_3 \right)^{k}(a_k,\ldots,a_1) =  \left(\Phi_k \circ \cdots \circ \Phi_3 \right)^{k}(a_k,\ldots,a_1)
\end{equation}
whenever $j \geq k$. The infinite composition
\begin{equation}
\Phi = \cdots \Phi_{k+1} \circ \Phi_k \circ \Phi_{k-1} \cdots \circ \Phi_3
\end{equation}
is therefore well defined. Moreover, Equation \eqref{eq:check_composition_defined} implies that the higher products on $\scrA_{\Phi}$ agree up to order $k$ with the higher products on $\scrA_{k}$. Since $k$ is arbitrary, we conclude that all higher products on  $\scrA_{\Phi}$ vanish, hence that $\scrA$ is formal.
\end{proof}

\begin{rem}
 Our proof in fact provides us with the following slightly sharper statement: if $b$ defines a pure structure on an $A_{\infty}$-algebra $\scrA$ over a characteristic zero field, then there is an $A_{\infty}$-equivalence $\scrA \to A$ which maps $b$ to the Euler vector field.
\end{rem}

\subsection{Formality for categories}

Theorem \ref{thm:Pure} has an obvious generalisation to $A_{\infty}$-categories rather than $A_{\infty}$-algebras.  First, the Hochschild complex of an $A_{\infty}$-category is defined exactly analogously to the case for algebras, using chains of composable morphisms.   Thus, $HH^*(\scrA,\scrA)$ is computed by a chain complex $CC^*(\scrA,\scrA)$ for which a degree $r$ cochain is a sequence $(h^d)_{d\geq 0}$ of collections of linear maps
\[
h^d_{(X_1,\ldots,X_{d+1})}: \bigotimes_{i=d}^1 hom_{\scrA}(X_i,X_{i+1}) \rightarrow hom_{\scrA}(X_1,X_{d+1})[r-d]
\]
for each $(X_1,\ldots,X_{d+1})\in \Ob(\scrA)^{d+1}$.  The differential is the obvious analogue of \eqref{eq:CC_differential}, but where the inputs are now composable sequences of morphisms in $\scrA$.

Suppose then $\scrA$ is an $A_{\infty}$-category and fix an \nc-vector field $b \in CC^1(\scrA,\scrA)$ on $\scrA$. The constant term $b^0$ of $b$ defines a cocycle $b^0|_L \in \hom_{\scrA}^1(L,L)$ for every object $L\in \Ob \scrA$. 

The straightforward generalisation of the assumption that $\Phi(b)=0$  in Definition \ref{def:purity_algebra} is to require that $b^0|_L  $ vanish. It is convenient to consider a slightly more general setup:

\begin{Definition} An (infinitesimally) equivariant object is a pair $(L,c)$, with  $L \in \Ob \scrA$ and $c \in \hom_{\scrA}^0(L,L)$, with $dc = b^0|_L$. 
\end{Definition}
\begin{rem}
The intuition, compare to Remark \ref{Rem:justify-nc}, is that the vector field $b$ integrates to a flow, and the condition we have written corresponds to being a fixed point. Since we shall not consider any other notion of equivariant object in this paper, we shall often elide ``infinitesimally'' from our terminology.
\end{rem}
There is a natural notion of equivalence for equivariant objects, in which two choices of $c$ which differ by a degree zero cocycle are regarded as equivalent.  For a given $L$, the obstruction to the existence of any suitable $c$ is given by $[b^0|_L] \in H^1(\hom_{\scrA}(L,L))$, and the set of choices when this vanishes forms an affine space over $H^0(\hom_{\scrA}(L,L))$. 

 Given two infinitesimally equivariant objects $(L,c_L)$ and $(L', c_{L'})$, there is a distinguished endomorphism of $H^*(\hom_{\scrA}(L,L'))$  induced by the linear part $b^1$ of $b \in CC^1(\scrA,\scrA)$.  Because we have not assumed that $b^0$ vanishes, $b^1$ is not necessarily a chain map. However, the endomorphism of $ \hom_{\scrA}(L,L') $ defined by the equation
\begin{equation} \label{Eqn:Endomorphism}
\Phi \co \phi \mapsto b^1(\phi) - \mu^2(c_L, \phi) + \mu^2(\phi,  c_{L'})
\end{equation}
 is a chain map, and descends to cohomology (preserving the cohomological degree). In particular, one can then decompose $H^*(\hom_{\scrA}(L,L'))$ into the generalised eigenspaces of \eqref{Eqn:Endomorphism}, which gives an additional ``grading" of this group, which we shall call the \emph{weight} and denote by $\wt$. That grading depends only on the equivalence class of the equivariant structures on $L$ and $L'$.  \emph{A priori}, the weight grading is indexed by elements of the algebraic closure $\bar{\bk}$ of the coefficient field.  

 For later use, we record some  general properties of these weight gradings. Consider equivariant objects  $(L,c_L)$, $(L',c_{L'})$ for a vector field $b \in CC^1(\scrA,\scrA)$.

\begin{Lemma} \label{lem:WeightGradings} Suppose that $H^0(\hom_{\scrA}(L,L))$ and $H^0(\hom_{\scrA}(L',L'))$ both have rank one, and fix the unique $\bk$-linear identifications of these groups with the ground field that map the unit in $ H^0(\hom_{\scrA}(L,L))$ resp.  $H^0(\hom_{\scrA}(L',L'))$ to $1 \in \bk$. Then:
\begin{enumerate}
\item The endomorphism \eqref{Eqn:Endomorphism} is a derivation. 
\item A change in equivariant structures changes the weights by a shift:
\begin{equation} \label{Eqn:ShiftWeights}
\Phi_{(L,c_L), (L', c_{L'})} \ = \  \Phi_{(L,c_L+s), (L', c_{L'}+s')}  + (s-s')\id 
\end{equation}
for any $s, s' \in \bk$. 
\item The weights on $H^*(\hom_{\scrA}(L,L))$ are independent of the choice of equivariant structure $c_L$ on $L$.
\end{enumerate}
\end{Lemma}

\begin{proof} The first statement follows from the cocycle condition for $b$. The second statement follows from the definition \eqref{Eqn:Endomorphism}, and in turn implies the third statement.
\end{proof}

Since $H^0(\hom_{\scrA}(L,L)) \cong \bk$ has rank one, it is generated by the identity endomorphism of $L$. The fact that \eqref{Eqn:Endomorphism} is a derivation thus implies that it acts by zero on $H^0(\hom_{\scrA}(L,L))$, and is furthermore compatible with product structures, meaning that if $\alpha$ and $\beta$ are of pure weight (live in single generalised eigenspaces), then
\begin{equation}\label{Eqn:WeightFloerProduct}
  \wt(\mu^2_{\scrA}(\alpha,\beta)) = \wt(\alpha) + \wt(\beta).
\end{equation}
See \cite{Seidel-Solomon}, Remark 4.4, Equation (4.9) and Corollary 4.6 for the corresponding statements for dilations in symplectic cohomology.

 We say the category admits a pure vector field (``is pure") if there is some $b \in CC^1(\scrA,\scrA)$, and lifts of all objects to infinitesimally equivariant objects, in such a way that  the above endomorphism $\Phi$  agrees with the Euler vector field for every pair of objects.  Generalising the case of algebras, we obtain the following result:

\begin{Corollary} \label{cor:Pure2} 
Suppose $\bk$ has characteristic zero. If $\scrA$ is pure, then $\scrA$ is formal.
\end{Corollary} 
\begin{proof}
Note that the assumptions are invariant under quasi-isomorphisms, so it suffices to prove the result in the case when $\scrA$ is minimal and strictly unital, since every $A_{\infty}$-category is quasi-isomorphic to one which is minimal and strictly unital by \cite[Lemma 2.1]{FCPLT}. In this case, $b^0$ vanishes, and we may define a new vector field
\begin{align}
\tilde{b}^1| \hom_{\scrA}(L,L') & = b^1 - \mu^2(c_L, \_ ) + \mu^2(\_,  c_{L'}) \\
 \tilde{b}^d & = b^d \textrm{ if } d \neq 1.
\end{align}
Since $\scrA$ is minimal, and both $H^0(\hom_{\scrA}(L,L))$ and $H^0(\hom_{\scrA}(L',L'))$ have rank one by assumption, $c_{L}$ and $c_{L'}$ are each multiples of the corresponding units. The assumption that $\scrA$ is strictly unital implies that $\tilde{b}$ is a cocycle.  By construction, $\tilde{b}^1$ agrees with the Euler vector field on all morphism spaces; the reader may now easily repeat the argument we gave for algebras to prove the formality of $\scrA$.
\end{proof}

\newcommand{\Mbar}{\bar{M}}
\newcommand{\Mdbar}{\bar{\bar{M}}}
\newcommand{\Mod}[2]{\scrR^{#2}_{#1}}
\newcommand{\Modbar}[2]{\bar{\scrR}^{#2}_{#1}}
\newcommand{\Fuk}{\scrF}
\newcommand{\Chord}{{\EuScript X}}
\newcommand{\ro}{{\mathrm o}}
\newcommand{\CO}{{\mathcal CO}}
\newcommand{\sCO}{co}
\newcommand{\Lagr}{{\EuScript L}}
\newcommand{\ev}{\mathrm{ev}}
\section{Geometry generalities} \label{Sec:Generalities}

This section abstracts the particular features of the geometric situation encountered later which enable us to define an \nc-vector field on an exact Fukaya category $\scrF(M)$ via counting discs in a partial compactification.  This geometric set-up is by no means the most general possible.  In Section \ref{Sec:Milnor} we shall apply this construction when $M$ is the complex two-dimensional $A_{2k-1}$-Milnor fibre, whilst the case  $M = \scrY_k$, viewed as an open subset of $\Hilb^{[k]}(A_{2k-1})$, is covered in Section \ref{Sec:HilbScheme}.  At the start of each of Sections \ref{Sec:Milnor} and \ref{Sec:HilbScheme} we present a short dictionary for comparison with the notation and hypotheses of this section.

\subsection{Set-up}

We begin with a smooth projective variety $\Mdbar$ of complex dimension $n$, equipped with a triple of reduced (not necessarily smooth or irreducible) effective divisors $D_0$, $D_{\infty}$, $D_r$. We denote by $\Mbar$ the symplectic manifold obtained by removing $D_{\infty}$ from $\Mdbar$, and by $M$ the symplectic manifold obtained by removing the three divisors from $\Mdbar$.  When the meaning is clear from context, we shall sometimes write $D_{0}$ for $D_{0} \cap \Mbar$ and $D_{r}$ for  $D_{r} \cap \Mbar$. We assume:

\begin{Hypothesis} \label{Hyp:Main} 
  \begin{align} \label{eq:hyp_main_1}
& \parbox{30em}{the union  $D_0 \cup D_{\infty} \cup D_r$ supports an  ample divisor $D$ with strictly positive coefficients of each of $D_0, D_{\infty}, D_r$.} \\
& \parbox{30em}{$D_{\infty}$ is nef (or, at least, non-negative on rational curves).}\\ \label{eq:hyp_main_3}
&  \parbox{30em}{$\Mbar$ admits a meromorphic volume form $\eta$ which is non-vanishing in $M$, holomorphic along $D_r \cap \Mbar$, and with simple poles along $D_0 \cap \Mbar$.}  \\
& \parbox{30em}{Each irreducible component of the divisor $D_0 \cap \Mbar$ moves in $\Mbar$, with base locus containing no rational curves.} 
\end{align}
\end{Hypothesis}

Let $D'_0 \subset \Mbar$ be a divisor linearly equivalent to and sharing no irreducible component with $D_0$, and $B_0 = D_0 \cap D'_0$, which is then a subvariety of  $\Mbar$ of complex codimension $2$.

Fix a K\"ahler form $\omega_{\Mdbar}$ in the cohomology class Poincar\'e dual to $D$. Ampleness implies that $M$ is an affine variety, in particular an exact symplectic manifold which can be completed to a Stein manifold of finite type, modelled on the symplectization of a contact manifold near infinity.  We will denote by $\lambda$ a primitive of the symplectic form $\omega_M$ given by restricting $\omega_{\Mdbar}$ to $M$, so $d\lambda = \omega_M$.     By the third assumption above, $M$ has vanishing first Chern class.

The assumption that each irreducible component of $D_0$ moves in $\Mbar$ is not essential, but simplifies some of the arguments, cf. Remark \ref{Rem:Conormal} for an indication of how to proceed otherwise.  

We shall write $J$ for the natural complex structure on $\Mdbar$, $\Mbar$, and $M$.  

\begin{Lemma} \label{lem:chern_0_positive_D_r}
 Let $C \subset \Mbar$ be the image of a non-constant rational curve $u: \bP^1 \rightarrow \Mbar$. 
 
 \begin{enumerate}
 \item The intersection number  $\langle D_0, C\rangle $ is non-negative, and agrees with the Chern number $ \langle c_1(\Mbar), C \rangle$.
 \item If $C \cap D_0 \neq \emptyset$, then $\langle c_1(\Mbar), C \rangle > 0$ is strictly positive.
 \item If $\langle c_1(\Mbar), C \rangle = 0$, then $C$ intersects  $D_r$ strictly positively.
 \end{enumerate}
\end{Lemma}
\begin{proof}
Since $D_0 \cap \Mbar$ moves,  and the base locus of that linear system contains no rational curves, we can suppose that $C$ is not completely contained in $D_0$. Then $\langle c_1(\Mbar), C\rangle = \langle D_0, C\rangle \geq 0$, with equality only if $C \cap D_0 = \emptyset$. In the latter case, since $D$ is ample  and $C$ is disjoint from $D_{\infty}$, $D_r$ meets $C$ strictly positively, giving the final conclusion.
\end{proof}

\subsection{The Fukaya category} \label{sec:fukaya-category}

Denote by $\Fuk(M)$  the Fukaya category of $M$.  We shall work  within the setting of \cite{FCPLT} to obtain a $\bZ$-graded category over an arbitrary characteristic field $\bk$. To start, fix a (typically finite) collection $\Lagr$ of Lagrangians which we require to be exact, closed, and disjoint from a neighbourhood $\nu D$ of $D$ that contains $D'_0$. We equip the elements of $\Lagr$ with brane data  comprising an orientation, a  $Spin$ structure and a grading with respect to the complex volume form $\eta$. There is a minor difference with \cite{FCPLT}:  we ensure compactness of  moduli spaces of curves in $M$ by using positivity of intersection in $\Mdbar$, rather than the maximum principle  (see Lemma \ref{lem:c-0-estimate-discs} below).

Given a pair  $(L_0, L_1)$ of  Lagrangians, we choose a compactly supported Hamiltonian
\begin{equation}
  H_{L_0,L_1} \co [0,1] \times M \to \bR
\end{equation}
whose time-$1$ Hamiltonian flow maps $L_0$ to a Lagrangian that is transverse to $L_1$. Let $\Chord(L_0,L_1)$ denote the set of intersection points of the time-$1$ image of $L_0$ and $L_1$, and define
\begin{equation}
  CF^*(L_0,L_1) \equiv \bigoplus_{x \in \Chord(L_0,L_1)}  \ro_{x}
\end{equation}
where $\ro_x$ is a $1$-dimensional $\bk$-vector space associated to $x$ by index theory, see \cite[Section 11h]{FCPLT}. The differential in $ CF^*(L_0,L_1)  $ counts rigid Floer trajectories with respect to $J$, and we assume that $H_{L_0,L_1}  $ is chosen generically so that these moduli spaces are regular, see \cite{FHS}.

To define the $A_{\infty}$-structure, for $k\geq 2$ let $\scrR^{k+1}$ denote the moduli space of discs with $k+1$ punctures on the boundary; we fix a distinguished puncture $p_0$, and order the remainder $\{p_1,\ldots, p_k\}$ counter-clockwise along the boundary. As in Section (9g) of \cite{FCPLT}, we choose families of strip-like ends for all punctures, i.e. denoting 
\begin{equation} \label{eqn:in-out-striplike-ends}
  Z_- = (-\infty,0] \times [0,1] \textrm{ and } Z_+ = [0,\infty) \times [0,1]
\end{equation}
we choose, for each surface $\Sigma$ representing a point in $ \scrR^{k+1} $, conformal embeddings of punctured half-strips
\begin{equation}
\epsilon_{0}: Z_- \rightarrow \Sigma, \quad \epsilon_i: Z_+ \rightarrow \Sigma \quad  \textrm{for} \, 1 \leq i \leq k\end{equation}
 which take $\partial Z_{\pm}$ into $\partial \Sigma$, and which converge at the end to the punctures $p_i$. 

Given a sequence  $(L_0, \ldots, L_k)$ of objects we choose inhomogeneous data on the space of maps parametrised by the universal curve over $ \scrR^{k+1} $. Given a curve $\Sigma$ representing an element of $\scrR^{k+1}  $ and a point $z \in \Sigma$, this datum consists of a map 
\begin{equation}
  K \co T_{z} \Sigma \to C^{\infty}_{ct}(M, \bR)
\end{equation} 
subject to the constraint that the pullback of $K$ under $\epsilon_i$ agrees with $H_{L_{i-1},L_{i}} dt $.

 Having fixed these choices, $K$ defines a $1$-form on $\Sigma$ valued in the space of vector fields on $M$, obtained by taking the Hamiltonian vector field associated to a function on $M$:
\begin{equation}
  Y(\xi) = X_{K(\xi)}.
\end{equation}
We obtain a pseudo-holomorphic curve equation:
\begin{equation} \label{eq:Floer_equation_0-puncture}
 J \left( du(\xi) - Y(\xi)  \right) = du(j \xi) - Y(j\xi) ,
\end{equation}
on the space of maps
\begin{equation}
  u \co \Sigma \to M
\end{equation}
with the property that the image of  the segment along the boundary from $p_{i-1}$ to $p_{i}$ lies in $L_i$.  We denote by
\begin{equation}
  \scrR^{k+1}(M | x_0; x_k ,\cdots, x_1)
\end{equation}
the space of such solutions which have finite energy, and converge to $x_i$ along the end $\epsilon_i$.

\begin{Lemma} \label{lem:c-0-estimate-discs}
All elements of $ \scrR^{k+1}( M|x_0; x_k ,\cdots, x_1 )   $ have image contained in a fixed compact subset of $M$.
\end{Lemma}
\begin{proof}
Suppose for contradiction that this is not true. Considering such maps as pseudo-holomorphic discs in $\Mdbar$; Gromov compactness implies the existence of a configuration of discs and rational curves which intersects one of the divisors at infinity, but such that the total configuration has trivial intersection number with $D$. We shall show that this is impossible.

Since all Lagrangians are disjoint from $D$, any disc component in the limit has non-negative intersection number with $D$ since the coefficients of that divisor are positive by Hypothesis \eqref{eq:hyp_main_1}. On the other hand, ampleness implies that the intersection number of any sphere component with $D$ is non-negative, and vanishes only if the sphere is constant. We conclude that there are no sphere components, and that all disc components are disjoint from $D$, yielding the desired contradiction.
\end{proof}

This result implies that the Gromov-Floer construction produces a compactification of $  \scrR^{k+1}(M| x_0; x_k ,\cdots, x_1 )  $ consisting only of stable discs mapping to $M$ (recall that $M$ is exact, so such discs have no sphere components).   Standard regularity results imply that, for generic data, this space is  a smooth manifold of dimension
\begin{equation}
k - 2 +  \deg(x_0) - \sum_{i=1}^k \deg(x_i).
\end{equation}

In particular, whenever the above expression vanishes, the signed count of elements of the moduli space $  \scrR^{k+1}(M|x_0; x_k, \ldots, x_1) $ defines a map:
\begin{equation} \label{eqn:summand}
   \ro_{x_k} \otimes  \cdots \otimes \ro_{x_1} \to \ro_{x_0}.  
\end{equation}
which is canonical up to a choice of orientation of the Stasheff associahedron. Let $\Delta$ denote the unit disc in $\bC$. We follow the conventions of Section (12g) of \cite{FCPLT}, and orient the moduli space of discs by fixing the positions of $p_0$, $p_1$, and $p_2$ on the boundary, and using the corresponding identification of the interior of $  \scrR^{k+1} $ with an open subset of $(\partial \Delta)^{k-2}$, which is naturally oriented.

By definition, \eqref{eqn:summand}, twisted by a sign whose parity is $\sum_{i=1}^{k} i \deg x_i $,  defines the $\ro_{x_0}$-component of the restriction of the $A_{\infty}$-operation $\mu_{\Fuk(M)}^k$ to 
\begin{equation}
   \ro_{x_k} \otimes  \cdots \otimes \ro_{x_1}  \subset CF^*(L_{k-1}, L_k) \otimes \cdots \otimes CF^*(L_{0}, L_1)  \to CF^*(L_{0}, L_k).
\end{equation}

\subsection{A Gromov-Witten invariant}
Let $A \in H_2(\Mbar; \bZ)$ be a $2$-dimensional homology class, with the property that
\begin{equation} \label{eq:intersection_0_1}
  \langle A, D_{r} \rangle = 0 \textrm{ and }   \langle A, D_{0} \rangle = 1. 
\end{equation}
Consider the moduli space of stable rational curves in $\Mbar$ with one marked point
\begin{equation}
\scrM_{1}(\Mbar| 1) = \coprod_{\substack{A \in  H_2(\bar{M};\bZ) \\ \textrm{Condition \eqref{eq:intersection_0_1} holds}}} \scrM_{1; A}(\Mbar)
\end{equation}
which can be decomposed according to  the homology class $A \in H_2(\Mbar;\bZ)$ represented by each element. Recall that this is the Deligne-Mumford partial compactification of the quotient of the space of holomorphic maps $\bP^1 \to \Mbar$  by the subgroup of automorphisms of $\bP^1$ preserving the point $1 \in \bP^1$. (It may be worth emphasising that the Deligne-Mumford space is only a partial compactification since $\Mbar$ itself is not compact.) There is a natural evaluation map
\begin{equation} \label{eq:evaluation_sphere}
\ev_1: \scrM_{1}(\Mbar| 1) \rightarrow \Mbar.
\end{equation}

\begin{Lemma}
Equation \eqref{eq:evaluation_sphere} defines   a proper map.
\end{Lemma}

\begin{proof}
Consider the moduli space of $\scrM_{1}(\Mdbar| 1) $ of rational curves in homology classes $A$ satisfying
\begin{equation}
  \langle A, D_{r} \rangle = 0, \, \,  \langle A, D_{\infty} \rangle = 0, \textrm{ and }   \langle A, D_{0} \rangle = 1.
\end{equation}
These constraints fix the energy of any curve in $\scrM_{1}(\Mdbar| 1) $, so that this moduli space is compact. The evaluation map extends to a map
\begin{equation}
  \bar{\ev}_1: \scrM_{1}(\Mdbar| 1) \rightarrow \Mdbar.
\end{equation}
Given a point $p \in \Mbar$, we claim that $\bar{\ev}_1^{-1}(p)= \ev_{1}^{-1}( p)$. Indeed, the condition that a stable curve $u$ pass through $p$ implies that it cannot lie entirely within $D_{\infty} $. Since this divisor is nef, we conclude that every component of the image of  $u$ is disjoint from $D_{\infty}$. It follows that  $\ev_{1}^{-1}( p )$  is compact.
\end{proof}

Recall that the virtual complex dimension of $\scrM_{1;A}(\Mbar| 1) $ is 
\begin{equation}
  n+c_1(A) - \dim_{\bC}(\Aut(\bP^1,1) ) = n-1,
\end{equation}
where $ \Aut(\bP^1,1) $ is the $2$-dimensional group of M\"obius transformations fixing the point $1\in\bP^1$. To simplify the discussion, we impose the following assumption. A  holomorphic curve $u$ is regular if the linearisation of the Cauchy-Riemann equation defines a surjective operator at the point $u$.  Classical methods achieve regularity away from curves with multiply covered components.  In our setting, any such components of a stable curve in class $A$ have vanishing Chern number. Indeed, we consider classes $A$ with $\langle A, D_0 \rangle = 1$; intersections with $D_0$ are non-negative for each component of a stable curve, by the first part of Lemma \ref{lem:chern_0_positive_D_r}, so a multiply covered component must be disjoint from $D_0$ or it would contribute at least its multiplicity to the intersection number. The first part of Lemma \ref{lem:chern_0_positive_D_r} then  implies that the Chern number of the multiply covered component is zero, and the last part of Lemma \ref{lem:chern_0_positive_D_r} in turn implies that the component meets $D_r$.

\begin{Hypothesis} \label{Hyp:regular}  There is a subvariety $B_{r} \subset \Mbar$ of complex codimension $2$ such that any element of $ev_{1}^{-1}( \Mbar \setminus B_{r})$  is regular and has image disjoint from $D_r$.
\end{Hypothesis}

 By assumption, $\ev_1 \scrM_1(\Mbar|1) $ is an algebraic subscheme of complex codimension $1$, hence admits a well-defined class in the second cohomology of $\Mbar$.  It is technically convenient to use Poincar\'e duality to define this class: the image of the set of non-regular points is a subvariety contained in $B_r$, which by assumption has real codimension $2$ in $\ev_1 \scrM_1(\Mbar|1)$. We can therefore pick a triangulation of $ \ev_1 \scrM_1(\Mbar|1) $  so that the interior of all cells of codimension $0$ and $1$ are contained in the complement of $B_r$. The (weighted) sum of all top-dimensional simplices defines a fundamental class
 \begin{equation}
   [\ev_{1}  \scrM_{1}(\Mbar| 1) ]  \in  C_{2n-2}^{lf}(\Mbar;\bZ),
 \end{equation}
where $C_{*}^{lf}(\Mbar;\bZ) $ are the locally finite (also known as Borel-Moore) chains of $\Mbar$; the elements of this complex are (possibly) infinite linear combinations of simplices satisfying the property that each compact subset intersects the image of only finitely many simplices. This complex is quasi-isomorphic to the cochains of $\Mbar$. Using the isomorphism
\begin{equation}
  H_{2n-2}^{lf}(\Mbar;\bZ) \cong  H^{2}(\Mbar;\bZ),
\end{equation}
we conclude:


\begin{Lemma}
The evaluation image gives a well defined class
\begin{equation}
GW_{1}  \equiv  [\ev_{1}  \scrM_{1}(\Mbar| 1) ] \in  H^{2}(\Mbar; \bZ). 
\end{equation} \qed
\end{Lemma}

\begin{Remark}
Hypothesis \ref{Hyp:regular} is restrictive because we have fixed the almost complex structure $J$ throughout the discussion.  Transversality and smoothness for the space of maps from $\bP^{1}$ to $\Mbar$ follows by standard methods,  see \cite{McD-S}, since the assumption $\langle A, D_0\rangle = 1$ implies that no holomorphic curve in class $A$ can be multiply covered.  However, there may be configurations in $\scrM_1(\Mbar|1)$ which have some multiply covered components, and the definition of the fundamental class in this situation would \emph{a priori} require more sophisticated techniques. Hypothesis \ref{Hyp:regular}  allows us to bypass this problem.
\end{Remark}

\subsection{Infinitesimally equivariant Lagrangians} \label{sec:equiv-struct}

We now add another two hypotheses on the ambient geometry:

\begin{Hypothesis}  \label{Hyp:GWvanishes} 
  \begin{align}
 & GW_{1} = \sum_A GW_{1;A}|_M = 0 \in  H^2(M; \bZ); \\  
& \parbox{30em}{$B_0$ is homologous to a cycle  supported on the union $(D_0 \cap D_{r}) \cup D_0^{sing}$, where $D_0^{sing}$ denotes the singular locus of $D_0$.}
  \end{align}
\end{Hypothesis}
The second hypothesis above can be stated more precisely in terms of locally finite homology: triangulating $B_0$ yields a fundamental class
\begin{equation}
  [B_0] \in  C_{2n-2}^{lf}(D_0 , (D_{0} \cap D_{r})  \cup D_{0}^{sing}  ; \bZ).
\end{equation}
obtained by taking the image of the fundamental class in $C_{2n-2}^{lf}(D_0   ; \bZ)$ under the natural projection map.

Appealing to Hypothesis \ref{Hyp:GWvanishes}, we can therefore fix a cochain
fix cochains 
\begin{align} \label{eq:bounding_chain_infinity}
\beta_0 & \in C_{2n-3}^{lf}(D_0 , (D_{0} \cap D_{r})  \cup D_{0}^{sing}  ; \bZ) \\
\partial (\beta_0) & = [B_0],
\end{align}
as well as a cochain
\begin{align}
\label{eq:bounding_chain_gw}
  gw_{1} & \in C^{lf}_{2n-1}(\Mbar ; \bZ) \\
\partial  (gw_{1}) & = GW_{1}.
\end{align}

Next, we consider the moduli space 
\begin{equation}
  \scrR_{2}^{1}(\Mbar; (1,0)| L  )
\end{equation}
of maps from a disc to $\Mbar$ with $2$ interior marked points denoted $(z_0, z_1)$ and one boundary marked point, with boundary mapping to $L$, and intersection number $1$ with $D_{0}$ and $0$ with $D_r$. There is a unique way to identify the domain with the standard unit disc $\Delta \subset \bC$, in such a way that the first interior marked point maps to $0$, and the boundary marked point maps to $1$. The position of the second marked point, together with the evaluation maps at the two interior points, defines a map
\begin{equation} \label{eq:evaluation_two_marked_points_projection}
    \scrR_{2}^{1}(\Mbar; (1,0)| L  ) \to \Delta \times \Mbar \times \Mbar.
\end{equation}

We define
\begin{equation}
  \Mod{(0,1)}{1}(L) \subset   \scrR_{2}^{1}(\Mbar; (1,0)| L  )
\end{equation}
to be the inverse image of
\begin{equation}
  (0,1) \times D_0 \times D_0'
\end{equation}
under the evaluation map in Equation \eqref{eq:evaluation_two_marked_points_projection}.  A typical representative of this moduli space is depicted in Figure \ref{Fig:BasicDomain}.
\begin{center}
\begin{figure}[ht]
\includegraphics[scale=0.5]{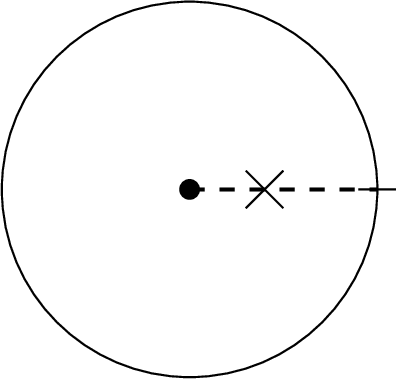}
\caption{A representative of $\Mod{(0,1)}{1}(L)$, with boundary in $L$, $\bullet$ mapping to $D_0$ and $\times$ mapping to $D_0'$.{\label{Fig:BasicDomain}}}
\end{figure}
\end{center}

We shall be interested in describing the boundary of $  \Mod{(0,1)}{1}(L) $. To this end,  we introduce the moduli space 
\begin{equation}
  \scrR_{1}^{1}(\Mbar; (1,0)| L  )
\end{equation}
of discs in $\Mbar$ with boundary on $L$ and one interior and one boundary marked point, with intersection numbers $(1,0)$ with $(D_0,D_r)$; this has a natural map $\scrR_1^1(\Mbar; (1,0)|L ) \rightarrow \Mbar$ via evaluation at the interior marked point.  If the inhomogeneous data for the Floer equation are chosen generically in a relatively compact open subset of $\Mdbar$ with closure disjoint from $D$, then the moduli spaces $\Mod{(0,1)}{1}(L)$, $\scrR_{1}^{1}(\Mbar; (1,0)| L  )$ considered above, and the evaluation maps from these moduli spaces to $\Mbar$, are $C^1$-smooth; for generic Floer data the evaluation maps can moreover be assumed transverse to any fixed finite set of smooth maps representing locally finite cycles in $\Mbar$.   With that understood, we impose the following  additional Hypothesis on the Lagrangian submanifold $L$.
\begin{Hypothesis} \label{hyp:no_Maslov_0_disc_Maslov_1}
  \begin{align}
& \parbox{30em}{$B_{r}$ is disjoint from $L$.} \\
& \parbox{30em}{The moduli spaces of discs $  \Mod{(0,1)}{1}(L)$,    $ \scrR_{1}^{1}(\Mbar; (1,0)| L  ) \times_{\Mbar} B_{0}$, and of spheres $\scrM_{1}(\Mbar| 1) \times_{\Mbar} L$,  are each  transverse fibre products.} \\
& \parbox{30em}{Every disc in $\Mbar$ whose intersection number with $D_{0}$ vanishes and whose boundary lies on $L$ is constant.} 
  \end{align}
\end{Hypothesis}
The first condition makes the transversality of $\scrM_{1}(\Mbar| 1) \times_{\Mbar} L$  unambiguous, but can be weakened to the requirement that the evaluation map from the moduli space of discs on $L$ with one interior marked point is transverse to $B_{r}$. The second set of conditions allows us to avoid using virtual perturbations even whilst fixing the almost complex structure; it could be weakened by choosing domain dependent inhomogeneous perturbations. In order to remove the last condition, we would need to use multivalued perturbations.

\begin{lem} \label{lem:description_boundary_moduli_toy}
The Gromov compactification $\Modbar{(0,1)}{1}(L) $ is a manifold of dimension $n-1$ with boundary strata:
\begin{align} \label{eq:first_stratum_interior_bubble_toy}
&   \scrM_{1}(\Mbar| 1) \times_{\Mbar} L  \\ \label{eq:second_stratum_interior_bubble_toy}
& B_{0}\times_{\Mbar} \scrR_{1}^{1}(\Mbar; (1,0)| L  ).
\end{align}
\end{lem}
\begin{proof}
There are three cases to consider: (i) $z_1 \to 1$ (ii) $z_1 \in (0,1)$, and (iii) $z_1 \to 0$, cf. Figure \ref{Fig:FourTypes}. We shall show that only the last possibility can define non-trivial strata.

{\bf Case (i):}  The domain has two disc components,  each carrying an interior marked point, mapping respectively to (the homologous locally finite cycles) $D_0$ and $D'_0$. Since any other component intersects $D_0$ non-negatively by Lemma \ref{lem:chern_0_positive_D_r}, we conclude that the intersection number with $D_0$ is $\geq 2$, contradicting the original assumption that the intersection number is $1$.  

{\bf Case (ii):} Since $z_1 \neq 0$, the principal disc component of the stable map meets $D_0$, hence has Maslov index 2, which implies that any rational component has Maslov index zero since the total intersection number with $D_0$ is $1$.  
We conclude that all rational curve components of the image meet $D_{r}$ strictly positively by Lemma \ref{lem:chern_0_positive_D_r}.   Since the total intersection number with $D_{r}$ vanishes by assumption, positivity of intersection again implies that all components other than the one containing the interior marked points are discs with image inside $M$. Since $L$ is exact in $M$, there are no such discs, and we conclude that there are no boundary strata corresponding to Case (ii).

Finally, we analyse the more delicate case $z_1 \to 0$. As above, there can be only one component with non-vanishing intersection number with $D_0$.

{\bf Case (iii\,a):} If this component is a disc, then all rational curve components have vanishing Maslov index. Since such curves have positive intersection number with $D_{r}$, and any disc component has non-negative intersection, the fact that the configuration has total intersection number $0$ with $D_{r}$ implies that the only non-constant components other than the one that  meets $D_0$ are discs whose image is contained in $M$. Since $L$ is exact in $M$, there are no such additional disc components, and hence the given  configuration corresponds to a disc attached to a ghost sphere carrying the marked points  $z_1$ and $z_0$. Since these interior marked points are required to map to $D_0$ and $D'_0$, this corresponds to Equation \eqref{eq:second_stratum_interior_bubble_toy}.

{\bf Case (iii\,b):} Assume all discs are disjoint from $D_0$, which by Hypothesis \ref{hyp:no_Maslov_0_disc_Maslov_1}  implies that there is only one disc component on which  the map is moreover constant. This disc is attached to a tree of sphere bubbles,  which represents an element of $\scrM_{1}(\Mbar| 1)$. We conclude that this stratum corresponds to Equation \eqref{eq:first_stratum_interior_bubble_toy}. To show that the moduli space is a manifold along this stratum, we use the first two parts of Hypothesis \ref{hyp:no_Maslov_0_disc_Maslov_1}.
\end{proof}

By evaluation at $1 \in \partial D$, we obtain a cochain
\begin{equation}
   \tilde{b}_{D}^{0} = \ev_{*}[\Modbar{(0,1)}{1}(L)  ]  \in C^{1}(L; \bZ).
\end{equation}
As is evident from Lemma \ref{lem:description_boundary_moduli_toy}, this cochain is not closed.  To construct a cycle, consider the fibre product of $ \beta_{0}$ and $\scrR_{1}^{1}(\Mbar; (1,0)| L  ) $ over evaluation at the interior marked point.  Using  the evaluation map $\ev_*$ at the boundary marked point of elements of $\scrR_{1}^{1}(\Mbar; (1,0)| L  ) $, we obtain a cochain:
\begin{align} 
\sCO^{0}( \beta_0)  =  \ev_{*}  [\beta_{0}\times_{\Mbar} \scrR_{1}^{1}(\Mbar; (1,0)| L  )] & \in C^{1}(L; \bZ).\label{eqn:cancel-strata}
\end{align}

\begin{lem} \label{Lem:order-zero-part-of-CO}
  The sum of the restriction of $gw_{1}$ with $b_{D}^{0}$, and $\sCO^{0} \beta_0$ defines a cycle
\begin{equation} \label{eq:0-part-CC_cocycle}
b^0_D =  \tilde{b}_{D}^{0} + gw_1|L + \sCO^{0}( \beta_0) \in C^1(L; \bZ) .
\end{equation}
\end{lem}
\begin{proof}
We view $\beta_0$ as a locally finite chain with boundary in $(D_0 \cap D_r) \cup D_0^{sing}$. Every holomorphic curve which is not contained in $D_0$, and which passes through the singular locus of $D_0^{sing}$, has intersection number with $D_0$ strictly greater than $1$, see Fulton's \cite[Proposition 7.2]{Fulton}.  It follows that the evaluation map at the interior marked point $\bar{\scrR}_1^1(\Mbar, (1,0)|L) \rightarrow \Mbar$ has image disjoint from $D_0^{sing}$.  

We claim that the evaluation image is also disjoint from $D_0 \cap D_r$.   To see this, we argue by contradiction as in the proof of Lemma \ref{lem:description_boundary_moduli_toy}.   First, suppose that the interior marked point lies on a disc component with boundary on $L$. Since $L$ is disjoint from $D_0$ and $D_r$, this component would then have strictly positive intersection number with $D_r$ and strictly positive Maslov index. Since all Chern zero rational curves themselves have strictly positive intersection number with $D_r$, by Lemma \ref{lem:chern_0_positive_D_r}, the total intersection number of the stable map with $D_r$ would be positive, which is a contradiction. 

Alternatively, the interior marked point lies on a sphere component. If the disc component of the stable map meets $D_0$, the sphere bubble has vanishing Chern number, and again the total intersection number of the configuration with $D_r$ is strictly positive.  Therefore, the disc component is disjoint from $D_0$, hence constant as in Case (iii b) of Lemma \ref{lem:description_boundary_moduli_toy}.  We have a tree of sphere bubbles, the principal one of which passes through $L$, hence is not contained in $D_0$ or $D_r$ and which has positive Maslov index.  All other sphere components have vanishing Chern number, hence meet $D_r$ strictly positively.  This would again lead to a contradiction.

 It follows that the fibre product in \eqref{eqn:cancel-strata} can be made transverse; the result is then immediate from combining Lemma \ref{lem:description_boundary_moduli_toy} with the definitions of $\beta_0$, and $gw_1$. 
\end{proof}

\begin{defn}
A Lagrangian brane $L$ is \emph{(infinitesimally) invariant} if the cycle in Equation \eqref{eq:0-part-CC_cocycle} is null-homologous. An \emph{(infinitesimally) equivariant} structure on $L$, over $\bk$,  is a choice of bounding cochain in $ C^0(L; \bk) $ for this cycle. 
\end{defn}

\begin{rem}
  The terminology is justified as follows, cf. Remark \ref{Rem:justify-nc}. Under mirror symmetry, a holomorphic vector field on the mirror of $M$ gives rise to an $\nc$-vector field. Assume that such a vector field integrates to a $\bC^{*}$ action, and hence to an action on the category of coherent sheaves, to which one can associate an equivariant category. The condition of a Lagrangian being (infinitesimally) invariant is mirror to a sheaf being invariant under the $\bC^*$-action, and the equivariant structure corresponds to a lift to the equivariant category.
\end{rem}

In the next sections, we extend the construction above to the Hochschild cochains of the Fukaya category. We will then work with discs with strip-like ends rather than with boundary marked points.  The passage between the two involves gluing abstract operators over half-planes to strip-like ends, cf. Figure \ref{Fig:BdryVersusStriplike}.
\begin{center}
\begin{figure}[ht]
\includegraphics[scale=0.4]{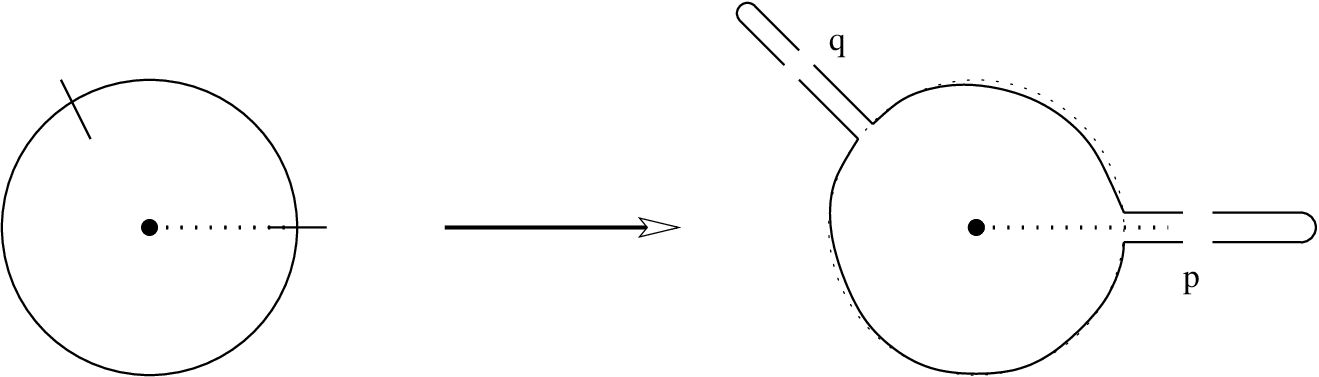}
\caption{Replacing boundary marked points by strip-like ends\label{Fig:BdryVersusStriplike}}
\end{figure}
\end{center}
  Since the endomorphism $b^1: HF(L,L') \rightarrow HF(L,L')$ preserves the absolute grading, in our setting the absolute indices associated to the two boundary marked points co-incide.  Gluing on the relevant operators over half-planes, the resulting closed boundary condition is a loop of Maslov index zero.  In the following sections, we will move back and forth between these two points of view when $L=L'$.

\subsection{From the closed sector to Hochschild cochains} \label{sec:relating-closed-open}

In order to relate Gromov-Witten theory to the Fukaya category, we consider   the cyclohedron $\scrR_{1}^{k+1}$, which is the moduli space of discs with $k+1$ boundary marked points and a single interior marked point.  There is a natural map
\begin{equation}
  \scrR_{1}^{k+1} \to \scrR^{k+1}
\end{equation}
obtained by forgetting the interior marked point.

Having fixed Floer data on the moduli space $\scrR^{k+1}$ for the purpose of defining the $A_{\infty}$ structure, we choose such data on the moduli space $\scrR^{k+1}_{1}$ of discs with one interior marked point. The setup is completely standard; the reader may compare to \cite[Section 4]{Abouzaid:generation} for a related construction involving symplectic cohomology.   As in \eqref{eq:Floer_equation_0-puncture}, we work with the fixed integrable complex structure $J$ on $\Mbar$.   On a given surface, the Floer data defining the pseudo-holomorphic curve equation therefore comprise a closed 1-form $\alpha_{\Sigma}$ vanishing on the boundary whose pullback under the ends agrees with $dt$, a family of Hamiltonians $H_{z}$ on $\Mbar$, parametrised by $z \in \Sigma$, with compact support disjoint from the neighbourhood $\nu D$ containing $D_0'$ and which are constant on the ends $\epsilon, \epsilon_i$. If $X_z$ is the Hamiltonian flow of $H_z$, we obtain a pseudo-holomorphic curve equation
\begin{equation} \label{eq:dbar_equation_discs_two_punctures}
(du - X_z \otimes \alpha_{\Sigma})^{0,1} = 0.
\end{equation}

To state precisely  the required properties of the \emph{inhomogeneous term} $X_z \otimes \alpha_\Sigma  $, consider a  sequence $(L_0, \ldots, L_k)$ of Lagrangians. Label the segments along the boundary of $\Delta$ counterclockwise, starting with the segment $(p_0,p_1)$, by the Lagrangians $L_i$, and choose a  pseudo-holomorphic equation subject to the following constraints:
\begin{itemize}
\item  all inhomogeneous terms vanish outside a compact set that is disjoint from $D$; 
\item if $z$ lies in a neighbourhood of $p_{i+1}$, then $H_z$ agrees with the Hamiltonian $H_{L_i,L_{i+1}}$ used to define $CF^*(L_i,L_{i+1})$.
\end{itemize}

 If $(x_0, \ldots, x_k)$ is a sequence of chords $x_{i} \in \Chord(L_{i-1}, L_i)$, \begin{equation} \label{eq:discs_one_interior_marked_point}
\scrR_{1}^{k+1}(\Mbar | x_0; x_k ,\cdots, x_1 ) 
\end{equation}
denotes the corresponding moduli space of maps into $\Mbar$, with boundary conditions $(L_0, \ldots, L_k)$ and asymptotic conditions $(x_0, \ldots, x_k)$.  
Evaluation at the interior marked point defines a map
\begin{equation} \label{eq:evaluate_disc_one_interior}
   \scrR_{1}^{k+1}(\Mbar| x_0; x_k ,\cdots, x_1 ) \to \Mbar.
\end{equation}

This moduli space decomposes as a union
\begin{equation}
 \scrR_{1}^{k+1}(\Mbar | x_0; x_k ,\cdots, x_1  )  = \coprod_{0 \leq d_0, \, 0\leq d_r } \scrR_{1}^{k+1}(\Mbar, (d_0,d_r)  | x_0; x_k ,\cdots, x_1  ) ,
\end{equation}
where each component of the right hand side consists  of curves whose intersection number with $D_0$ is $d_0$ and with $D_r$ is $d_r$. We shall be particularly interested in the cases $d_0=0$, which consists of curves whose image lies in $\Mbar \backslash D_0$, and $d_0=1$, which are curves meeting $D_0$ once. These moduli spaces have virtual dimension
\begin{equation} \label{eq:dimension_1_marked_point_degree_d}
      k +  2d_0 + \deg(x_0) - \sum_{i=1}^k \deg(x_i).
\end{equation}

We choose the inhomogeneous data defining these moduli spaces so that the spaces of discs with one interior marked point are regular, and so that
\begin{equation} \label{eq:transversality_bad_locus_evaluation}
    \parbox{34em}{the restriction of the evaluation map to $ \scrR_{1}^{k+1}(\Mbar, (0,d_r)  | x_0; x_k ,\cdots, x_1  ) $ is transverse to $B_r$, and the restriction to $\scrR_{1}^{k+1}(\Mbar, (1,0)  | x_0; x_k ,\cdots, x_1  )  $ is transverse to $B_0$.}
\end{equation}
Whenever $d_0=d_r = 0$, we write
\begin{equation} \label{eq:moduli_space_OC}
 \scrR_{1}^{k+1}(M| x_0; x_k ,\cdots, x_1  )  =   \scrR_{1}^{k+1}(\Mbar, (0,0)  | x_0; x_k ,\cdots, x_1  ).
\end{equation}
Indeed, positivity of intersection and  ampleness of the divisor $D$ supported on $ D_0 \cup D_{\infty} \cup D_r$ implies that curves in $\Mbar = \Mdbar \backslash D_{\infty}$ whose intersection numbers with $D_0$ and $D_r$ both vanish are disjoint from both, which implies that their image lies in $M$.

The moduli  spaces in Equation \eqref{eq:moduli_space_OC} give rise to a map
\begin{equation}
\CO \co  C^*(M; \bk) \to CC^*(\Fuk(M),\Fuk(M)).
\end{equation}
 In the compact setting, such a map was defined by Fukaya, Oh, Ohta, and Ono in \cite{FO3}, and in the non-compact setting  by  Seidel in \cite{Seidel:ICM} (see \cite{Ganatra} for a detailed implementation).  In brief, if we denote by
 \[
 \scrF(L_0,\ldots,L_n) \ = \ hom_{\scrF(M)}(L_{n-1},L_n) \otimes \cdots \otimes hom_{\scrF(M)}(L_{0},L_1)
 \]
 then the map
\[
\CO^d: C^*(M;\bk) \to \prod_{n\geq 0} \prod \mathrm{Hom}(\scrF(L_0,\ldots,L_n),\\\scrF(L_0,L_n)[d])
\]
(where the second product is over tuples of objects $(L_0,\ldots,L_n)$) 
is defined by counting holomorphic discs whose source is an arbitrary element of the moduli space \eqref{eq:moduli_space_OC}, with $k=n$, with an interior marked point constrained to lie on a locally finite cycle representing a generator in $C^*(M;\bk)$, and with Lagrangian boundary conditions given by the $L_i$.  It is a fact that $\CO$ defines a chain map, and is then a unital ring homomorphism on cohomology.  The most non-trivial point in the construction of $\CO$ is to ensure that the Gromov-Floer compactifications of these moduli spaces do not involve any maps whose images leave $M$. The argument of Lemma \ref{lem:c-0-estimate-discs} applies verbatim to establish this.  

The first term of $\CO$ assigns to every cochain in $M$ a class in $CF^*(L,L)$; we denote this map
\begin{align}
   C^*(M ; \bk) & \to CF^*(L,L) \\
c & \mapsto c|L,
\end{align}
since it agrees with the classical restriction of cochains used in Section \ref{sec:equiv-struct} under the natural quasi-isomorphism from the classical cochains of $L$ with its self-Floer cochains.

Next, we study the moduli spaces  $\scrR_{1}^{k+1}(\Mbar, (1,0)  | x_0; x_k ,\cdots, x_1  )  $. In order to define operations using this moduli space, we need better control of the compactification.

\begin{Lemma} \label{lem:moduli_space_one_marked} 
The image of  $ \bar{\scrR}_{1}^{k+1}(\Mbar, (1,0)  | x_0; x_k ,\cdots, x_1  ) $ under the evaluation map is disjoint from  the  singular locus of $D_0$.  It is also disjoint from $D_0 \cap D_{r}$ if the virtual dimension of $ \bar{\scrR}_{1}^{k+1}(\Mbar, (1,0)  | x_0; x_k ,\cdots, x_1  ) $  is less than $6$, i.e. if 
\begin{equation} \label{eq:virtual_dimension_bound}
        k + \deg(x_0) - \sum_{i=1}^k \deg(x_i)  < 4.
\end{equation}
\end{Lemma}
\begin{proof}
As in Lemma \ref{Lem:order-zero-part-of-CO}, 
every curve not contained in $D_0$ and which meets $D_0^{sing}$  has intersection number $\geq 2$ with $D_0$ by \cite[Proposition 7.2]{Fulton}. This implies the first part of the Lemma.

Next, assume for contradiction that some element $u$ of $  \bar{\scrR}_{1}^{k+1}(\Mbar, (1,0)  | x_0; x_k ,\cdots, x_1  )  $ maps to $D_0 \cap D_r$ under the evaluation map $\ev$ at the interior marked point.  There are two cases to consider:

{\bf Case (i):} The interior marked point lies on a disc component. In this case, since the intersection number with $D_0$ is $1$, all sphere components must have vanishing Chern number. On the other hand, Chern zero spheres meet $D_r$ strictly positively, which contradicts the assumption that the intersection number with $D_r$ vanishes.

{\bf Case (ii):} The interior marked point lies on a sphere component.  By considering the intersection number with $D_0$, we see that the disc component of $u$ is  an element of 
\begin{equation} \label{eq:discs_one_interior_M_h}
 \coprod_{0 \leq d_r } \scrR_{1}^{k+1}(\Mbar, (0,d_r)  | x_0; x_k ,\cdots, x_1  ) ,
\end{equation}
which is a moduli space of  virtual dimension less than $4$ if the inequality \eqref{eq:virtual_dimension_bound} holds. There is a unique sphere component of Chern number $1$, and all other components are spheres of vanishing Chern number.  Since the total intersection number of $u$ with $D_r$ vanishes, and all spheres of vanishing Chern number intersect $D_r$ positively, the Chern $1$ sphere is contained in $D_r$. 

Recall that the singular locus of $\scrM_1(\Mbar|1)$ evaluates into the real codimension 4 subset $B_r \subset \Mbar$.   Assumption \eqref{eq:transversality_bad_locus_evaluation} and Equation \eqref{eq:virtual_dimension_bound} imply that the moduli spaces in \eqref{eq:discs_one_interior_M_h} have evaluation image outside $B_r$.  For generically chosen  perturbation data, these evaluation images  will therefore be transverse to the evaluation image of $\scrM_1(\Mbar|1)$; that transversality implies that the images are in fact disjoint, for dimension reasons.  We conclude that  the restriction of the evaluation map to the moduli spaces in \eqref{eq:discs_one_interior_M_h}  has image disjoint from stable rational curves of total Chern number $1$.  This excludes the existence of such a configuration $u$.
\end{proof}

Lemma \ref{lem:moduli_space_one_marked} implies that, if $j \leq 3$ we have a well defined map
\begin{equation} \label{eq:map_cycles_D_0_relative_CC}
 \sCO \co  C_{2n-2-j}^{lf}(D_0 ,  (D_{0} \cap D_{r})  \cup D_{0}^{sing} ; \bk) \to CC^{j}(\Fuk(M),\Fuk(M))
\end{equation}
which can be constructed as follows. Let $\Diff(\Mbar, D_0 \cup  D_r)$ denote the group of diffeomorphisms of $\Mbar$ which preserves $D_0$ and $D_r$, in particular $D_0 \cap D_r $ and $ D_{0}^{sing} $ are preserved. Given a (smooth map on a simplex contributing to a locally finite)  chain $\beta \co \Delta^{2n-2-j} \to D_0$, and a moduli space $\scrR_{1}^{k+1}(\Mbar, (1,0)  | x_0; x_k ,\cdots, x_1  ) $ of virtual dimension $j$, the fibre product
\begin{equation} \label{eq:fibre_product_chain_D_0}
\scrR_{1}^{k+1}(\Mbar, (1,0)  | x_0; x_k ,\cdots, x_1  )  \times_{\Mbar} (\phi_{\beta}  \circ \beta),
\end{equation}
is transverse for a generic map
\begin{equation}
  \phi_{\beta} \co \Delta^{2n-2-j} \to \Diff(\Mbar, D_0 \cup D_r ),
\end{equation}
where $\phi_{\beta} \circ \beta$ denotes the pointwise composition over $\Delta^{2n-2-j}$.
The moduli space $\scrR_{1}^{k+1}(\Mbar, (1,0)  | x_0; x_k ,\cdots, x_1  ) $ is moreover naturally oriented relative to the orientation lines of the chords $x_i$. Hence, whenever
\begin{equation}
   k + \deg(x_0) - \sum_{i=1}^k \deg(x_i)  = j,
\end{equation}
we multiply by a sign whose parity is $j + \sum_{i=1}^k \deg x_i  $ and we obtain a map
\begin{equation}
 \sCO_{u}(\beta) \co \ro_{x_k} \otimes \cdots \otimes \ro_{x_1} \to  \ro_{x_{0}}
\end{equation}
associated to elements of  the finite set \eqref{eq:fibre_product_chain_D_0}. Taking the sum over all such elements defines the desired map in Equation \eqref{eq:map_cycles_D_0_relative_CC}.

\begin{Lemma}
  If $j \leq 3$, we have a commutative diagram
\begin{equation}
\xymatrix{C^{lf}_{2n-2-j}(D_0 , (D_{0} \cap D_{r})  \cup D_{0}^{sing} ; \bk) \ar[r]^-{\sCO} & CC^{j}(\Fuk(M),\Fuk(M)) \\
C^{lf}_{2n-2-j-1}(D_0 , (D_{0} \cap D_{r})  \cup D_{0}^{sing} ; \bk) \ar[r]^-{\sCO}  \ar[u]^{\partial} & CC^{j-1}(\Fuk(M),\Fuk(M)) \ar[u]_{\delta} }
\end{equation} 
with $\partial$ and $\delta$ the relevant chain complex boundary operators.\qed
\end{Lemma}
\begin{proof}
It suffices to ensure that the boundary of the fibre product in Equation \eqref{eq:fibre_product_chain_D_0} decomposes into strata which correspond to the composition of $\sCO$ with the differential in the source and target.  Lemma \ref{lem:moduli_space_one_marked} implies such a decomposition holds, provided the perturbations $\phi_{\beta}$ are compatible across boundary strata of the simplices $\beta$, and the Floer perturbation data as usual is compatible with  the compactification $ \bar{\scrR}_{1}^{k+1}(\Mbar, (1,0)  | x_0; x_k ,\cdots, x_1  ) $ of $\scrR_{1}^{k+1}(\Mbar, (1,0)  | x_0; x_k ,\cdots, x_1  ) $. We therefore choose maps $\phi_{\beta}$ first for all $2n-5$-dimensional chains $\beta$ so that transversality is achieved (this is the case $j=3$), then choose the maps for higher dimensional simplices by extending the boundary values given by induction.
\end{proof}
We note that the $0$\th order term of the map $\sCO$, together with the natural quasi-isomorphism from Floer to ordinary cohomology, recovers the map $\sCO^0$ used in Section \ref{sec:equiv-struct}, see Equation \eqref{eq:0-part-CC_cocycle}.

\subsection{Seidel-Solomon moduli spaces of discs}\label{Sec:moduli}


Let $\Delta$ denote the closed unit disc. For $k\geq 0$, let $\Mod{(0,1)}{k+1}$ denote the moduli space of domains comprising the disc $\Delta$
\begin{enumerate}
\item with two marked points  $z_0 = 0$ and $z_1  \in (0,1)$, and
\item with $k+1$ boundary punctures at  $p_0 = 1 \in\partial \Delta$ and points $\{p_1,\ldots,p_k\} \subset \partial \Delta \backslash \{1\}$ ordered counter-clockwise.
\end{enumerate}
It is important to note that the point $1$ will play the role of an output, while the points $p_i$ will correspond to inputs. At the level of Floer theory, this asymmetry is implemented by choosing strip-like ends of a different nature near these punctures: let $Z_- = (-\infty,0] \times [0,1]$ and $Z_+ = [0,\infty) \times [0,1]$.    Given $\Sigma \in \Mod{(0,1)}{k+1}  $, a choice of strip-like ends is a collection of conformal embeddings of punctured half-strips 
\begin{equation}
\epsilon_{0}: Z_- \rightarrow \Sigma, \quad \epsilon_i: Z_+ \rightarrow \Sigma \quad \textrm{for} \, 1 \leq i \leq k\end{equation}
 which take $\partial Z_{\pm}$ into $\partial \Sigma$, which converge at the end to the punctures $p_i$.

This space has a natural  compactification $\Modbar{(0,1)}{k+1} $ of Deligne-Mumford type. To describe its  codimension one boundary facets, recall that we denote the associahedron by $\scrR^{k+1}$. 
The boundary strata of $\Modbar{(0,1)}{k+1} $ arise from the following degenerations: 
\begin{enumerate}
\item ($\{p_i, \ldots, p_{i+j}\}$ move together) A nodal domain in which a collection of input boundary punctures bubble off:
  \begin{equation} \label{eq:first_boundary_mod}
    \coprod_{\substack{2 \leq  j \leq k \\ 1 \leq i \leq k-j}} \Mod{(0,1)}{k-j+1} \times \scrR^{j+1} 
  \end{equation}
Note that the parameter $i$ does not appear in the components, but corresponds to the boundary marked point of $\Delta $ along which the two discs are attached.
\item ($z_1 \rightarrow 0$) A domain with a sphere bubble carrying the two interior marked points attached to a disc with only boundary marked points. Letting $\scrM_{0,3}$ denote the moduli space of spheres with $3$ marked points, this component is
  \begin{equation} \label{eq:second_boundary_mod}
   \scrR^{k+1}_{1} \times \scrM_{0,3}
  \end{equation}
\item ($\{p_{k-l}, \ldots, p_k,  p_1,\ldots, p_i \} \rightarrow 1$) A domain with two discs, one carrying both interior marked points and a subset $\{p_{i+1},\ldots, p_{k-l-1}\}$ of the boundary inputs, the other carrying the remaining inputs and the output $p_0$.
  \begin{equation}\label{eq:third_boundary_mod}
    \coprod_{1 \leq  i + l  \leq k}  \scrR^{i+l+1}  \times \Mod{(0,1)}{k-i-l+1} 
  \end{equation}
\item ($\{z_1\} \cup \{p_1,\ldots, p_i, p_{k-l}, \ldots, p_k\} \rightarrow 1$) A nodal domain with two discs, one carrying $k-i-l$ input boundary marked points and one interior marked point, the bubble carrying the second interior marked point, the remaining boundary marked points and the outgoing point $p_0$.
  \begin{equation} \label{eq:extra_boundary_empty}
    \coprod_{0 \leq  i + l  \leq k}  \scrR^{i+l+1}_{1}  \times \scrR^{k-i-l}_{1}
  \end{equation}
\end{enumerate}

Figure \ref{Fig:FourTypes} depicts the four kinds of degeneration.
\begin{center}
\begin{figure}[ht]
\includegraphics[scale=0.6]{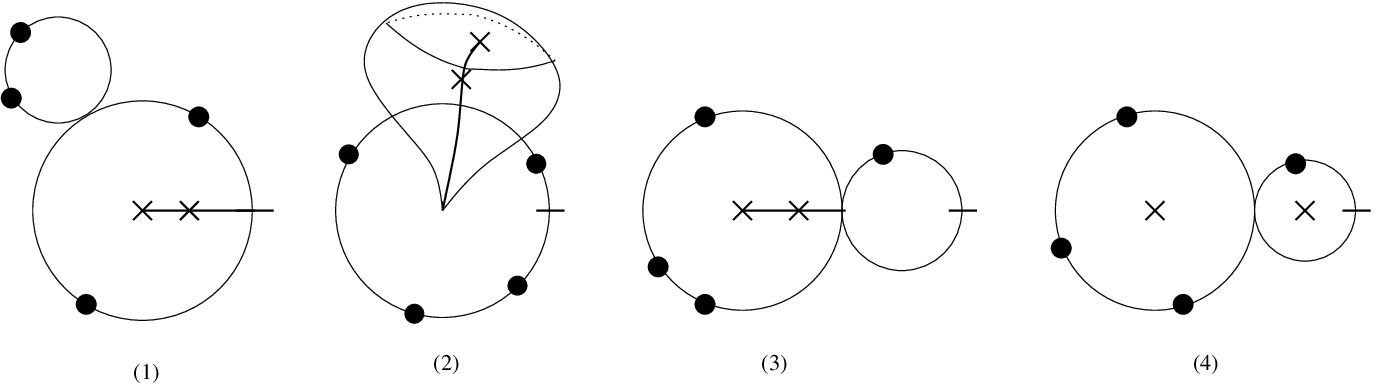}
\caption{Boundary strata of the compactification $\Modbar{(0,1)}{k+1} $ \label{Fig:FourTypes}}
\end{figure}
\end{center}

  Once families of strip-like ends are chosen for all elements of the moduli space, there are gluing maps defined for sufficiently large parameters, which for the first stratum for instance takes the form 
\begin{equation}
\Mod{(0,1)}{k-j+1} \times \scrR^{j+1} \times [\rho,\infty) \rightarrow  \Mod{(0,1)}{k+1} .
\end{equation}

\subsection{Seidel-Solomon moduli spaces of maps}

We study moduli spaces of pseudo-holomorphic maps from domains in $  \Mod{(0,1)}{k+1} $ to $\Mbar$, and hence require perturbation data parametrised by such moduli spaces as in Section \ref{sec:relating-closed-open}. The data will still be comprised of the auxiliary terms appearing in Equation \eqref{eq:dbar_equation_discs_two_punctures}.

Given such a choice, we define
\begin{equation}
  \Mod{(0,1)}{k+1}(x_0; x_k, \ldots, x_1)
\end{equation}
to be the moduli space of finite energy maps
\begin{equation}
  u \co \Sigma \to \Mbar,
\end{equation}
whose relative homology class satisfies
\begin{equation} \label{eq:intersection_numbers_discs}
\begin{aligned}
{[u] \cdot [D_{r}]} & = 0 \\
{[u] \cdot [D_0]} & =1,
\end{aligned}  
\end{equation}
and which solve Equation \eqref{eq:dbar_equation_discs_two_punctures}, mapping the boundary segment between $p_i$ and $p_{i+1}$ to $L_i$, which converge along the $i$\th end to $x_i$, and such that
\begin{equation} \label{eq:constraints_marked_points}
  u(z_0) \in D_0 \textrm{ and } u(z_1) \in D'_0.
\end{equation}

Using positivity of intersection, we can recast Conditions  \eqref{eq:intersection_numbers_discs} and \eqref{eq:constraints_marked_points} as
\begin{equation} \label{eq:divisor_constraints}
 u^{-1}(D_r) = \emptyset, \,\, u^{-1} (D_0) = z_0  \textrm{, and } u^{-1} (D'_0) = z_1.
\end{equation}

\begin{Lemma} \label{lem:vir_dim_Mod}
  The virtual dimension of $\Mod{(0,1)}{k+1}(x_0; x_k, \ldots, x_1)$ is
  \begin{equation}
    k - 1 +  \deg(x_0) - \sum_{i=1}^k \deg(x_i).
  \end{equation}
\end{Lemma}
\begin{proof}
  The Fredholm index of any element of $\Mod{(0,1)}{k+1}(x_0; x_k, \ldots, x_1)$ is
  \begin{equation}
    2 + \deg(x_0) - \sum_{i=1}^k \deg(x_i).
  \end{equation}
The constant term is due to the fact that $D_0$ represents twice the Maslov class in $H^{2}(\Mbar,L)$ for each Lagrangian $L$ that we consider, and the fact that we are considering discs whose intersection number with $D_0$ is $1$. The dimension of $\Mod{(0,1)}{k+1} $ is $k+1 $, and the conditions imposed at the two interior marked points amount to a real codimension $4$ constraint.
\end{proof}

\subsection{Compactification and transversality for Seidel-Solomon spaces}

We shall describe the Gromov-Floer  compactification $ \Modbar{(0,1)}{k+1}(x_0; x_k, \ldots, x_1) $ introduced above whenever its virtual dimension is $0$ or $1$, under appropriate transversality assumptions.  

For a sequence $x_i$ of Floer generators, we define
\begin{equation} \label{eq:moduli_in_M_r}
  \scrR^{k+1}(\Mbar | x_0; x_k ,\cdots, x_1) 
\end{equation}
to be the moduli space of solutions to Equation \eqref{eq:Floer_equation_0-puncture} in $\Mbar$, with the same boundary conditions as in Section \ref{sec:fukaya-category}. 

 Note that Hypothesis \ref{hyp:no_Maslov_0_disc_Maslov_1} implies that 
if $x \in \Chord(L,L)$, then $ \scrR^{1}(\Mbar|x)$ contains no element with trivial intersection with $D_0$.  More precisely, Gromov compactness shows that by taking the perturbation used to define $\Chord(L,L)$ small, an element of $ \scrR^{1}(\Mbar|x)$ is close to a disc with boundary on $L$.  The 3rd part of Hypothesis  \ref{hyp:no_Maslov_0_disc_Maslov_1} then implies no element of $\scrR^1(\Mbar|x)$ has trivial intersection with $D_0$ up to some preassigned energy level $E$, and the energy level $E$ goes to infinity as the perturbation goes to zero.  For simplicity we will elide this point.

For $k \geq 1$, we assume that the data in Section \ref{sec:fukaya-category} are chosen generically so that
\begin{equation} \label{eq:regularity_discs_M_r}
  \parbox{34em}{the moduli spaces $  \scrR^{k+1}(\Mbar \backslash D_0 | x_0; x_k ,\cdots, x_1)  $  are regular.}
\end{equation}
The techniques for achieving transversality are again standard: for $k \geq 2$, the underlying curves are stable, so we can use inhomogeneous perturbations depending on the source as in  \cite{FCPLT}, while for $k=1$, we use the result of Floer-Hofer-Salamon \cite{FHS} which yields a somewhere injective point even in the presence of an $\bR$-symmetry group of translations of the domain.

We can now analyse the strata of $ \Modbar{(0,1)}{k+1}(x_0; x_k, \ldots, x_1) $. A picture of the possible degenerations of one-dimensional moduli spaces is given in Figure \ref{Fig:DilationDegenerate}. 
\begin{center}
\begin{figure}[ht] 
\includegraphics[scale=0.5]{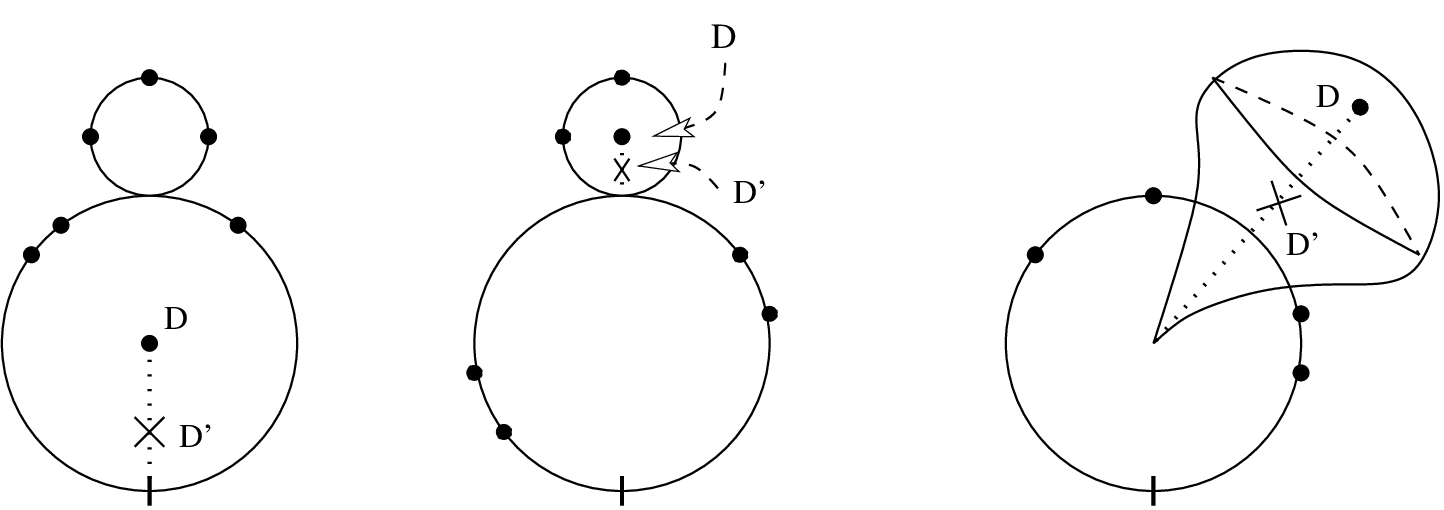}
\caption{\label{Fig:DilationDegenerate}
}
\end{figure}
\end{center}
\begin{Lemma} 
  The strata corresponding to $z_1 \to 1$ are empty.
\end{Lemma}
\begin{proof}
  For such strata, $z_0$ and $z_1$ lie on distinct components of the domain. The constraints that $z_0$ map to $D_0$ and $z_1$ to $D'_0$, together with positivity of intersection, imply that the intersection numbers of these components with $D_0$ are both strictly positive. Since $D_0$ is non-negative on all rational curves in $\Mbar$, we conclude that any such stable curve has algebraic intersection number with $D_0$ strictly greater than $1$. This contradicts the fact that elements of $ \Mod{(0,1)}{k+1}(x_0; x_k, \ldots, x_1)  $ have intersection number $1$ with $D_0$.
\end{proof}
Next, we consider the case $z_1 \in (0,1)$. We shall show that there can be no component which is a rational curve. Just as in Lemma \ref{lem:description_boundary_moduli_toy}, any rational curve component would have vanishing Chern number, so meet $D_0$ trivially and $D_r$ strictly positively, by  Lemma \ref{lem:chern_0_positive_D_r}. This would then contradict the fact that the total intersection number with $D_r$ vanishes, and again implies that all components other than the one carrying the interior marked points are discs with image in $M$. Having assumed regularity for these moduli spaces, we conclude:
\begin{Lemma} \label{lem:boundary_mod_1}
  The strata of $ \Modbar{(0,1)}{k+1}(x_0; x_k, \ldots, x_1) $ corresponding to $z_1 \in (0,1) $ have dimension equal to their virtual dimension. In particular, they are empty whenever the virtual dimension of $ \Mod{(0,1)}{k+1}(x_0; x_k, \ldots, x_1) $ is negative,  and consist of points when the virtual dimension of this space is  $0$, and arcs when it is $1$.  In the $1$-dimensional case, the boundary strata are:
\begin{align} \label{eq:boundary_stratum_differential_output_1}
& \coprod_x \  \Modbar{(0,1)}{k-k_1+1}(x_0; x_k, \ldots, x_{i+k_1+1}, x, x_{i},  \ldots,  x_1) \times \scrR^{k_1+1}(M| x; x_{i+k_1} ,\cdots, x_{i+1}  ) \\\label{eq:boundary_stratum_differential_output_2}
& \coprod_x \  \scrR^{k_1+1}(M| x_0, x_k, \ldots, x_{i+k_1+1}, x, x_{i},  \ldots,  x_1) \times \Modbar{(0,1)}{k_1}(x; x_{i+k_1+1},  \ldots , x_{i}) 
\end{align} 
where the union is over generators $x$ satisfying
\begin{align}
 {\deg(x) = \deg(x_{i+1}) + \cdots + \deg(x_{i+k_1}) - k_1 +2 } & \quad \mathrm{in} \ \eqref{eq:boundary_stratum_differential_output_1} \\ 
 {\deg(x) = \deg(x_{i+1}) + \cdots + \deg(x_{i+k_1}) - k_1 +1}&  \quad \mathrm{in} \ \eqref{eq:boundary_stratum_differential_output_2}.
\end{align}
\qed
\end{Lemma}
Finally, we analyse the more delicate case $z_1 \to 0$. 
\begin{lem} \label{lem:description_boundary_moduli}
If the virtual dimension of $\Mod{(0,1)}{k+1}(x_0; x_k, \ldots, x_1) $ is strictly less than $3$, all strata corresponding to $z_1 \to 0 $ are regular and those of top dimension are
\begin{align} \label{eq:first_stratum_interior_bubble}
&   \scrM_{1}(\Mbar| 1) \times_{\Mbar} \scrR_{1}^{k+1}(M| x_0; x_k ,\cdots, x_1  )  \\ \label{eq:second_stratum_interior_bubble}
& B_{0}\times_{\Mbar} \scrR_{1}^{k+1}(\Mbar; (1,0)| x_0; x_k ,\cdots, x_1  ).
\end{align}
\end{lem}
\begin{proof}
The proof again parallels Lemma \ref{lem:description_boundary_moduli_toy}. 
As usual there is only one component with non-vanishing intersection number with $D_0$.

{\bf Case (i):} If this component is a disc, then all rational curve components have vanishing Maslov index, which excludes the presence of bubbles and implies that all other disc components map to $M$ using the vanishing of intersection with $D_{r}$. We conclude that all the moduli spaces that appear in such a configuration are regular by our choice of inhomogeneous data, hence that the top dimensional strata in this case have only one disc component which is an element of $\scrR_{1}^{k+1}(\Mbar, (1,0)| x_0; x_k ,\cdots, x_1  ) $. This stratum is  attached to a ghost sphere carrying the marked points  $z_1$ and $z_0$. Since these interior marked points are required to map to $D_0$ and $D'_0$, this corresponds to Equation \eqref{eq:second_stratum_interior_bubble}.

{\bf Case (ii):} Assume all discs are disjoint from $D_0$. Since our data is chosen, cf. \eqref{eq:regularity_discs_M_r},  to ensure regularity of the moduli spaces of discs with $0$ or $1$ interior marked points in the complement of $D_0$, in codimension $1$ the following is the only possible type of configuration.   There is one disc component, which is an element of $ \scrR_{1}^{k+1}(\Mbar, (0,d_r)| x_0; x_k ,\cdots, x_1  )  $ for some integer $d_r$, which  is attached to a tree of sphere bubbles, one of which is a sphere which intersects $D_0$ once.  Forgetting the other components and marked points, we identify this sphere component with an element of $\scrM_1(\bar{M}|1)$.  Condition \eqref{eq:transversality_bad_locus_evaluation} implies that, if the dimension is less than $3$, this sphere is disjoint from $D_{r}$. Since the total intersection with $D_r$ vanishes, we conclude from Lemma \ref{lem:chern_0_positive_D_r} that there can be no sphere of vanishing Chern class. We conclude that this boundary stratum is given by Equation \eqref{eq:first_stratum_interior_bubble}.
\end{proof}

\subsection{Construction of an \nc-vector field} \label{sec:construction-an-nc}
The moduli spaces $  \Mod{(0,1)}{k+1}(x_0; x_k, \ldots, x_1) $  admit natural orientations relative to the orientation lines $\ro_{x_i} $ and the moduli spaces of underlying curves $  \Mod{(0,1)}{k+1}  $. We fix the orientation for $ \Mod{(0,1)}{k+1}  $ induced by identifying the interior points of this moduli space with an open subset of
\begin{equation}
(0,1) \times  (\partial \Delta)^{k+1}.
\end{equation}
Using the fact that a $0$-dimensional manifold admits a natural orientation, which we twist by  $\sum_{i=1}^{k} i \deg x_i $, we obtain a canonical map
\begin{equation}
\tilde{b}^{u} \co \ro_{x_k}  \otimes \cdots  \otimes \ro_{x_1} \to \ro_{x_0}
\end{equation}
whenever $\deg(x_0) =k + 1 +   \sum_{i=1}^{k} \deg(x_i) $, and $u$ is an element of $  \Mod{(0,1)}{k+1}(x_0; x_k, \ldots, x_1) $. In particular, we define a Hochschild cochain
\begin{align}
  \tilde{b}_{D} &  \in CC^{*}( \Fuk(M), \Fuk(M) ) \\
\tilde{b}_{D}^{k} |  \ro_{x_k}  \otimes \cdots  \otimes \ro_{x_1} & = \sum_{x_0} \sum_{u \in  \Mod{(0,1)}{k+1}(x_0; x_k, \ldots, x_1)}  b^{u}.
\end{align}

Consider the sum
\begin{equation} \label{eq:total_hochschild_cochain}
\tilde{b}_{D} + \CO(gw_1) + \sCO( \beta_0 ) \in  CC^*(\Fuk(M),\Fuk(M)).
\end{equation}
To see that this is closed, we return to the description of the boundary strata of the moduli space in Lemma \ref{lem:boundary_mod_1}: these two strata respectively correspond to
\begin{equation}
\tilde{b}^i_D(\ldots, \mu_{\Fuk}^j(\ldots, \ldots), \ldots) \quad \textrm{and} \quad \mu_{\Fuk}^k(\ldots, \tilde{b}^j_D(\ldots, \ldots), \ldots).
\end{equation} 
The $A_{\infty}$-operations which arise in the first two pictures in Figure \ref{Fig:DilationDegenerate} involve products in the exact manifold $M$, i.e.  the disc not containing the interior marked point lies wholly in $M$; hence such degenerations are indeed governed by the structure coefficients $\mu^*_{\Fuk}$.  Lemma \ref{lem:description_boundary_moduli} accounts for two additional terms in the boundary of $b_D$, which are cancelled by the differential of the second two terms in Equation \eqref{eq:total_hochschild_cochain}. Indeed, Equations \eqref{eq:first_stratum_interior_bubble} and \eqref{eq:second_stratum_interior_bubble} respectively correspond to
\begin{equation}
   \CO(GW_1)  \textrm{ and } \sCO( [B_0] ).
\end{equation}
  Hypothesis \ref{Hyp:GWvanishes}, together with the fact that  $\CO$ and $\sCO$ are  chain maps, allows us to conclude:

\begin{Proposition} \label{prop:ConstructDilation}
The sum
\begin{equation}
b_D \ = \   \tilde{b}^{D} + \CO(gw_1) + \sCO( \beta_0 ) 
\end{equation}
defines an \nc-vector field $b$ for $\Fuk(M)$.
\end{Proposition}

The resulting \nc-vector field enables one to bigrade Floer cohomology groups between equivariant Lagrangian submanifolds by weights lying in the algebraic closure $\bar{\bk}$. The corresponding gradings are compatible with the Floer product,  in the sense that if $\alpha, \beta \in HF^*(L,L')$ are eigenvectors for the endomorphism associated to the \nc-vector field, with eigenvalues $\lambda_{\alpha}$ and $\lambda_{\beta}$,  then $\alpha \cdot \beta$ is also an eigenvector, and 
\begin{equation}
\lambda_{\alpha\cdot\beta} = \lambda_{\alpha} + \lambda_{\beta}.
\end{equation}
The proof follows from consideration of Figure \ref{Fig:WeightProduct}, which indicates the relevant degenerations of discs in $  \Modbar{(0,1)}{2}(x_0; x_2, x_1) $ with two input marked points (we have excluded degenerations involving sphere bubbles, which are cancelled by the choice of ambient bounding cycle, and for simplicity suppose that the equivariant structures on $L$ and $L'$ are trivial).  
\begin{center}
\begin{figure}[ht]
\includegraphics[scale=0.4]{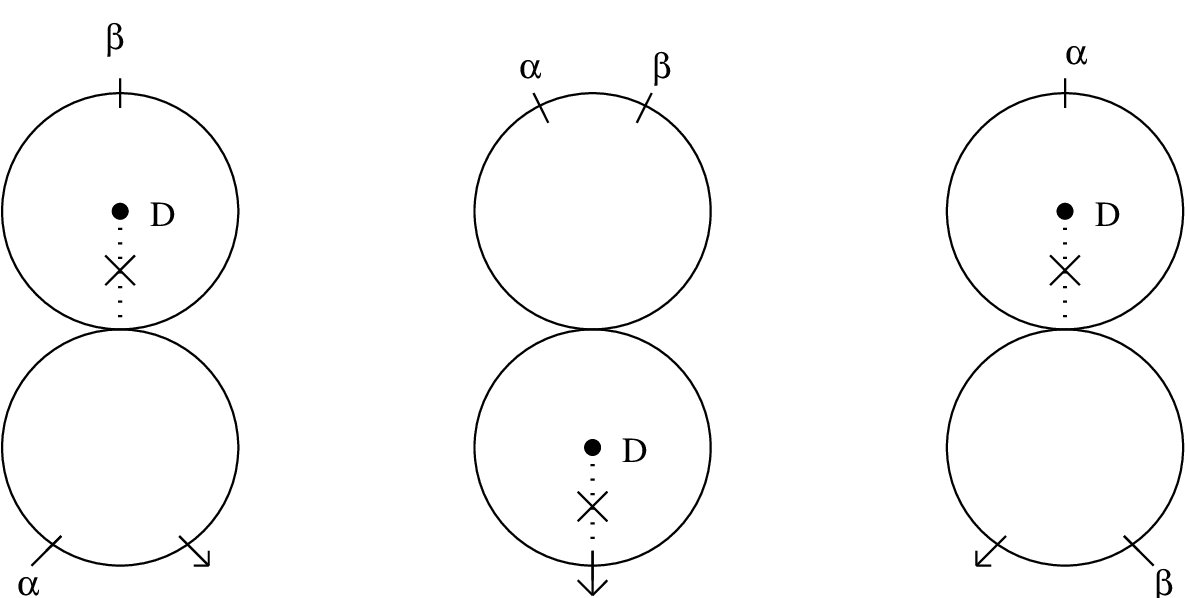}
\caption{Compatibility of weights and the Floer product\label{Fig:WeightProduct}}
\end{figure}
\end{center}

\begin{Remark} \label{Rem:IsotopyCaveat}
Note that we  have not proved that the weights on $HF^*(L,L')$ are invariant under Hamiltonian isotopies of $L, L'$, as this will in general require virtual perturbations.   Our discussion of transversality  assumed, Hypothesis \ref{hyp:no_Maslov_0_disc_Maslov_1},  that $L$ bounds no Maslov zero disc and is disjoint from $B_r$, and neither of those properties will be preserved under general isotopies.  
\end{Remark}

\begin{Remark} \label{Rem:Conormal}
Fix a domain $\Delta^{\to}$, which is the closed unit disc with an interior marked point at the origin and a tangent vector $v \in T_0\Delta$ which points along the positive real axis. Note this is not fixed by any non-identity element of $PSL_2(\bR)$.  Fix a generic conormal section $\nu$ on the smooth locus of the divisor $D_0$; the space of such is connected.  Rather than defining the element $b^0 \in CF^1(L,L)$ by counting discs with $0\in D_0$ and an intersection point with $D_0'$ along the arc $[0,1] \subset \Delta$, one can instead count holomorphic maps from $\Delta^{\to}$ to $\bar{M}$ which take $0 \mapsto D_0$ and for which the conormal section $\nu(v) \in \bR^{+}$.  There are similar counts with additional boundary inputs for defining the higher terms of the cocycle $\{b^i\}$.  This gives a route to building an \nc-vector field if some component of the divisor $D_0$ does not move, i.e. if there is no suitable $D_0'$.\end{Remark}

\section{Milnor fibres} \label{Sec:Milnor}

This section illustrates the previous construction of \nc-vector fields in a model case of the Milnor fibre of the $A_k$-singularity $\bC^2 / \bZ_{k+1}$.  The corresponding Fukaya category is known to be intrinsically formal \cite{SeidelThomas}, from which existence of a pure vector field can be inferred \emph{a posteriori}. However, the geometric construction of the \nc-vector field which appears in this section will underpin the more general construction for Hilbert schemes of the Milnor fibre which appears later, where the methods of \cite{SeidelThomas} do not apply.

\subsection{Dictionary}  To help orient the reader, we give a dictionary relating the objects appearing in the rest of this section with their avatars from Section \ref{Sec:Generalities}.  Comparing to Hypotheses \ref{Hyp:Main}, \ref{Hyp:regular}, \ref{Hyp:GWvanishes} and \ref{hyp:no_Maslov_0_disc_Maslov_1}, one has:

\begin{itemize}
\item The projective variety $\Mdbar$ is the rational surface obtained from blowing up $\bP^1\times\bP^1$ at $(k+1)$ points lying in distinct fibres of the second projection $\bP^1\times\bP^1 \rightarrow \bP^1$.  
\item Fix sections $s_0, s_{\infty}$ of the projection $\Mdbar \rightarrow \bP^1$ with the property that at each reducible fibre, one section hits each component.  Let $F_{\infty}$ denote the fibre over $\infty \in \bP^1$. Then $D_0 = \im(s_0) \cup \im(s_{\infty})$, $D_{\infty} = F_{\infty}$ and $D_r = \emptyset$. With these definitions, Hypothesis \ref{Hyp:Main} is established at the end of this section.  
\item The exact manifold $M = A_{k}$ is the usual Milnor fibre, and the partial compactification $\Mbar = \bar{A}_k$ is the properification of the Lefschetz fibration $M \rightarrow \bC$ (compactifying each $\bC^*$ fibre by points at $\pm\infty$).  We will be concerned with Lagrangian matching spheres in $M$.
\item The Chern one spheres in $\Mbar$ are components of reducible fibres, the moduli space $\scrM_1(\Mbar|1)$ is smooth, and we take $B_r = \emptyset$, which makes Hypothesis \ref{Hyp:regular} elementary, cf. Lemma \ref{Lem:monotoneAbar}.
\item Hypothesis \ref{Hyp:GWvanishes} (vanishing of the Gromov-Witten invariant) is proved in Section \ref{Sec:MilnorGW} below. The second part of Hypothesis \ref{Hyp:GWvanishes} is vacuous in this case, since $B_0$, $D_r$ and $D_0^{sing}$ are all empty.  
\item  Since $B_r = \emptyset$, the first part of Hypothesis  \ref{hyp:no_Maslov_0_disc_Maslov_1} is immediate;  the 3rd part of Hypothesis  \ref{hyp:no_Maslov_0_disc_Maslov_1} follows from Lemma \ref{Lem:monotoneAbar}; and the 2nd part of Hypothesis  \ref{hyp:no_Maslov_0_disc_Maslov_1} follows from the explicit description of the Maslov two discs with boundary on a matching sphere obtained in Lemma \ref{lem:vanishingequivt}.
\end{itemize}

To clarify the construction of $\Mdbar$, consider the trivial fibration $p: \bP^1 \times \bP^1 \rightarrow \bP^1$ and fix disjoint sections $s, s'$ and a fibre $F_{\infty}$.  We blow up $k+1$ distinct points which all lie on $s'$.  There is an induced Lefschetz fibration $\pi: \Mdbar \rightarrow \bP^1$ with $(k+1)$ reducible fibres.  The fibration $\pi$ admits two disjoint sections $s_0, s_{\infty}$, given by the proper transforms of $s$ respectively $s'$; the former is disjoint from all the exceptional divisors, and the latter meets each once.   There is  an effective divisor representing $-K_{\Mdbar}$ given by $s_0 + s_{\infty} + 2F_{\infty}$, noting that $s_{\infty} = s_0 - \sum E_i$. For large $k$,  $-K_{\Mdbar}$ is not nef since $s_{\infty}$ has square $-(k+1)$, but the affine variety $M$ admits a nowhere zero holomorphic volume form which has simple poles along $D_0$, and double poles on $D_{\infty}$.  The component $s_0 \subset D_0$ has holomorphically trivial normal bundle and moves on $\Mdbar$.  The component $s_{\infty} \subset D_0$ does not move on $\Mdbar$, since it has negative square, but $D_0 \cap \Mbar$ moves, and we take $B_0 = \emptyset$. The divisor $D_{\infty} = F_{\infty}$ is clearly nef, and for sufficiently large $N$ there are ample divisors in the class $ND_{\infty} + D_0$.  Thus, we satisfy all parts of Hypothesis \ref{Hyp:Main}.

\subsection{Real spheres}\label{Sec:Real_Spheres}

Let $A_k$ denote the $A_k$-Milnor fibre
\begin{equation}
\{x^2 + y^2 + p_{k+1}(z) = 0\} \subset \bC^3,
\end{equation}
where $p_{k+1}$ is a degree $k+1$ polynomial with distinct roots $\xi_i$. 
Projection to the $z$-co-ordinate defines a Lefschetz fibration $\pi: A_k \rightarrow \bC$ with $k+1$ critical values $\xi_i$.  

Let $\bar{A}_k$ denote the properification of the fibration $\pi: A_k \rightarrow \bC$;  this is a Lefschetz fibration over $\bC$ with generic fibre $S^2$ and $k+1$ special fibres; it can be seen as an open subset of a blow-up of $\bP^1\times \bP^1$, minus a fibre over infinity of the resulting Lefschetz fibration over $\bP^1$.  We denote by $D_0 = D_{\pm}$ the divisor  $\bar{A}_k \backslash A_k$, comprising the disjoint union of sections at $\pm \infty$.  As indicated above,  $\bar{A}_k$ admits a monotone symplectic structure satisfying the conditions of Hypothesis \ref{Hyp:Main}, which we take (after completing the base $\bP^1 \backslash \{\infty\}$ to $\bC$) to come from blowing up the standard product form on $\bC\times \bP^1$.  The induced symplectic structure $\omega$ on the Stein open subset $A_k \subset \bar{A}_k$ has finite area fibres, hence is still not complete, but has completion symplectomorphic to the standard form obtained by restriction from $(\bC^3, \omega_{std})$.   

\begin{Lemma} \label{lem:MatchingSphere}
An embedded path $\gamma: [0,1] \rightarrow \bC$ with $\gamma(0), \gamma(1) \in \{\xi_i\}$ and $\gamma$ disjoint from $\{\xi_i\}$ except at the end-points defines an embedded Lagrangian sphere $L_{\gamma} \subset (A_k, \omega)$, well-defined up to Hamiltonian isotopy.
\end{Lemma}

\begin{proof} This is well-known for the complete symplectic structure, for instance as an application of  Donaldson's ``matching path" construction \cite[Section 16]{FCPLT}.  Each fibre of $\pi$ is symplectomorphic, with respect to $\omega$, to a disc cotangent bundle of $S^1$, hence contains a unique exact Lagrangian circle up to Hamiltonian isotopy. The vanishing cycles associated to the paths $\gamma|_{[0,1/2]}$ and $\gamma|_{[1/2,1]}$ are exact, hence exact isotopic. After a suitable deformation of the symplectic parallel transport over $\gamma$ -- which by exactness of the base does not change the global symplectic structure -- one can ensure that the vanishing cycles associated to the end-points of $\gamma$ co-incide precisely. \end{proof}

 In particular, taking $p_{k+1}$ to have real roots one finds that $A_k$ retracts to an $A_k$-chain of $k$ Lagrangian spheres, whence the name.  Henceforth take $p_{k+1}(z) = \prod_{j=1}^{k+1} (z-j)$.

\begin{Definition}
The segment of the real axis between two consecutive critical values $\{j, j+1\}$, $1 \leq j \leq k$, defines an embedded Lagrangian 2-sphere $L\subset A_k$, which we will refer to as a \emph{real} matching sphere.
\end{Definition}

There is an action of the braid group $\Br_{k+1}$ on $A_k$ generated by Dehn twists in Lagrangian matching spheres forming an $A_k$ chain.  The group $\Br_{k+1}$  injects into $\pi_0\Symp_{ct}(A_k)$ by the main result of \cite{khovanov-seidel}. By contrast, the  representation 
 \begin{equation}\Br_{k+1} \rightarrow \pi_0\Symp_{ct}(\bar{A}_k)\end{equation} factors through the symmetric group $\Sym_{k+1}$. Indeed, 
a real matching sphere is obtained from a symplectic cut of a product Lagrangian cylinder in $\bC \times \bP^1$, cf. \cite[Example 4.25]{Smith:quadric}.  Any such sphere is therefore contained in the blow-up of a ball, $B^4 \# 2\bar{\bP}^2 \subset \bar{A}_k$, which implies by \cite{Seidel:4dtwist}  that the squared Dehn twist is Hamiltonian isotopic to the identity.

\begin{Lemma} \label{Lem:monotoneAbar}
The divisor $D_0 \subset \bar{A}_k$ comprising the two sections at infinity is cohomologous to the first Chern class and the symplectic form. $\bar{A}_k$ is monotone of minimal Chern number 1. The only Chern number 1 spheres are the components of the singular fibres of the fibration $\pi$.
\end{Lemma}

\begin{proof} Any holomorphic sphere projects to a point under $\bar{A}_k \rightarrow \bC$.  Each component of a  singular fibre is the exceptional curve of a blow-up or the proper transform of the fibre through a point which was  blown up, in each case of Chern number one.
\end{proof}

We may therefore satisfy Hypothesis \ref{Hyp:regular} with $B_r = \emptyset$.

\subsection{Computing the Gromov-Witten invariant} \label{Sec:MilnorGW}
We now turn to Hypothesis \ref{Hyp:GWvanishes} in this example.

\begin{Lemma} \label{lem:real_sphere}
Each real matching sphere meets exactly four Chern 1 spheres in $\bar{A}_k$, two with intersection number $+1$ and two with intersection number $-1$. 
\end{Lemma}
\begin{proof}
Let $\gamma$ be a real matching path between critical values $\xi, \xi'$ of $\pi$. The only Chern one spheres meeting $L_{\gamma}$ are the four components of the fibres $\pi^{-1}(\xi)$, $\pi^{-1}(\xi')$. Since the fibre class itself is (homologically) disjoint from $L_{\gamma}$, the two components of any reducible fibre must have opposite intersection with $L_{\gamma}$, so it suffices to prove that the intersection number is non-trivial.  This can be checked by a local computation; we work in an open subset  $U \subset \bar{A}_k$ which is the restriction of the Lefschetz fibration to an open neighbourhood of the matching path $\gamma$.  Let  $D_{\pm}$ denote the two compactification divisors at infinity which are the irreducible components of $D_0$.  If we remove either of these from $U$, the resulting space, which contains only two irreducible rational curves $E_i$, is diffeomorphic to $\bC^2 \# 2\overline{\bC\bP}^2$ and the Lagrangian has homology class $\pm[E_1-E_2]$. The result follows. 
\end{proof}

The previous Lemma implies that if $GW_{1} = \sum_{\beta} GW_{1; \beta}$ is the chain swept by Chern one spheres in $\bar{A}_k$ over all possible homotopy classes $\beta$, then $GW_{1}|_L = 0 \in H^2(L)$ for any real matching sphere.  In fact, more is true:

\begin{Lemma} \label{GWAkvanishes}
The invariant $GW_{1} = 0 \in H^2(A_k)$.
\end{Lemma}

\begin{proof}
We work with the standard integrable complex structure $J$ on $\bar{A}_k$, which in particular makes projection $\pi: \bar{A}_k \rightarrow \bC$ holomorphic. The only Chern one spheres in $\bar{A}_k$ are the components of the reducible fibres of $\pi$. The normal bundle to any such component is $\mathcal{O}(-1)$, so these spheres are regular by Sikorav's automatic regularity criterion \cite[Lemma 3.3.1]{McD-S}.  

For each critical value $p$ of $\pi$, fix a vanishing path $\gamma_p$ from $p$ to infinity, with the property that distinct such paths  are pairwise disjoint. 
The cycle  $\sum_\beta GW_{1;\beta} \in H^2(\bar{A}_k)$  is  Poincar\'e dual to the union of the exceptional fibres of $\pi$, equipped with its natural orientation and of multiplicity one. This vanishes in $H^2(\bar{A}_k)$, being the boundary of the locally finite cycle  $\cup_p \pi^{-1}(\gamma_p)$, hence its image under restriction to $A_k$ also vanishes.
\end{proof}

Since $D_0^{sing} = \emptyset = D_r$, and $L\subset \bar{A}_{k}$ is monotone,  the remaining parts of Hypotheses \ref{Hyp:GWvanishes} and the first and last parts of Hypothesis \ref{hyp:no_Maslov_0_disc_Maslov_1} are immediate. The central part of Hypothesis \ref{hyp:no_Maslov_0_disc_Maslov_1} asserts transversality of certain fibre products.  One of these is empty (a fibre product with $B_0$), whilst the other two are the fibre product of the moduli space of Maslov 2 discs with boundary on $L$ with $D_0, D_0'$ respectively the fibre product of the space of Chern one spheres with $L$ itself.  The explicit description of all the Chern one spheres in $\bar{A}_k$ obtained in Lemma \ref{GWAkvanishes} shows transversality of $\scrM_1(\bar{A}_k|1) \times_{\bar{A}_k} L$.  The corresponding transversality for $\Mod{(0,1)}{1}(L)$ follows directly from the analogous description of all the Maslov index 2 discs with boundary on $L$, which we provide in Lemmas \ref{lem:vanishingequivt} and \ref{lem:WeightForOneLag} below.  

It follows that there is an \nc-vector field on $\scrF(A_k)$ defined by the general machinery of Section \ref{Sec:Generalities}.

 \subsection{Orientations}\label{Sec:Orientations}

Let $X$ be a real algebraic variety, meaning a complex algebraic variety equipped with an anti-holomorphic involution $\sigma$, and $L \subset X$ a smooth Lagrangian submanifold which is a component of the fixed point set of $\sigma$.  

Fixing a $Spin$ structure on $L$ determines an orientation on the moduli space of holomorphic discs with boundary on $L$ carrying interior and boundary marked points \cite{deSilva,FO3,FCPLT}. The involution $\sigma$, together with the action by conjugation on the domain,  yields a natural involution on this moduli space which was studied by Solomon in \cite{Solomon}. Whether this action preserves or reverses orientations can be analysed by considering the action on determinant lines of the linearised Cauchy-Riemann problem, and keeping track of the re-ordering of the input boundary marked points under complex conjugation of the domain. In particular, the following result is a special case of \cite[Proposition 5.1]{Solomon}.
\begin{Lemma} \label{lem:Solomon}
The action induced by $\sigma$ on the moduli space of holomorphic discs preserves orientation on the component with $k$ boundary marked points, $2$ interior marked points, and Maslov index $\mu$, if and only if
\begin{equation} \label{SolomonFormula}
\frac{\mu(\mu+1)}{2} + \frac{(k-1)(k-2)}{2} + k = 0 \mod 2.
\end{equation} \qed
\end{Lemma}

 Now consider the diagram
 \begin{equation}
 \begin{array}{ccc}
 \Mod{(0,1)}{1}(L) & \longrightarrow & \Mod{2}{1}(L) \\ 
 \downarrow & & \downarrow \\
 \Mod{(0,1)}{1} & \stackrel{\alpha}{\longrightarrow} & \Mod{2}{1} 
 \end{array}
 \end{equation}
relating the abstract moduli space of discs with two interior marked points and one boundary point $\Mod{2}{1}$, the space of such discs with boundary on a Lagrangian $L$, and the corresponding spaces when the two interior marked points are constrained to lie on the interval $[0,1) \subset \Delta$ as in our applications.  The map $\alpha$ is an inclusion onto a codimension one submanifold,  and complex conjugation reverses the co-orientation of its image.  More precisely,   let $D^4_{op}$ denote the space of distinct pairs $(D^2 \times D^2) \backslash \textrm{Diagonal}$.
 The moduli space of discs with two interior marked points and one boundary point,  not taken modulo automorphisms, is obviously in bijection with  $D^4_{op} \times S^1$. The complex conjugation map $(z,w,\theta) \mapsto (\bar{z}, \bar{w},-\theta)$ reverses the co-orientation of the interval $I\subset D^4_{op} \times S^1$ defined by taking the marked points to be $(0,t,1)_{t\in[0,1)}$.  
 
 Taking $\mu=2$ and $k=1$ in \eqref{SolomonFormula}, and incorporating this additional sign change, shows that:

\begin{Lemma} \label{lem:cancel}
The involution induced by $\sigma$ reverses orientation on the moduli space $ \Mod{(0,1)}{1}(L) $.  
\end{Lemma}

\subsection{Equivariant structure and purity} 
Let $L\subset A_k$ be a Lagrangian 2-sphere.   Since $L$ is (necessarily) exact, we have an isomorphism $HF^*(L,L) \cong H^*(S^2)$ of graded rings. Vanishing of $HF^1(L,L)$ implies that $L$ admits an equivariant structure, i.e. $b^0|_L = dc_L$ is the coboundary of some element $c_L \in CF^0(L,L)$.   The following Lemma is not strictly required for the proof of formality, but is instructive in view of the argument for purity of the equivariant structure in Lemma \ref{lem:WeightForOneLag}.

\begin{Lemma} \label{lem:vanishingequivt}
One can choose data defining $\scrF(A_1)$ such that $c_L$ vanishes identically for the unique real  matching sphere.
\end{Lemma}

\begin{proof}
Since the compactification divisor $D_0 = D_{\pm} \subset \bar{A}_1$ has components with trivial normal bundle, it is straightforward to write down push-offs $D_0'$. We take $p(z)$ to be a polynomial with real critical values $\{1,2\}$, so the $A_1$-space is given by $\{x^2+y^2+(z-1)(z-2)=0\}$.  Along the interval $[1,2] \subset \bR \subset \bC_z$ in the $z$-plane, $x^2+y^2<0$ and hence the real matching sphere is defined by $\{x,y \in i\bR\}$.  The antiholomorphic involution 
\begin{equation}
\iota: (x,y,z) \mapsto (-\bar{x}, -\bar{y}, \bar{z}) \quad \textrm{on} \ \bC^3
\end{equation}
preserves the hypersurface defining $A_1$ and fixes the real sphere $L$ pointwise.  
The involution naturally extends to the compactification $\bar{A}_k$, exchanging the two components of the compactification divisor $D_{\pm}$.  We choose the push-offs $D_0' = \cup D'_{\pm}$ to be parallel copies of $D_{\pm}$ which are also exchanged by the involution $\iota$.  We work with perturbations of the standard almost complex structure which make projection to $\bC$ holomorphic and which are also anti-equivariant with respect to $\iota$. 
\begin{figure}[ht]
\begin{center}
\includegraphics[scale=0.5]{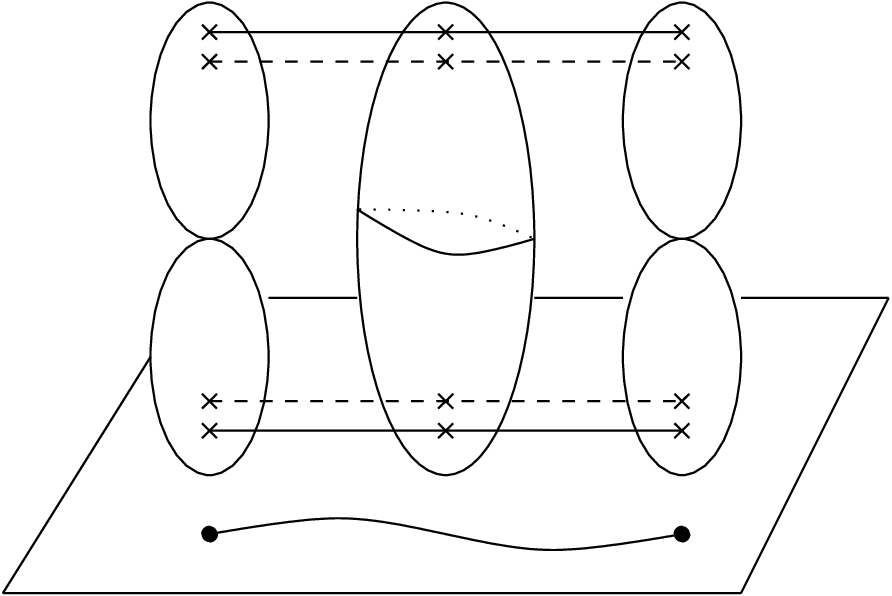}
\caption{Divisors $D_{\pm}$, $D'_{\pm}$ in the fibres over a real matching sphere}
\end{center}
\end{figure}
Equivariant transversality is not problematic since no disc is fixed pointwise by $\iota$, cf. \cite[Proposition 5.13]{khovanov-seidel}.

Lemma \ref{GWAkvanishes} implies that the 1-cochain $gw_{1}|_L$ is trivial, since the corresponding chain $\cup_p \pi^{-1}(\gamma_p)$ meets $L$ in a finite set (assuming that we choose the vanishing paths to be disjoint from the interior of the projection of $L$).  The chain $\beta_0$  introduced after Hypothesis \ref{Hyp:GWvanishes} and occuring in Lemma \ref{Lem:order-zero-part-of-CO} is also trivial, since the divisor $D_0$ is everywhere non-singular and $D_0\cap D_0' = \emptyset$.  Therefore, the Floer cocycle 
\[
b^0|_L = \tilde{b}^0|_L \, = \,  ev_*[\bar{\scrR}^1_{(0,1)}(L)]
\] of Lemma \ref{Lem:order-zero-part-of-CO}  is exactly the locus swept via evaluation at $1\in \partial \Delta$ by discs with two interior marked points, one at $0$ constrained to lie in $D_0$ and the other at a variable point on $(0,1) \subset \Delta$ which is constrained to lie on $D_0'$.  There are two such discs over every point of the matching path $\pi(L) \subset \bC$, coming from the two hemispheres of each fibre of $\pi$ lying over the matching path bound by the vanishing circle in that fibre.  Each family of hemispheres gives rise to an oriented arc  in $L$ between the two Lefschetz critical points in $L$.  

The $\iota$-invariance of $D_0, D_0'$ implies that the holomorphic discs in each fibre are exchanged up to conjugation by  $\iota$. 
Therefore the two families of discs give rise to identical arcs between the same end-points (the intersections with critical points of $\bar{A}_1\rightarrow \bC$). Lemma \ref{lem:cancel} implies the arcs carry opposite orientations. The corresponding singular 1-chains therefore cancel identically, which implies the desired vanishing of the equivariant structure on $L$.
\end{proof}

\begin{Remark} \label{Rem:Morseequivt}
Consider a Morse model for Floer cohomology $CF^*(L,L) = C^*_{Morse}(L)$.  If we pick a perfect Morse function $f: L \rightarrow \bR$, with minimum and maximum lying in the fibre over the mid-point of the associated matching path, then $CF^1(L,L) = 0$, hence $b^0|_L = 0$ as a chain and one can take $c_L = 0$.    
\end{Remark}

\begin{Lemma} \label{lem:WeightForOneLag}
$\scrF(A_k)$ admits an \nc-vector field for which 
real Lagrangians are pure: weight and grading co-incide on $HF^*(L,L)$ for each matching sphere $L$ in the $A_k$-chain.
\end{Lemma}

\begin{proof}
For any given real sphere $L \subset A_k$ there is some anti-holomorphic involution of $A_k$ which preserves $L$ pointwise. Weights on self-Floer-cohomology are independent of the choice of equivariant structure and hence of the choice of almost complex structure, so we can reduce to the situation where $L$ is fixed by an involution $\iota$ as in Lemma \ref{lem:vanishingequivt}. Exactly as in the proof of that Lemma, there are precisely two holomorphic discs passing through the generic point $p$ of $L$, namely the two hemispheres of the $\bP^1$-fibre of $A_k$ containing that point.  

To define the endomorphism $b^1$ on $CF^2(L,L)$, we should consider configurations as follows.  We rigidify the domain disc $\Delta$  by fixing the parametrization for which the origin is the interior marked point which will map to $D_0$ and the evaluation output is the point $1 \in \partial \Delta$.  We have an additional \emph{variable} input boundary marked point, which should be constrained to a generic point $p\in L$ representing the generator of $CF^2(L,L)$, and an additional interior marked point on the segment $(0,1) \subset \Delta$ which should map to $D_0'$.

Taking $p$ generic, it lies over an interior point of the matching path $\gamma$ defining $L$, and hence on the boundary of two hemisphere discs.  To actually perform the computation, it is helpful to  consider instead the parametrization of such a hemisphere so that $0 \in D_0$ and the fixed section $D_0'$ of the Lefschetz fibration lies on the real arc $(0,1)\subset \Delta$.  This rigidifies the domain;  the point $p$ then determines a unique modulus of the variable input boundary marked point, and evaluating at $1 \in \partial \Delta$ yields a point $q \in L$ lying on the same fibre of the projection $L \to \gamma$ as $p$.  The complex conjugate disc (i.e. the other hemisphere) defines another contribution to $b^1(p)$, in which the input point $p$ now lies at the conjugate point of the boundary, and evaluation at $1$ again yields exactly the point $q$, since we have chosen the two components of each of $D_0$ and $D_0'$ to be strictly exchanged by the involution.

Lemma \ref{lem:Solomon} and the subsequent discussion implies that for domains with \emph{two} input boundary marked points and two interior marked points constrained to $[0,1] \subset \Delta$, a pair of conjugate discs contribute with the \emph{same} sign to the relevant moduli space. This  implies that the two output contributions of the point $q$ obtained above occur with the same sign, and hence the weight is $\pm 2$ on $HF^2(L,L)$.  Reversing the orientation on $D_0$ replaces the \nc-vector field $b \in CC^1(\scrF(A_k),\scrF(A_k))$ by its negative $-b$, so we can assume that the weight is  $+2$.

The fact that the weight on $HF^0(L,L)$ is trivial since the \nc-vector field is a derivation (see Lemma \ref{lem:WeightGradings}), and $HF^0(L,L)$ is generated by the unit.
 \end{proof}

\begin{Remark} \label{Rem:Robust}
The above argument places a misleading emphasis on an apparently subtle compatibility of signs.  If the two discs contributing to the weight on $HF^2$ had the same sign, one could  twist any disk $u$ mapping to $\bar{A}_k$ by $(-1)^{u\cdot D_{-}}$, i.e. by a sign counting intersection number with a single component of the compactification divisor. This is compatible with all breaking, since we only count discs of Maslov index $2$. Discs inside the open locus $A_k$ are untwisted, hence the $A_{\infty}$-operations defining the Fukaya category are not themselves altered. Thus, one could correct the sign of the two discs contributing to $b^1|_L$ after the fact if needed.  In other words, the important point for this paper is not the specific sign computation carried out by Solomon, but rather the weaker fact that the sign changes if one changes the parity of the number of boundary marked points.
\end{Remark}

We now fix gradings on the Lagrangian compact core of $A_k$ so that all Floer gradings are symmetric, i.e. the unique intersection point of two adjacent matching spheres has Floer grading $1$ in both directions.  This is obviously possible: grade the first sphere arbitrarily, and then shift the subsequent spheres in order to achieve symmetry.   In fact, given two matching spheres $L_{\pm}$ lying respectively over arcs in the upper and lower half-plane, Figure \ref{Fig:LocalGrading} governs the grading of the generator of $HF(L_{+}, L_{-})$ corresponding to a transverse isolated intersection point,  with respect to the standard holomorphic volume form and the natural choices of graded structures on the two Lagrangians.  (Note that  \cite[Equation 6.5]{khovanov-seidel} implies that for $2$-complex-dimensional Milnor fibres the phase function of a matching sphere is given by twice the argument of the defining arc in $\bC$.)

\begin{center}
\begin{figure}[ht]
\includegraphics[scale=0.6]{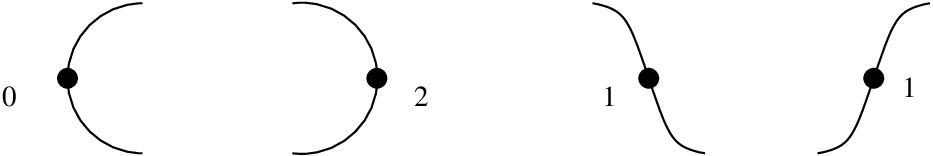}
\caption{Floer gradings for matching spheres in the Milnor fibre}
\label{Fig:LocalGrading}
\end{figure}
\end{center}

\begin{Lemma} \label{Lem:Shift}
Fix the zero equivariant structure on the first matching sphere of an $A_k$-chain.  There is a unique choice of equivariant structure on the remaining spheres such that the induced weights on $HF(L,L')$ and $HF(L', L)$ are symmetric for every $L, L'$.
\end{Lemma}

\begin{proof}
We know that for every pair $L, L'$ of intersecting Lagrangians, the weights on the individual Lagrangians are $\{0,2\}$ by Lemma \ref{lem:WeightForOneLag}.  Recall that weight gradings are compatible with Floer product. Moreover, if $L_i, L_{i+1}$ are adjacent real matching spheres, the product
\begin{equation} \label{Eqn:MilnorProduct}
HF(L_{i+1}, L_i) \otimes HF(L_i, L_{i+1}) \longrightarrow HF(L_i, L_i)
\end{equation}
hits the top class, by Poincar\'e duality.  Hence the weights of the unique intersection point in the two directions sum to 2 by Poincar\'e duality and the derivation property, Lemma \ref{lem:WeightGradings}.  The result can now be achieved by shifting the equivariant structures appropriately iteratively down the $A_k$-chain, appealing to \eqref{Eqn:ShiftWeights}.
\end{proof}

We can now recover the formality result of Seidel and Thomas from \cite{SeidelThomas}:

\begin{Corollary}
If $L_1,\ldots, L_{k} \subset A_k$ form the $A_k$-chain of real matching spheres, then the $A_{\infty}$-algebra $\oplus_{i,j} HF^*(L_i,L_j)$ is formal over characteristic zero fields.
\end{Corollary}

\begin{proof}
Both the Floer gradings and weights are symmetric, hence they necessarily co-incide, and formality follows from Corollary \ref{cor:Pure2}.
\end{proof}



\section{The symplectic arc algebra} \label{Sec:Slice}
We recall some features of the geometry of the space $\scrY_k$ from \cite{SS}, and Manolescu's embedding  \cite{Manolescu} of $\scrY_k$ into the Hilbert scheme of the Milnor fibre.  The space $\scrY_k$ contains a distinguished finite collection of Lagrangian submanifolds, whose Floer cohomology algebra defines the ``symplectic arc algebra" from the Introduction.  While later sections will focus on the Hilbert scheme description, there is one key result from \cite{SS} (recalled as Proposition \ref{Prop:Markov1} below) which we do not know how to prove from that viewpoint.

\subsection{The slice} \label{Sec:Slices}
 Fix an integer $k \geq 1$. Let $\Slice_k \subset \mathfrak{sl}_{2k}(\bC)$ be the affine subspace consisting of matrices of the form
\begin{equation} \label{eq:y-matrix}
A = \begin{pmatrix}
A_{1} & I &&& \\
A_{2} && I && \\
\dots &&& \dots & \\
A_{k-1} &&&& I \\
A_{k} &&&& 0
\end{pmatrix}
\end{equation}
with $A_1 \in \mathfrak{sl}_2(\bC)$, $A_j \in \mathfrak{gl}_2(\bC)$ for $j>1$, and where $I \in \mathfrak{gl}_2(\bC)$ is the identity matrix.

Let $\Sym_0^{2k}(\bC)$ be the subspace of the symmetric product $\Sym^{2k}(\bC)$ consisting of collections with center of mass zero. Symmetric polynomials yield an isomorphism $\Sym_0^{2k}(\bC) \iso \bC^{2k-1}$. Consider the adjoint quotient map 
\begin{equation} \label{Eqn:AdjointQuotient}
\chi: \Slice_k \rightarrow \Sym_0^{2k}(\bC)\end{equation}
 which takes a matrix $A$ to the collection of its eigenvalues. If we identify $\Sym_0^{2k}(\bC) \iso \bC^{2k-1}$ as before, this map is just given by the nontrivial coefficients of the characteristic polynomial. In our case,
\begin{equation} \label{eq:det2}
\begin{aligned}
& \det(x - A) = \det(A(x)), \quad \text{where} \\
& A(x) = x^k I - x^{k-1} A_1 - x^{k-2} A_2 - \cdots - A_k \in \mathfrak{gl}_2(\bC[x]).
\end{aligned}
\end{equation}

The part of $\chi$ lying over the open subset $\Conf^{2k}_0(\bC) \subset \Sym_0^{2k}(\bC)$ of configurations of unordered $2k$-tuples of pairwise distinct points is a differentiable fibre bundle. Fix some $t \in \Conf^{2k}_0(\bC)$, and denote the fibre of $\chi$ at that point by $\scrY_k^t = \scrY_k$. By definition, this is a smooth affine variety of complex dimension $2k$.

An important feature of the symplectic fibre bundle $\chi: \Slice_k \rightarrow \bC^{2k-1}$ is that it has fibred $A_1$-singularities at points  $t_0 \in \Conf^{2k}(\bC)$ 
where  a  pair of eigenvalues co-incide. 
The critical locus of the relevant fibre
corresponds to a subregular adjoint orbit, and is hence denoted by
$\mathcal{O}^{sub}$. 

\begin{Lemma}[Lemma 24 of \cite{SS}] \label{th:2-coincide}
Fix a disc $D \subset \bC^{2k-1}$ parametrising eigenvalues
$(\mu-\sqrt{z},\mu+\sqrt{z},\mu_3,\dots,\mu_{2k})$, with $z$ small. 
Then there is a neighbourhood of $\mathcal{O}^{sub}$ inside
$\chi^{-1}(D)$, and an isomorphism of that with a
neighbourhood of $\mathcal{O}^{sub} \times \{0\}^3$
inside $\mathcal{O}^{sub} \times \bC\,^3$, fitting into a commutative diagram
\[
\begin{CD}
 \chi^{-1}(D)  @>{\text{local $\mathrm{iso}$ defined near
     $\mathcal{O}^{sub} \cap 
\Slice_k$}}>> \mathcal{O}^{sub} \times \bC\,^3 \\
 @V{\chi}VV @V{a^2+b^2+c^2}VV \\
 D @>{\quad\qquad\qquad z \qquad\qquad\quad}>> \bC
\end{CD}
\]
where $a,b,c$ are coordinates on $\bC\,^3$.  \qed
\end{Lemma}


\begin{Lemma}[Lemma 24 of \cite{SS}]  \label{Lem:2eigenvalues}
The subspace $\left \{ y \in \Slice_k \ \big| \ \ker(y) \ \textrm{is \ 2-dimensional} \ \right\}$ is canonically isomorphic to $\Slice_{k-1}$ by an isomorphism which is compatible with the adjoint quotient $\chi$. \qed
\end{Lemma}

It follows that the critical locus $\mathcal{O}^{sub}  \times \{0\}^3$ of Lemma \ref{th:2-coincide} is exactly the fibre 
\begin{equation} \label{Eqn:IdentifyCriticalSet}
\chi^{-1}\{\mu_3,\ldots,\mu_{2k}\} = \scrY_{k-1} \subset \Slice_{k-1}.
\end{equation}

We have a natural $\bC^{*}$ action $\lambda$ on $ \Slice_k$, 
\begin{equation}
 \lambda_r:  r \cdot (A_1, \ldots, A_k) = (r \cdot A_1, \ldots, r^{k} \cdot A_{k})
\end{equation}
preserving the fibre over $0$.  The explicit slice \eqref{eq:y-matrix} is not the only possible choice;  the usual slices are obtained from the Jacobson-Morozov theorem, whilst a different explicit slice occurs in \cite{Mirkovic-Vybornov}.  A basic fact \cite[Lemma 14]{SS} is that any two $\lambda$-invariant slices are $\bC^*$-equivariantly isomorphic by an isomorphism which moves points only in their adjoint orbits.

\subsection{K\"ahler forms}\label{Sec:Kaehler}

Following \cite{SS}, $\Slice_k$ carries an exact K{\"a}hler form $\Omega_k=\Omega$. Fix some $\alpha>k$, and for each $1 \leq j \leq k$ choose a strictly subharmonic function $\psi_j: \bC \rightarrow \bR$, such that $\psi_j(z) = |z|^{\alpha/j}$ at infinity. Apply $\psi_j$ to each entry of $A_j$, and let $\psi$ be the sum of the resulting terms, which is an exhausting plurisubharmonic function. Then set
\begin{equation} \label{eq:original-omega}
\Omega = -dd^c\psi.
\end{equation}
$\Omega$ defines a symplectic connection on $\chi: \Slice \rightarrow \Conf^{2k}_0(\bC)$. Since the fibres are non-compact, one may not be able to integrate the associated horizontal vector fields to obtain parallel transport maps. 
The choice of K\"ahler form is motivated by three related requirements:
\begin{itemize}
\item It is asymptotically $\bC^*$-equivariant at infinity, \cite[Lemma 40]{SS}, meaning that $\lim_{r\rightarrow \infty} (\psi \circ \lambda_r)/r^{2\alpha} = \psi$.

\item Take the horizontal vector fields defined by $\Omega$  and add a large multiple of the fibrewise Liouville vector field dual to $\Theta = -d^c\psi$. Asymptotic $\bC^*$-equivariance implies that these corrected horizontal vector fields can be integrated  (yielding ``rescaled parallel transport" maps) on arbitrarily large compact subsets of a fibre \cite[Section 5]{SS}.

\item Under the Morse-Bott degeneration discussed in Lemma \ref{th:2-coincide} and the identification of \eqref{Eqn:IdentifyCriticalSet}, $\Omega_k$ restricts on the critical locus of the singular fibre to the form $\Omega_{k-1}$ on the smaller slice.
\end{itemize}

The existence of rescaled parallel transport is sufficient to associate to any closed exact Lagrangian submanifold $L \subset \scrY_k^{t_0}$ lying over any regular value $t_0 \in \Conf^{2k}_0(\bC)$ and an embedded path $\gamma: [0,1] \rightarrow \Conf^{2k}_0(\bC)$ with $\gamma(0) = t_0$ and $\gamma(1) = t_1$ a Lagrangian $\Phi_{\gamma}(L) \subset \scrY_k^{t_1}$, well-defined up to Hamiltonian isotopy. This is sufficient to conclude that the compact exact Fukaya category of a regular fibre of $\chi_k$ does not depend, up to quasi-equivalence, on the choice of regular fibre.

 Let $\chi_{loc} = a^2+b^2+c^2$ denote the map appearing on the right side of the diagram of Lemma \ref{th:2-coincide}, defined with respect to a holomorphic trivialisation of the normal bundle of the critical locus of a fibre where two eigenvalues co-incide.  Since $\chi_{loc}$ has non-degenerate Hessian in normal directions to the critical locus, the function re$(\chi_{loc})$ is Morse-Bott. The stable manifold of the negative gradient flow $-\nabla(\textrm{re}(\chi_{loc}))$ is the Lefschetz thimble of the degeneration, whose intersection with a regular fibre is (by definition) the vanishing cycle of the degeneration.  Since the gradient flow of re$(f)$ is the Hamiltonian flow of im$(f)$, in particular preserves the symplectic form, this vanishing cycle defines a co-isotropic subspace of the regular fibre $\scrY_k$, which is an $S^2$-bundle over the critical locus. 

There is a delicacy at this point, since (rescaled) parallel transport maps for $\Omega$ are defined on compact sets, but not on the entire fibre, which means that the co-isotropic is strictly only  defined as an $S^2$-bundle over compact subsets of the critical locus $\scrY_{k-1} = \mathrm{Crit}(\scrY_k^{t_0})$, where $t_0$ is a tuple of eigenvalues $(0,0,t_3,\ldots, t_{2k})$ with the $t_i \in \bC^*$ pairwise distinct. Nonetheless, this implies that any closed Lagrangian submanifold $L\subset \scrY_{k-1}$ has an associated vanishing cycle, diffeomorphic to $L \times S^2$, in a nearby smooth fibre $\scrY_k^t$ with $t=(-\epsilon, \epsilon, t_3,\ldots, t_{2k})$.

\subsection{The Hilbert scheme}

Let $S$ be a smooth quasiprojective complex surface. Let $\Hilb^{[k]}(S)$ denote the $k$-th Hilbert scheme of $S$, which is a smooth quasiprojective variety. There is a canonical Hilbert-Chow morphism
\begin{equation}
\Psi: \Hilb^{[k]}(S) \longrightarrow \Sym^k(S)
\end{equation} which is a crepant resolution of singularities.  Let $E$ denote the exceptional divisor of $\Psi$. 

\begin{Lemma} \label{lem:H2}
For $k\geq 2$, $H^2(\Hilb^{[k]}(S; \bR))\  \cong \ H^2(S^k;\bR)^{\Sym_k} \oplus \bR\langle E \rangle $. In particular, if $b_1(S)=0$ then $b_2(\Hilb^{[k]}(S)) = b_2(S) + 1.$
\end{Lemma}

\begin{proof}
See \cite[Chapter 6]{Nakajima}.
\end{proof}

For a class $a\in H^2(S;\bR)$ we will write $a_k$ for the class in $H^2(\Hilb^{[k]}(S))$ defined by the $\Sym_k$-invariant class $(a,\ldots,a,0) \in H^2(S^k)\oplus\bR\langle E\rangle$ via the isomorphism of Lemma \ref{lem:H2}. 

Suppose now $\pi: S \rightarrow \bC$ is a quasi-projective surface fibred over $\bC$. If $k=2$, the relative Hilbert scheme $\Hilb^{[2]}(\pi) \subset \Hilb^{[2]}(S)$, which away from Crit$(\pi)$ is  the second relative symmetric product along the fibres of $\pi$, is a divisor.  By definition, when $k>2$, the variety $\Hilb^{\pi,[k]}(S)$ is the complement of the divisor $\Hilb^{[k]}(\pi)$ which is the image of the finite map 
\begin{equation}
\Hilb^{[k-2]}(S) \times \Hilb^{[2]}(\pi) \rightarrow \Hilb^{[k]}(S).
\end{equation}
 Explicitly, the excluded divisor comprises all length $k$ subschemes whose projection to the base $\bC$ under $\pi$ has length $<k$.  We will refer, by abuse of notation, to this divisor as the ``relative Hilbert scheme"; it is not birational to the relative $k$-th symmetric product along the fibres, which would not be divisorial when $k>2$.

\begin{Lemma} \label{Lem:H2equality}
There is an identity $[\Hilb^{[k]}(\pi)] = -[E]/2 \ \in H^2(\Hilb^{[k]}(S);\bk)$.
\end{Lemma}

\begin{proof}
Consider first the general case of a fibred surface $\pi: S \rightarrow B$ over a complex curve $B$.  The fibre $\pi^{-1}(pt) = [F]$ defines a class in $H^2(S;\bR)$ and hence an associated class  $[F]_k \in H^2(\Hilb^{[k]}(S);\bR)$.  Since we are working over a field $\bk$, it suffices to evaluate $H^2$-classes on the Hilbert scheme on a basis of $H_2$, bearing in mind Lemma \ref{lem:H2}.  Given a curve $C\subset S$ we obtain a curve  $C_k \subset \Hilb^{[k]}(S)$ by adding $k-1$ distinct points to $C$ which we assume lie away from $C$ and in different fibres of $\pi$.  $C_k$ is trivially disjoint from the exceptional divisor $E$, and has algebraic intersection $(k-1)d$ with $[\Hilb^{[k]}(\pi)]$, where $d$ is the degree of the projection $\pi|_{C}$.  All the curves in the fibres of $\Psi$ are homologically proportional. For a rational curve in a fibre of $\Psi$ over a generic point of the diagonal, the intersection with $E$ is $-2$ (by crepancy) and with the relative Hilbert scheme is $+1$ (the intersection comes from the unique point of $\Hilb^{[2]}(\pi)$ lying over a point of the diagonal of $\Sym^2(S)$ which corresponds to a double point with a vertical tangency).  These computations, and Poincar\'e duality, yield an identity
\[
[\Hilb^{[k]}(\pi)] = (k-1)[F]_k - [E]/2.
\]
Now return to the special case $B = \bC$. In that case, the fibre $[F] = [\pi^{-1}(q)]$ of $\pi$ is trivial in $H_{2}^{lf}(S) \cong H^2(S)$, by considering the preimage of a half-line from $q$ to infinity, and hence the associated class $[F]_k \in H^2(\Hilb^{[k]}(S))$ also vanishes.   The result follows.
\end{proof}

\subsection{K\"ahler forms on the Hilbert scheme} \label{Sec:Varouchas}

The following result is due to Varouchas \cite{V1,V2} (see also \cite{Perutz}).

\begin{Lemma} \label{lem:Kahler}
If $\omega$ is a K\"ahler form on $S$, then for all $\epsilon > 0$ sufficiently small, there are K\"ahler forms on $\Hilb^{[k]}(S)$ in the class $[\omega]_k - \epsilon E$.  Moreover, one can find such K\"ahler forms which are product-like (i.e. which agree on pullback to $S^k$ with a product K\"ahler form) away from any given fixed analytic open neigbourhood of the big diagonal.
\end{Lemma}

\begin{proof}[Sketch]
 Suppose $(Z,\omega_Z)$ is K\"ahler, and let
 \begin{equation}
 \scrX \hookrightarrow Z \times B \stackrel{pr_B}{\longrightarrow B}
 \end{equation}
 be a flat family of $n$-dimensional complex analytic subvarieties of $Z$, parametrized by a complex analytic space $B$, which contains at least one smooth fibre.  Over the locus $B^0 \subset B$ where $\scrX_b$ is smooth, there is a Weil-Petersson form
 \begin{equation}
 \omega_B = \int_{\scrX/B} (pr_Z^*\omega_Z)^{n+1}|_{\scrX}.
 \end{equation}
 Since integration along the fibre commutes with base-change, this is both intrinsic and K\"ahler.  The form can be identified with the first Chern form of a determinant line bundle, cf. for instance \cite{BiswasSchumacher}. That bundle extends from $B_0$ to $B$;  the corresponding extension of the Weil-Petersson form exists as a closed positive  $(1,1)$-current, which \cite[Lemma 3.4]{V1} has a \emph{continuous} $\partial \cdbar$-potential.  A general smoothing lemma for continuous strictly plurisubharmonic functions \cite[Theorem 2.5]{V2} then yields a K\"ahler metric on $B$.  Away from the diagonals, the universal family $(X \times \Hilb^{[k]}(X)) \supset \scrX \rightarrow X$ is \'etale, and the Weil-Petersson form is product-like, which gives the Lemma.
\end{proof}

One can also obtain K\"ahler forms on the Hilbert scheme $\Hilb^{[k]}(S)$ by considering  a projective completion $\bar{S}$ of $S$ and constructing the compactification $\Hilb^{[k]}(\bar{S})$ as a  projective geometric invariant theory quotient, or by appealing  to Voisin's construction \cite{Voisin} of a symplectic form starting from a symplectic form on the underlying four-manifold.  The parameter $\epsilon$ of Lemma \ref{lem:Kahler} is the area of an exceptional sphere $\bP^1 \subset E$ contracted by the Hilbert-Chow morphism. If $\pi: S \rightarrow C$ is a fibration over a curve, then on the open subset $\Hilb^{\pi,[k]}(S) \subset \Hilb^{[k]}(S)$ the K\"ahler forms of Lemma \ref{lem:Kahler} become cohomologous, independent of $\epsilon$.

\subsection{Embedding the nilpotent slice}\label{sec:embed-nilpotent}

Fix a tuple of eigenvalues $\tau \in \Conf^{2k}(\bC)$ and a corresponding space  $\scrY_k^{\tau}$.  A point of $\scrY_k^{\tau}$ is given by a tuple of matrices $A_i$ satisfying 
\begin{equation}
\begin{aligned}
& \det(z - A) = \det(A(z)) = W_{\tau}(z), \quad \text{where} \\
& A(z) = z^k I - z^{k-1} A_1 - z^{k-2} A_2 - \cdots - A_k \in \mathfrak{gl}_2(\bC[z])
\end{aligned}
\end{equation}
where $W_{\tau}(z) = \prod_{t\in\tau}(z-t)$.  Write
\begin{equation}
A(z) = \left( \begin{array}{cc} P(z) & Q(z) \\ R(z) & S(z) \end{array} \right).
\end{equation}
Further, write $U(z) = (Q(z)+ R(z))/2$ and $V(z) = (Q(z)-R(z))/2i$.  The condition $\det A(z) = W_{\tau}(z)$ then becomes
\begin{equation} \label{Eqn:Reformulate}
U(z)^2 + V(z)^2 + W_{\tau}(z) = P(z) S(z).
\end{equation}
Let $A^{\tau}_{2k-1}$ denote the Milnor fibre
\begin{equation}
\{x^2 + y^2 + W_{\tau}(z) = 0\} \subset \bC^3.
\end{equation}
Then a point of $\scrY_k^{\tau}$ defines an ideal 
\begin{equation}
A(z) \ \mapsto \ (U(z)-x, V(z)-y, P(z)) \subset \bC[x,y,z]
\end{equation}
\eqref{Eqn:Reformulate} implies that the corresponding subscheme of $\bC^3$ is supported on $A^{\tau}_{2k-1}$. Moreover, the map which associates to $A(z)$ the ideal
\[
\left\{ \alpha(x,y,z) \ \big| \  \alpha(U(z),V(z),z)  \ \textrm{is divisible by} \ P(z) \right\} \subset \bC[x,y,z]/\langle x^2+y^2+W_{\tau}(z)\rangle
\]
then associates to any element of $\scrY_k^{\tau}$ a length $k$ subscheme of $A^{\tau}_{2k-1}$.

Let $\pi: A^{\tau}_{2k-1} \rightarrow \bC$ denote the canonical Lefschetz fibration defined by projection to the $z$ co-ordinate, which has nodal singular fibres over the points of $\tau$.  Let $\Hilb^{[k]}(A^{\tau}_{2k-1})$ denote the $k$-th Hilbert scheme of $A^{\tau}_{2k-1}$, which is a smooth quasiprojective variety of complex dimension $2k$.

\begin{Lemma}[Manolescu] \label{Lem:Manolescu}
There is a holomorphic open embedding $\scrY_k^{\tau} \hookrightarrow \Hilb^{[k]}(A^{\tau}_{2k-1})$.  The complement of the image is the relative Hilbert scheme of the projection $\pi$, i.e. the subschemes whose projection to $\bC_z$ does not have length $k$.
\end{Lemma}

\begin{proof}
The inclusion is the map defined above;  for injectivity (which implies openness) and the characterisation of the complement, see \cite{Manolescu}.  \end{proof}

We will henceforth identify $\scrY_k = \Hilb^{\pi, [k]}(A_{2k-1})$.

\subsection{Crossingless matchings} 

Fix the set $\tau = \{1,2,\ldots, 2k\} \subset \bR \subset \bC$.  A \emph{crossingless matching} on $k$ strands (for the set $\tau$) is any collection of $k$ pairwise disjoint arcs in $\bC$ which join the points of $\tau$ in pairs.   We tacitly identify isotopic matchings.

\begin{Lemma}
Any crossingless matching $\wp$ defines a Lagrangian submanifold  $L_{\wp} \cong (S^2)^k \subset \scrY_k$, well-defined up to Hamiltonian isotopy.  There are exactly $\frac{1}{k+1} {2k \choose k}$ crossingless matchings whose interiors lie in the upper half-plane.  
\end{Lemma}

\begin{proof}
The first statement follows from the iterated vanishing cycle construction discussed at the end of Section \ref{Sec:Kaehler}, and is essentially a  consequence of Lemmas \ref{th:2-coincide} and \ref{Lem:2eigenvalues}; see \cite{SS} for further details.  The  second statement is well-known, see \cite{Khovanov:Cohomology}. \end{proof}

\begin{Proposition} \label{Prop:Markov1}
The Lagrangian submanifolds $L_{\wp}$ and $L_{\wp'}$ for crossingless matchings $\wp, \wp'$ which differ by moves as in Figure \ref{Fig:Markov1} are Hamiltonian isotopic in $\scrY_k$.
\end{Proposition}

\begin{proof} This is \cite[Lemma 49]{SS}. \end{proof}

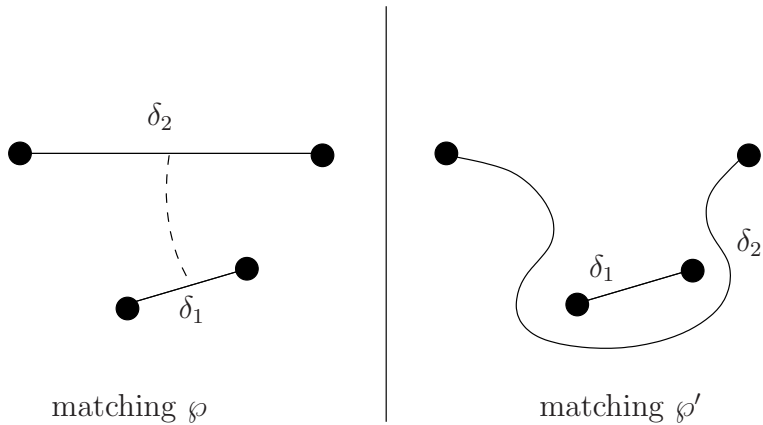
\begin{figure}[ht]
\begin{centering}
\begin{picture}(0,0)%
\includegraphics{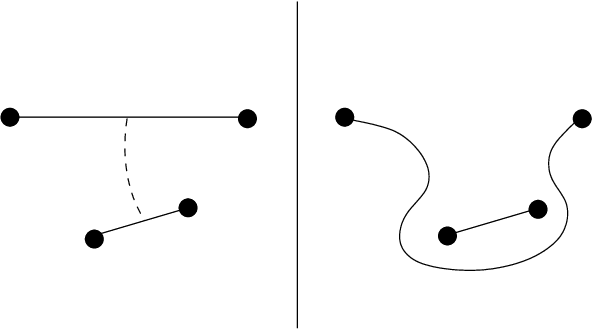}%
\end{picture}%
\setlength{\unitlength}{3947sp}%
\begingroup\makeatletter\ifx\SetFigFont\undefined%
\gdef\SetFigFont#1#2#3#4#5{%
  \reset@font\fontsize{#1}{#2pt}%
  \fontfamily{#3}\fontseries{#4}\fontshape{#5}%
  \selectfont}%
\fi\endgroup%
\begin{picture}(4736,2647)(1110,-2034)
\put(1989,-149){\makebox(0,0)[lb]{\smash{{\SetFigFont{12}{14.4}{\rmdefault}{\mddefault}{\updefault}$\delta_2$}}}}
\put(5689,-936){\makebox(0,0)[lb]{\smash{{\SetFigFont{12}{14.4}{\rmdefault}{\mddefault}{\updefault}$\delta_2$}}}}
\put(2189,-1361){\makebox(0,0)[lb]{\smash{{\SetFigFont{12}{14.4}{\rmdefault}{\mddefault}{\updefault}$\delta_1$}}}}
\put(1389,-1974){\makebox(0,0)[lb]{\smash{{\SetFigFont{12}{14.4}{\rmdefault}{\mddefault}{\updefault}matching
$\wp$}}}}
\put(4764,-1099){\makebox(0,0)[lb]{\smash{{\SetFigFont{12}{14.4}{\rmdefault}{\mddefault}{\updefault}$\delta_1$}}}}
\put(4451,-1974){\makebox(0,0)[lb]{\smash{{\SetFigFont{12}{14.4}{\rmdefault}{\mddefault}{\updefault}matching
$\wp'$}}}}
\end{picture}%
\caption{The handleslide move preserves Hamiltonian isotopy type\label{Fig:Markov1}}
\end{centering}
\end{figure}

There is a more direct construction of a Lagrangian submanifold $\hat{L}_{\wp} \subset \Hilb^{[k]}(A_{2k-1})$ from a crossingless matching.  Each arc of the matching $\wp$ defines a Lagrangian matching sphere in $A_{2k-1}$, and the crossingless condition implies that the product of these spheres is a Lagrangian in $(A_{2k-1})^k$ which is disjoint from the big diagonal.  Since the K\"ahler form on $\Hilb^{[k]}(A_{2k-1})$ is product-like on the complement of a small neighbourhood of the diagonal, this Lagrangian defines a Lagrangian submanifold of the Hilbert scheme. This is disjoint from the relative Hilbert scheme and hence lies inside the image of the embedding $\scrY_k \subset \Hilb^{[k]}(A_{2k-1})$.

\begin{Proposition} 
Under the symplectic embedding $\iota: \scrY_k \hookrightarrow \Hilb^{[k]}(A_{2k-1})$, the vanishing cycle Lagrangian $\iota(L_{\wp}) \simeq \hat{L}_{\wp}$ is Hamiltonian isotopic to the product Lagrangian.
\end{Proposition}

\begin{proof} This is part of \cite[Proposition 4.3]{Manolescu}. \end{proof}

We will henceforth denote by $L_{\wp}$ the Lagrangian associated to a crossingless matching, which one can define either by the vanishing cycle construction or by the more naive product construction\footnote{For the purposes of Floer cohomological computations the difference is irrelevant. When computing weight gradings later, it will matter that we specify our Lagrangian submanifolds exactly and not just up to isotopy, in view of Remark \ref{Rem:IsotopyCaveat}.  The naive product Lagrangians will then be better suited for importing the geometric set-up of Section \ref{Sec:Generalities}.}.   Two of the Lagrangian submanifolds $L_{\wp}$ play a special role: \begin{itemize}
\item The Lagrangian associated to the crossingless matching comprising a sequence of adjacent arcs joining the $\xi$ in pairs $\{1,2\}, \{3,4\}, \ldots, \{2k-1,2k\}$ is denoted $L_{\wpb}$; this is the \emph{plait} matching.  
\item The Lagrangian associated to the crossingless matching comprising a sequence of nested arcs joining the $\xi$ in pairs $\{1,2k\}, \{2,3\}, \ldots, \{2k-2,2k-1\}$ is denoted $L_{\wp_{\circ}}$; this is the \emph{mixed} matching.
\end{itemize}
\begin{center}
\begin{figure}[ht]
\includegraphics[scale=0.4]{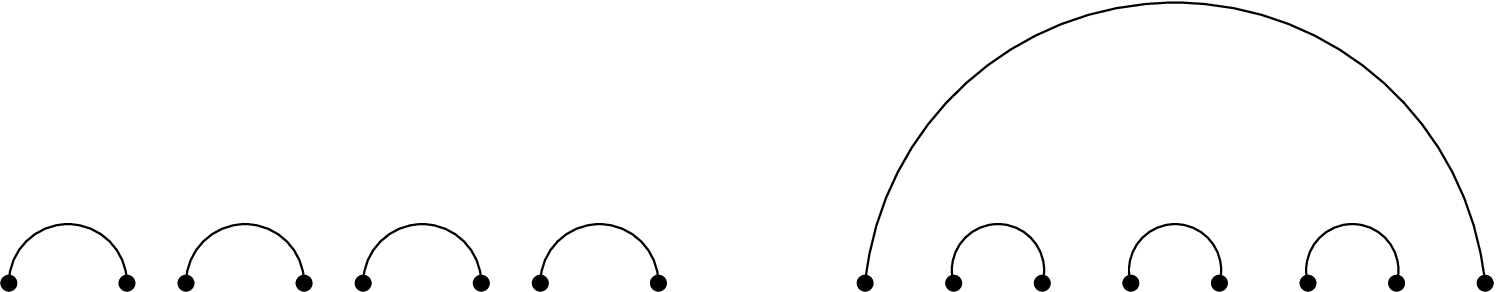}
\caption{The plait (left) and mixed (right) matchings of $2k$ points\label{Fig:PlaitMixed}.}
\end{figure}
\end{center}
There is a ``horseshoe" matching (joining pairs $\{1,2k\}, \{2,2k-1\}, \ldots, \{k-1,k+1\}$), denoted $\wp_{+}$ in \cite{SS}, which appears in the definition of symplectic Khovanov cohomology; that however will not play any special role in this paper.

\subsection{The symplectic arc algebra}

The \emph{symplectic arc algebra} is the $A_{\infty}$-algebra 
\begin{equation} \label{Eqn:SymplecticArc}
\oplus_{\wp,\wp'} HF^*(L_{\wp}, L_{\wp'})
\end{equation}
taking the sum over the finite set of upper-half-plane crossingless matchings, and computing Floer cohomology inside $\scrY_k$. 

 The group $HF^*(L_{\wp}, L_{\wp'})$ can be computed by placing the matching $\wp$ in the upper half-plane and $\overline{\wp'}$ in the lower half-plane, and taking the Floer cohomology of the corresponding Lagrangians.  Here one uses the handle-slide move of Proposition \ref{Prop:Markov1} to see that the Lagrangians $L_{\wp'}$ and $L_{\overline{\wp'}}$ for a given crossingless matching, placed in either the upper or lower half-plane, are Hamiltonian isotopic.

An intersection point of the Lagrangians $L_{\wp}$ and $L_{\overline{\wp'}}$, viewed as fibred over the corresponding upper respectively lower half-plane arcs, is given by a $k$-tuple of critical points $P \subset \{1,2,\ldots, 2k-1, 2k\}$ which are distributed evenly amongst the $k$ arcs of $\wp$ and the $k$ arcs of $\overline{\wp'}$. We will call such tuples admissible. 

Fix a (necessarily unknotted) component $C$ of the planar unlink $\wp\cup\overline{\wp'}$.
\begin{Lemma}
Any intersection point $P$ of the Lagrangians comprises either all the even critical points on $C$ or all the odd critical points on $C$.
\end{Lemma}
\begin{proof}
Relabel  the points of $C \cap \{1,\ldots, 2k-1\}$ starting with the left-most critical point as $1$ and then ordering the critical points clockwise along $C$, see Figure \ref{Fig:GradeSum}.   This relabelling preserves parity of critical points on $C$. Note that any arc of any matching joins points of opposite parity.   If an admissible $k$-tuple $P$ contains a point of even parity and a point of odd parity, then it contains some pair of consecutive integers, which (in the new labelling) necessarily both belong to a single arc in either $\wp$ or $\overline{\wp'}$. This violates admissibility. 
\end{proof}

Let $c(\wp,\wp')$ denote the number of components of the unlink $\wp \cup \overline{\wp'}$ obtained by reflecting $\wp'$ in the real axis and placing it  underneath $\wp$.   The above Lemma implies that the rank of $CF^*(L_{\wp}, L_{\overline{\wp'}})$ is given by $2^{c(\wp,\wp')}$.  We shall need to understand the grading on this Floer group. The Lagrangians $L_{\wp}$ and their intersections are disjoint from a neighbourhood of the diagonal in the Hilbert scheme, so one can take the holomorphic volume form on $\Hilb^{[k]}(A_{2k-1})$ to be induced from the product volume form on $A_{2k-1}$, compare to \cite[Lemma 6.2]{Manolescu}. This means that the gradings of intersection points can be computed by adding up local contributions governed by Figure \ref{Fig:LocalGrading}.  

Moreover, the grading of a generator of the Floer complex is the sum over components $C$ of $\wp \cup \overline{\wp'}$ of the contribution of the even or odd elements of that component. To this end, let $|C|$ denote the number of arcs of $\wp$ (equivalently $\overline{\wp'}$) appearing in $C$, so $\sum_C |C| = k$, where we sum over the distinct components of the unlink.  Moreover, given an admissible tuple $P$ we write $P_C$ for the subtuple of points lying on the component $C$.

\begin{Lemma}
The contribution of $P_C$ to the grading of $P$ viewed as a generator of $CF(L_{\wp}, L_{\overline{\wp'}})$ (from the upper half-plane to the lower half-plane matching) is $|C|- 1$ if the initial critical point of $C$ lies on $P_{C}$ and $|C|+1$ otherwise.
\end{Lemma}
 \begin{center}
\begin{figure}[ht]
\includegraphics[scale=0.4]{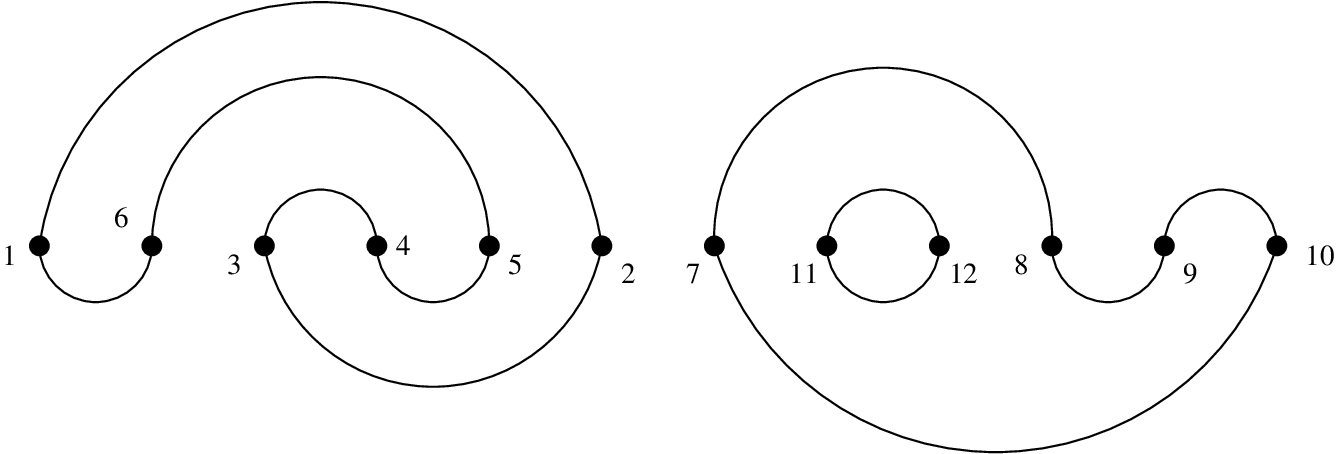}
\caption{Grading contributions from local Gauss map degree: the gradings are $\{2,4\}$, $\{1,3\}$ and $\{0,2\}$ for the left, right, nested components}
\label{Fig:GradeSum}
\end{figure}
\end{center}
\begin{proof}
For concreteness, we assume that the initial critical point of $C$ is odd; the even case follows from the same argument. We have oriented $C$ so that the Gauss map of $C$ has degree $+1$ with respect to the clockwise orientation of the unit circle in $\bC$.  Under the Gauss map, the preimage of the point $i\in S^1$ on the clockwise-oriented unknot $C$ is exactly the set of odd critical points on $C$, and the preimage of $-i \in S^1$  is the set of even critical points.  The local grading contribution of Figure \ref{Fig:LocalGrading} exactly says that an intersection point of a pair of matching arcs in $A_{2k-1}$ contributes to the Floer degree by $1\pm $(local contribution to the degree of the Gauss map), with the sign $-$ for an odd point and $+$ for an even point.  Summing over the $|C|$ arcs yields the result (see Figure \ref{Fig:GradeSum}). \end{proof}

\begin{Proposition}\label{Prop:GradeIt}
The graded vector spaces underlying the symplectic arc algebra are isomorphic to the corresponding graded vector spaces appearing in  Khovanov's arc algebra \cite[Section 2.4]{Khovanov:functor}. In particular:
\begin{itemize}
\item  As a relatively graded vector space, $HF^*(L_{\wp}, L_{\wp'}) \cong H^*(S^2)^{\otimes c(\wp,\wp')}$ with its natural relative grading.
\item The Lagrangians $L_{\wp}$ admit symmetric gradings, i.e. gradings with respect to which the Floer cohomology groups $HF^*(L_{\wp}, L_{\wp'})$ and $HF^*(L_{\wp'}, L_{\wp})$ lie in the same set of degrees, for every $\wp, \wp'$.
\end{itemize}
\end{Proposition}
\begin{proof}

Taking the tensor product over components $C$ of the unlink $\wp \cup \overline{\wp'}$, and adding gradings, it follows first that $CF^*(L_{\wp}, L_{\overline{\wp'}})$ is concentrated in a single mod 2 degree, hence the Floer differential vanishes, and that as a graded vector space $HF^*(L_{\wp}, L_{\overline{\wp'}}) \cong H^*(S^2)^{\otimes c(\wp,\wp')}[k-c(\wp,\wp')]$, living in symmetrically placed degrees $k-c(\wp,\wp') \leq * \leq k+c(\wp,\wp')$. That such gradings are symmetric in the sense of the Proposition now follows from Poincar\'e duality.
  \end{proof}

  \begin{rem}
 \cite[Proposition 5.6]{Reza2} derives an isomorphism of rings between the symplectic and combinatorial arc algebras, over the field $\bk = \bZ_2$, from a formal application of quilted Floer cohomology. 
  \end{rem}

\subsection{A cylindrical computation}

View the submanifolds $L_{\wp} \subset \scrY_k \subset \Hilb^{[k]}(A_{2k-1})$ as products of matching spheres  in $A_{2k-1}$.  Holomorphic polygons in the Hilbert scheme give rise to holomorphic curves in $A_{2k-1}$, which can be projected to $\bC$ and then have Lagrangian boundary conditions on the constituent arcs of crossingless matchings,  via the tautological correspondence between  
\begin{equation} \label{eqn:cylindrical-view}
\mathrm{maps} \  D^2 \rightarrow \Sym^k(A_{2k-1}) \qquad \mathrm{and \ diagrams} \ \ \begin{array}{ccc} \Sigma & \rightarrow & A_{2k-1} \\ \downarrow & & \\ D^2 & & \end{array}
\end{equation}
(with $\Sigma \stackrel{k:1}{\longrightarrow} D^2$ a degree $k$ branched covering). In the reverse direction, a map $\Sigma \rightarrow A_{2k-1}$ from a $k$-fold branched cover of the disc defines a map $u: D^2 \rightarrow \Sym^k(A_{2k-1})$. If this is not contained in the big diagonal, it lifts away from a finite set to the Hilbert scheme. That lift extends canonically to a map $\hat{u}: D^2 \rightarrow \Hilb^{[k]}(A_{2k-1})$, by taking the proper transform of $u$, viewing the Hilbert scheme as an iterated blow-up of the symmetric product, cf. \cite{Ekedahl-Skjelnes}.  Note that $\hat{u}$ is the only lift of $u$ to a  holomorphic map $D^2 \rightarrow  \Hilb^{[k]}(A_{2k-1})$, and its image lies in $\scrY_k$.  Moreover, all stable maps to $\Hilb^{[k]}(A_{2k-1}) $ which project to $u$ are obtained from  $\hat{u}$ by adding rational components in the Hilbert-Chow divisor, which meet the divisor $\Hilb^{[k]}(A_{2k-1}) \backslash \scrY_k$ non-trivially, so in fact $\hat{u} $ is the only lift of $u$ to a stable map to $ \scrY_k$. This leads to a viewpoint on the symplectic arc algebra akin to Lipshitz' cylindrical reformulation of Heegaard Floer theory \cite{Lipshitz}. 
\begin{Lemma} \label{Lem:Lift}
Let $p: S \rightarrow B$ be a complex surface smoothly fibred over a Riemann surface $B$, and let $C \subset S$ be a section of $p$.  There is a natural embedding $\Sym^k(C) \subset \Hilb^{p;[k]}(S)$ into the complement of the relative Hilbert scheme. 
\end{Lemma}

\begin{proof}
It suffices to consider the situation locally in the analytic topology, with $S = \bC^2 \rightarrow \bC = B$.  Take co-ordinates $(z,\xi) \in S$ with the map $p$ being projection to $z$, and consider $C = \{\xi=0\} \subset S$.  There is a $\bC^*$-action $t\cdot(z,\xi) = (z,t\xi)$ which fixes $C$ and induces an action on the Hilbert scheme. The fixed point set of that action is computed in \cite[Proposition 7.5]{Nakajima}, and has a connected component which is isomorphic to $\Sym^k(C)$.  More precisely, for a $k$-tuple of distinct points $\{z_1,\ldots, z_k\} \subset C$, the corresponding ideal is given by $\bigcap I_{z_i}$ with $I_{z_i} = ((z-z_i), (z-z_i)\xi) \subset \bC[z,\xi]$, whilst for a point $z_0$ of the small diagonal of $\Sym^k(C)$ the ideal is 
\[
((z-z_0)^k, (z-z_0)^{k-1}\xi, \ldots, (z-z_0)\xi^{k-1}, \xi^k) \ \subset \bC[z,\xi],
\] cf. \cite[Figure 7.3]{Nakajima} (the case at hand is when $k$ single box Young tableaux coalesce into one with a single row of width $k$ and height $1$).   The intersection of that subscheme with $\bC[z]$ is generated by the first element, hence the  subscheme still has length  $k$ after projection to $B$ (informally, none of the points come together with a vertical tangency).
\end{proof}

The planar unlink $\wp \cup \overline{\wp'}$ typically contains nested circles, and for ``cylindrical" Floer computations it is helpful to remove these. 

\begin{Lemma} \label{Lem:RemoveNesting}
For any pair $\wp, \wp'$ of upper half-plane crossingless matchings, there are matchings $\wp_{\flat}, \wp'_{\flat}$ (not in general contained in a half-plane) with the property that 
\begin{enumerate}
\item $L_{\wp} \simeq L_{\wp_{\flat}}$ and $L_{\wp'} \simeq L_{\wp'_{\flat}}$ are Hamiltonian isotopic;
\item the planar unlink $\wp_{\flat} \cup \wp'_{\flat}$ contains no nested components.
\end{enumerate}
\end{Lemma}

 \begin{center}
\begin{figure}[ht]
\includegraphics[scale=0.35]{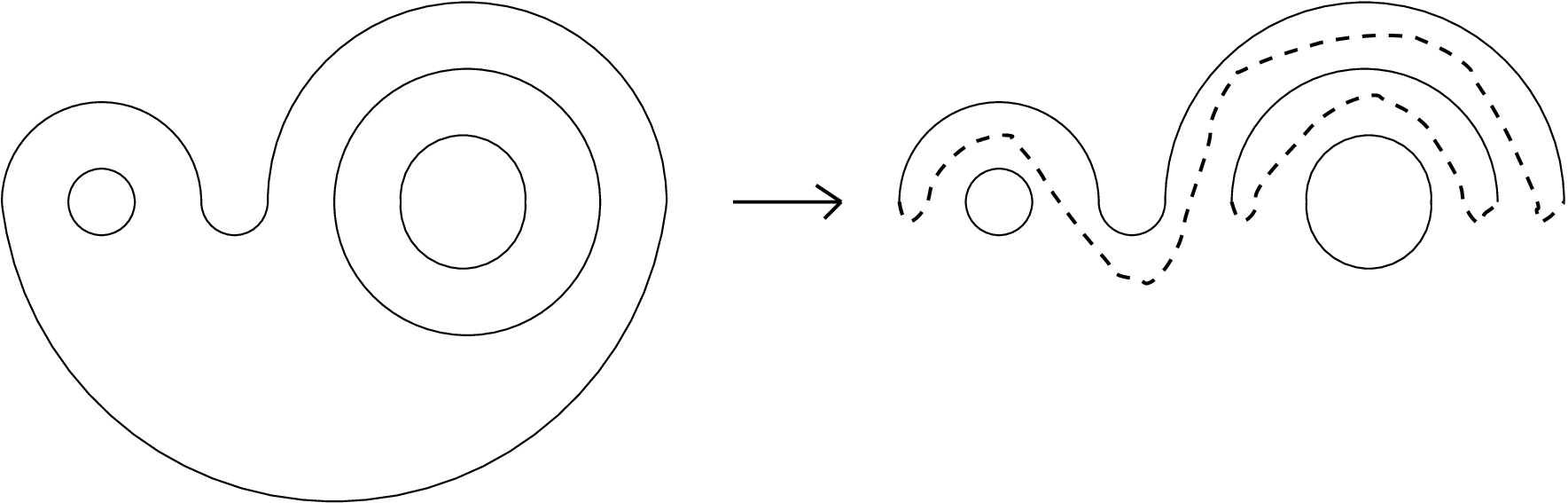}
\caption{After applying handle-slides, the planar unlink underlying $HF^*(L_{\wp}, L_{\wp'})$ has no nested components.}
\label{Fig:Markov1Slide}
\end{figure}
\end{center}

\begin{proof}
One can remove nested circles by applying the move of Proposition \ref{Prop:Markov1}, see Figure \ref{Fig:Markov1Slide}. In general, one argues inductively,   sliding across innermost discs to decrease the level of nesting. 
\end{proof}

The unlink $\wp_{\flat} \cup \wp'_{\flat}$ contains no nested components, hence comprises a bunch of circles bounding disjoint discs.  Each component of $\wp_{\flat} \cup \wp'_{\flat}$ arises by pairing the plait matching $\wp_{\bullet}$ with the mixed matching $\wp_{\circ}$ (from Figure \ref{Fig:PlaitMixed}), for some subset of the critical points; see Figure \ref{Fig:NonNested}.
 The upshot is that, from the point of view of computing Floer cohomology, the pair $L_{\wp}, L_{\wp'}$ can be replaced by products of these basic matchings inside a product of lower-dimensional Hilbert schemes of smaller Milnor fibres.
\begin{center}
\begin{figure}[ht]
\includegraphics[scale=0.4]{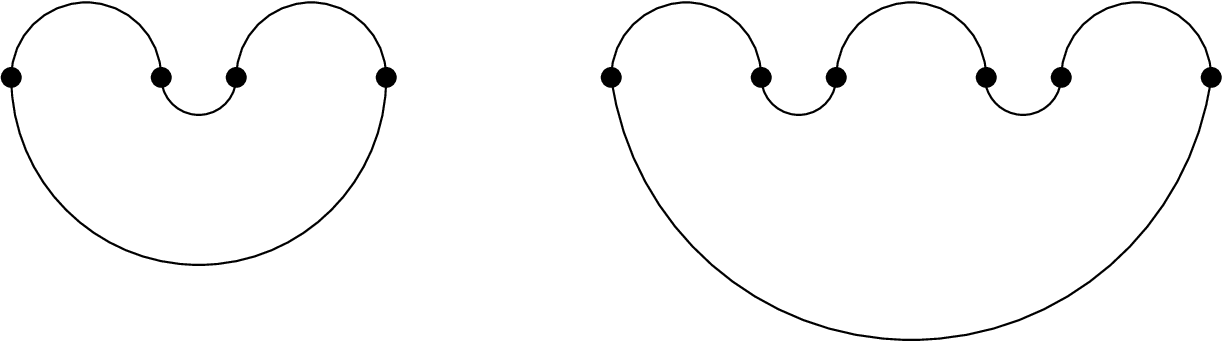}
\caption{After handle-slides, the unlink computing $HF^*(L_{\wp}, L_{\wp'})$ is made up of copies of the basic pieces as shown (for $k=2,3$).}
\label{Fig:NonNested}
\end{figure}
\end{center}

Next, we shall compute a non-trivial Floer product between the Lagrangians $L_{\wpb}$ and $ L_{\wp_{\circ}} $. To this end, it is illustrative to first consider the case $k = 3$. The only non-trivial graded components of $  HF^{*}(L_{\wpb}, L_{\wp_{\circ}}) $ are those of degree $2$ and $4$, so we consider the product
\begin{equation} \label{eq:product_plait_mixed_3}
  HF^{2}(L_{\wpb}, L_{\wp_{\circ}}) \otimes HF^{2}(L_{\wpb}, L_{\wpb}) \longrightarrow HF^{4}(L_{\wpb}, L_{\wp_{\circ}}). 
\end{equation}
We claim that this product is non-trivial; using the identification of $HF^{2}(L_{\wpb}, L_{\wpb}) $ with $H^2(S^2 \times S^2 \times S^2)$, we shall prove more specifically that the product with the class Poincar\'e dual to $ \{pt\} \times S^2 \times S^2$ is an isomorphism  (we would reach the same result by taking the codimension two cycle given by cyclically permuting the $S^2$ factors). The product counts holomorphic bigons with an additional boundary marked point whose image lies in $\{pt\} \times S^2 \times S^2 $. Projecting to the symmetric product, consider the diagram
\[
\begin{array}{ccc}
\Sigma & \rightarrow & A_{5} \\
\downarrow & &  \\
D^2 & \rightarrow & \Sym^k(A_{5})
\end{array}
\]
where $\Sigma \to D^2  $ is a branched triple cover, and $\Sigma$ is equipped with $3 \times 3 = 9$ marked points which are the inverse images of the $3$ marked points of $D^2$. Such a Riemann surface projects to a map to $\bC$ with boundary conditions the matchings shown in Figure \ref{Fig:DiskForProduct}; it must necessarily be a disc, as all higher genus curves map with degree greater than $1$ to $\bC$, and hence lift to curves which contribute to higher order operations in the Fukaya category of the Hilbert scheme.  Note furthermore that, the composite $\Sigma \to  A_5 \to \bC$ being degree $1$,  the map $\Sigma \to \bC$ defines a section over its image of the fibration $A_5 \to \bC$, hence any such map $D^2 \to \Sym^k(A_5)$ lifts to a map $D^2 \to \scrY_3 \subset \Hilb^{[3]}(A_5)$ as an application of Lemma \ref{Lem:Lift}.

\begin{center}
\begin{figure}[ht]
\includegraphics[scale=0.5]{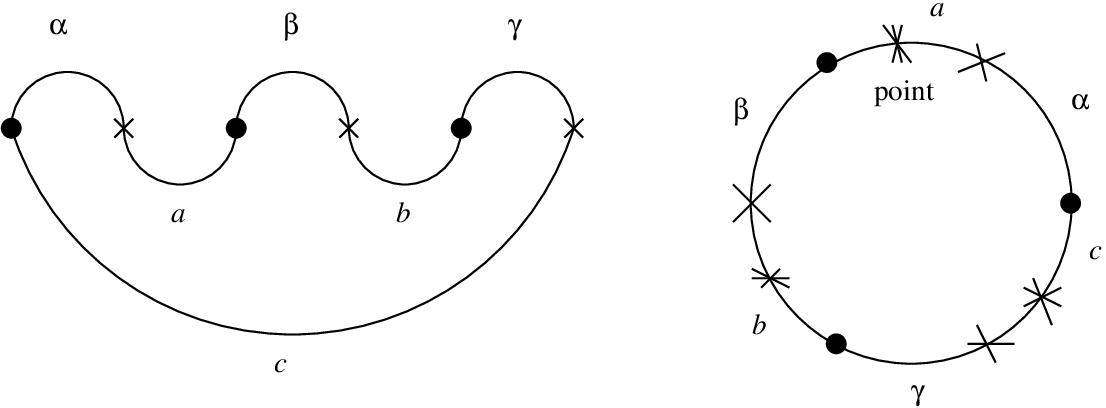}
\caption{A Holomorphic disc contributing to the module structure}
\label{Fig:DiskForProduct}
\end{figure}
\end{center}

Let $\scrB_{3} \subset \scrR^{9}$ denote the submanifold of the moduli space of discs with $9$ marked points which is obtained by the above procedure (i.e. comprising domains which are branched triple covers of a disc with $3$ marked points). In the setting at hand, $6$ marked points map to intersection points between the factors of $L_{\wpb} $ and $ L_{\wp_{\circ}} $, and one to a marked point $\{ pt\}$ in one of the factors of $ L_{\wpb}$. The remaining two marked points are unconstrained; forgetting them, we obtain a map
\begin{equation}\label{Eq:IsoOfModuliSpaces}
  \scrB_{3}  \to  \scrR^{7}.
\end{equation}
We shall presently see that this map has degree $1$, and is indeed an isomorphism. Assuming this, and that the curves arising from the moduli spaces on the two sides of \eqref{Eq:IsoOfModuliSpaces} contribute to the relevant $A_{\infty}$-products with the same signs (perhaps up to an overall global sign), we conclude that the count of holomorphic curves which contribute to the product in Equation \eqref{eq:product_plait_mixed_3} exactly corresponds to the count of curves which contribute to the higher $A_{\infty}$ operation $\mu^{6}$ among the constituent Lagrangians in the $A_{5} $ Milnor fibre.  The latter computation arose, for instance, in \cite[Lemma 4.13]{Smith:QuiverFukaya}, and the algebraic count is $\pm 1$ (this essentially comes from the $K$-theory relation satisfied by a closed cycle of matching spheres, compare to \emph{op. cit.} Equations 4.8 and 4.9).  

The previous discussion generalises  in a straightforward way to other values of $k$:
\begin{Lemma} \label{Lem:TopClassFor2-special}
If $\wpb$ and $ \wp_{\circ}$ are the plait and mixed matchings on $2k$ points respectively, the module maps
\begin{align} \label{eqn:module-for-two-special-matchings}
HF^{k-1}(L_{\wpb}, L_{\wp_{\circ}}) \otimes HF^{2}(L_{\wpb}, L_{\wpb}) & \longrightarrow HF^{k+1}(L_{\wpb}, L_{\wp_{\circ}}) \\
HF^{2}(L_{\wp_{\circ}}, L_{\wp_{\circ}}) \otimes HF^{k-1}(L_{\wpb}, L_{\wp_{\circ}}) & \longrightarrow HF^{k+1}(L_{\wpb}, L_{\wp_{\circ}}) 
\end{align}
are surjective.
\end{Lemma}

\begin{proof}
Let $\scrB_k  \subset \scrR^{3k}$ denote the submanifold of the moduli space of discs with $3k$ boundary marked points comprising domains obtained as a $k$-fold branched cover of a disc with $3$ marked points. It remains to show that, by forgetting all but one of the inverse images of one of the points in the disc downstairs, we obtain a map
\begin{equation}\label{Eq:IsoOfModuliSpaces2}
  \scrB_k \to \scrR^{2k+1}
\end{equation}
which  has degree $1$. We shall in fact prove that it is an isomorphism. Consider a collection of $2k-2$ points $1 < w_1 < z_1 <  \cdots < w_{k-1} < z_{k-1}$ on the boundary of the upper half-plane in $\bC$. The map 
\begin{equation}
  z \mapsto z \cdot \frac{ \prod_{i} z-z_i}{\prod_i z - w_i}
\end{equation}
defines a  branched cover of the upper half-plane (together with $\infty$) over itself, which takes the subset $\{0, z_1, \ldots, z_{k-1}\}$ to $0$ and $\{w_1, \ldots, w_{k-1}, \infty\}$ to $\infty$. Note that this map has degree $k$, and that there is a unique scaling of the above map which also maps $1$ to $1$. Moreover, the condition that $ w_i < z_i < w_{i+1} $ implies that $1$ has a unique inverse image in each interval $[z_i, w_{i+1}]$. 

Since every element of $\scrR^{2k+1} $  can be uniquely represented as the upper half-plane equipped with the marked points $(0, 1, w_1, z_1, \ldots, z_{k-1}, \infty)$, we conclude from the above that the map $ \scrB_k \to \scrR^{2k+1} $ is an isomorphism.  The final claim is that the corresponding $A_{\infty}$-products agree up to an overall sign, i.e. that the signs with which the curves in $\scrB_k$ respectively $\scrR^{2k+1}$ contribute differ by at most a global sign, not depending on the particular solution.  The reason is that signs in Floer theory depend on the map from the boundary of a given holomorphic disc to the Lagrangian Grassmannian (after stably trivialising the tangent bundle of the ambient manifold over the disc, a canonical choice of such stabilisation being given by the spin structure), which yields a trivialisation of the determinant line of the $\cdbar$-operator, together with choices of orientation on the moduli space of domains, and certain Koszul-type sign conventions.  The latter two choices do not depend on the particular holomorphic disc, so may introduce at most a global sign.  Using product spin structures and hence product orientation lines (i.e. stable trivialisations) for product Lagrangians shows that the  identification of moduli spaces induced by \eqref{Eq:IsoOfModuliSpaces2} respects signs in the sense that the relative signs for different solutions agree.  
This proves that the count of holomorphic triangles contributing to \eqref{eqn:module-for-two-special-matchings} is the same, up to global sign, as the count of holomorphic discs contributing to $\mu^{2k+1}$ in $ A_{2k-1}$, yielding the result.  
\end{proof}

To generalise the above result to arbitrary matchings, we note that the K\"unneth formula for Floer cohomology shows that
\begin{equation}
    HF^{k-c(\wp,\wp') }(L_{\wp}, L_{\wp'}) 
\end{equation}
is of rank $1$, and is the non-vanishing Floer group of minimal cohomological degree. Since the module structure of $ HF^{*}(L_{\wp}, L_{\wp'})  $ over the two self-Floer cohomology groups is compatible with the tensor product decomposition into factors arising from Lemma \ref{Lem:RemoveNesting}, we conclude that the map
\begin{equation} \label{eqn:module-for-two}
HF^{*}(L_{\wp'}, L_{\wp'}) \otimes HF^{k-c(\wp,\wp')  }(L_{\wp}, L_{\wp'}) \longrightarrow HF^*(L_{\wp}, L_{\wp'}).
\end{equation}
is surjective. Indeed, Lemma \ref{Prop:GradeIt} implies that every element of $HF^*(L_{\wp}, L_{\wp'}) $ can be expressed as a linear combination of tensor products of classes in the Floer cohomology of the plait and mixed Lagrangians corresponding to the components of $ \wp \cup \overline{\wp'}$ (or rather $\wp_{\flat} \cup \wp'_{\flat}$). By Lemma \ref{Lem:TopClassFor2-special} and the K\"unneth formula, any such class can be expressed as the product of a class in $ HF^{*}(L_{\wp'}, L_{\wp'}) $   with the minimal degree generator.

 We restate this more abstractly:
\begin{Corollary}  \label{Lem:TopClassFor2}
The Floer cohomology $HF^*(L_{\wp}, L_{\wp'})$ is a cyclic module over each of $HF^*(L_{\wp}, L_{\wp}) $ and $HF^*(L_{\wp'}, L_{\wp'}) $, generated by any non-zero minimal degree element. \qed
\end{Corollary}

We need a further computation of the Floer product involving three Lagrangians. Let $L, L', L''$ be closed exact Lagrangian submanifolds of an exact symplectic manifold which meet pairwise transversely.   Suppose $p \in L \cap L' \cap L''$ is an isolated point of the triple intersection, which defines a non-trivial Floer cocycle in each of the groups $HF(L,L')$, $HF(L', L'')$ and $HF(L, L'')$.  There is a constant holomorphic triangle at $p$, which is regular as a polygon with cyclically ordered boundary conditions $(L, L', L'')$ for one of the two cyclic orders, cf.  the right hand side of Figure \ref{Fig:ConstTriangle} (perturbing the three lines creates a non-trivial triangle in $\bC$, which is holomorphic or antiholomorphic depending on the cyclic order of the boundary Lagrangians).

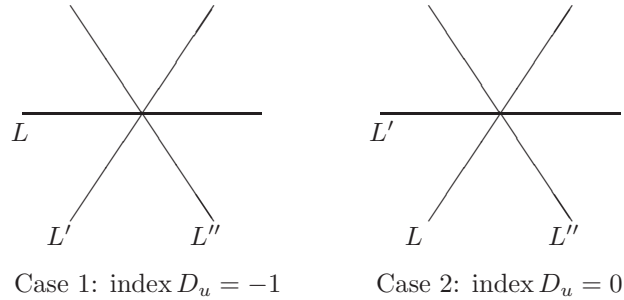
\begin{figure}[ht]
\setlength{\unitlength}{3947sp}%
\begingroup\makeatletter\ifx\SetFigFont\undefined%
\gdef\SetFigFont#1#2#3#4#5{%
  \reset@font\fontsize{#1}{#2pt}%
  \fontfamily{#3}\fontseries{#4}\fontshape{#5}%
  \selectfont}%
\fi\endgroup%
\begin{picture}(3837,1861)(226,-1310)
\put(2520,-1261){\makebox(0,0)[lb]{\smash{{\SetFigFont{10}{12.0}{\rmdefault}{\mddefault}{\updefault}{Case 2: $\mathrm{index} \, D_u = 0$}%
}}}}
\thinlines
{\put(601,-811){\line( 2, 3){900}}
}%
{\put(301,-136){\line( 1, 0){1500}}
}%
{\put(2851,539){\line( 2,-3){900}}
}%
{\put(2851,-811){\line( 2, 3){900}}
}%
{\put(2551,-136){\line( 1, 0){1500}}
}%
\put(226,-306){\makebox(0,0)[lb]{\smash{{\SetFigFont{10}{12.0}{\rmdefault}{\mddefault}{\updefault}{$L$}%
}}}}
\put(451,-961){\makebox(0,0)[lb]{\smash{{\SetFigFont{10}{12.0}{\rmdefault}{\mddefault}{\updefault}{$L'$}%
}}}}
\put(1351,-961){\makebox(0,0)[lb]{\smash{{\SetFigFont{10}{12.0}{\rmdefault}{\mddefault}{\updefault}{$L''$}%
}}}}
\put(2476,-306){\makebox(0,0)[lb]{\smash{{\SetFigFont{10}{12.0}{\rmdefault}{\mddefault}{\updefault}{$L'$}%
}}}}
\put(2701,-961){\makebox(0,0)[lb]{\smash{{\SetFigFont{10}{12.0}{\rmdefault}{\mddefault}{\updefault}{$L$}%
}}}}
\put(3601,-961){\makebox(0,0)[lb]{\smash{{\SetFigFont{10}{12.0}{\rmdefault}{\mddefault}{\updefault}{$L''$}%
}}}}
\put(246,-1261){\makebox(0,0)[lb]{\smash{{\SetFigFont{10}{12.0}{\rmdefault}{\mddefault}{\updefault}{Case 1: $\mathrm{index}\,  D_u = -1$}%
}}}}
{\put(601,539){\line( 2,-3){900}}
}%
\end{picture}%
\caption{Indices of constant holomorphic triangles\label{Fig:ConstTriangle}}
\end{figure}

When regular, i.e. of index 0, the constant triangle contributes to  
the Floer product
\begin{equation} \label{Floerprod2}
HF^*(L',L'') \otimes HF^*(L,L') \longrightarrow HF^*(L,L'')
\end{equation}

\begin{Lemma} \label{Lem:onlyconstant}
The constant triangle is the only contribution to the coefficient $\langle \mu^2(p,p), p\rangle$ of the  Floer product \eqref{Floerprod2}.  
\end{Lemma}

\begin{proof}
By hypothesis, the constant triangle does contribute to this product.  By exactness of the Lagrangians, the area of any contributing polygon is controlled by the action values of the intersection point $p$ viewed as a Floer generator for the three groups.  Therefore if a constant triangle contributes, only area zero and hence constant triangles can contribute.
\end{proof}

Lemma \ref{Lem:onlyconstant} applies to a triple of matching spheres with a common end-point in the $A_{2k-1}$ Milnor fibre; in that case, the matchings should be ordered clockwise, locally given by the thimbles for paths $(i\bR_{\geq 0}, \bR_{\geq 0}, i\bR_{\leq 0})$ when the critical point lies at $0\in\bC$.  By taking products and appealing to the K\"unneth theorem, the same situation is relevant for triples of crossingless matching Lagrangians in the Hilbert scheme. 

\begin{center}
\begin{figure}[t]
\includegraphics[scale=0.4]{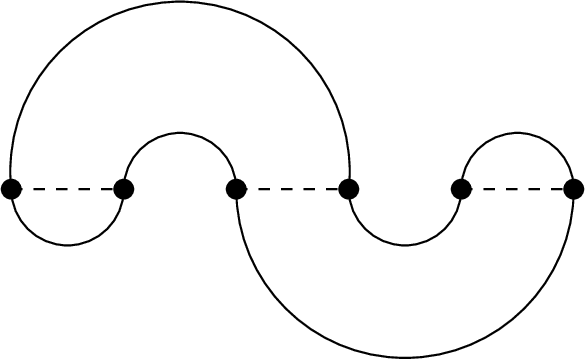}
\caption{A constant holomorphic triangle}
\label{Fig:ConstTriangleMatching}
\end{figure}
\end{center}

\begin{Corollary}\label{cor:hittopclass}
For any $\wp, \wp'$, the rank one subspace  of $HF^*(L_{\wp}, L_{\wp'})$ of largest cohomological degree lies in the image of the product 
\begin{equation} \label{eqn:keyproduct}
HF^*(L_{\wpb}, L_{\wp'}) \otimes HF^*(L_{\wp}, L_{\wpb}) \longrightarrow HF^*(L_\wp, L_{\wp'}).
\end{equation}
\end{Corollary}

\begin{proof}
Consider Figure \ref{Fig:ConstTriangleMatching}, in which we have drawn $L_{\wpb}$ dotted along the real axis, $\wp$ is the upper half-plane matching and $\overline{\wp'}$ the lower half-plane matching.
If the critical points of the Milnor fibre lie at $\{1,2,\ldots, 2k-1, 2k\} \subset\bC$, then the tuple $\{1,3,\ldots, 2k-1\}$ defines a point of the triple intersection $L_{\wp} \cap L_{\wp'} \cap L_{\wpb}$, which is transverse as an intersection point of any pair of such a triple of Lagrangians, and has at each separate point $2j+1$ the required clockwise cyclic order for that point to be regular as a constant map to the Milnor fibre.  
The discussion of Lemma \ref{Lem:onlyconstant} implies that the constant holomorphic triangle at that intersection point yields a non-trivial contribution to the product \eqref{eqn:keyproduct}, in particular that product does not vanish identically (recall that the Floer differentials vanish identically). 

Recall from Corollary \ref{Lem:TopClassFor2} that  $HF^*(L_\wp, L_{\wp'})$ is a cyclic module over $H^*(L_{\wp'})$, generated by an element $\alpha$ of minimal degree.   Suppose the element $v\cdot \alpha$ lies in the image of \eqref{eqn:keyproduct}, for $v \in H^*(L_{\wp'})$.  (In general, for degree reasons, it will not be true that the minimal degree element $\alpha$ itself lies in the image of \eqref{eqn:keyproduct}, so it may be that necessarily $v$ has strictly positive cohomological degree.)  Thus
\[
\mu^2(p',p) = v\cdot\alpha \qquad \textrm{for} \ p=\{1,3,\ldots,2k-1\} = p'
\] 
the given intersection point viewed as a generator of the two tensor factors of the domain of \eqref{eqn:keyproduct}, so $p\in  HF^*(L_{\wp}, L_{\wpb})$ and $p' \in HF^*(L_{\wpb}, L_{\wp'}) $. Cyclicity implies there is an element $w \in H^*(L_{\wp'})$ with the property that $(w\cdot v)\cdot \alpha$ lies in the rank one subspace of largest cohomological degree in $HF^*(L_{\wp}, L_{\wp'})$. Compatibility of the Floer product and the module structure implies that $\mu^2(w \cdot p',p) = w \cdot \mu^2(p',p)$, which implies the result.
\end{proof}

Corollary \ref{cor:hittopclass} is the key result we will later use to control weights of \nc-vector fields. It is worth remarking that it does not hold for arbitrary triples of crossingless matchings. For instance, there are triples of matchings in $\scrY_5$  indexing triples of components which pairwise have Floer cohomology of rank 2, living in degrees $4,6$. In that and similar cases, the Floer product cannot hit the top generator just for grading reasons.

\section{Formality via Hilbert schemes of Milnor fibres} \label{Sec:HilbScheme}

This section constructs a \nc-vector field on the space $\scrY_k$, by counting discs in the partial compactification $\Hilb^{[k]}(\bar{A}_{2k-1})$ according to the general scheme of Section \ref{Sec:Generalities}, and infers formality of the symplectic arc algebra. The \nc-vector field on $\scrY_k$ is essentially inherited from the \nc-vector field on $A_{2k-1}$ constructed in Section \ref{Sec:Milnor}.  

\subsection{Dictionary}  As at the start of Section \ref{Sec:Milnor}, we give a brief dictionary relating objects appearing in the rest of this section with their forebears in Section \ref{Sec:Generalities}. Let $Z$ denote the projective surface $\bar{\bar{A}}_{2k-1}$ which played the role of $\Mdbar$ in Section \ref{Sec:Milnor}, i.e. the blow-up of $\bP^1\times \bP^1$ at $2k$ points.  Recall that $Z$ contains divisors defined by sections $s_0, s_{\infty}$ of its Lefschetz fibration structure over $\bP^1$, and a fibre $F_{\infty}$ at infinity.  Comparing to Hypotheses \ref{Hyp:Main}, \ref{Hyp:regular}, \ref{Hyp:GWvanishes} and \ref{hyp:no_Maslov_0_disc_Maslov_1}, one now has:

\begin{itemize}
\item In Hypothesis \ref{Hyp:Main}, the projective variety is now $\Mdbar = \Hilb^{[k]}(Z)$.
\item $D_0$ is the divisor of subschemes whose support meets $s_0 \cup s_{\infty}$ in $Z$.  $D_{\infty}$ is the divisor of subschemes whose support meets $F_{\infty}$.
Finally,  $D_r \subset \Mdbar$ is the relative Hilbert scheme of the projection $Z \rightarrow \bP^1$.  Lemma \ref{Lem:Manolescu} implies that $M=\scrY_k$ and $\Mbar = \Mdbar \backslash D_{\infty} = \Hilb^{[k]}(\bar{A}_{2k-1})$.
\item  The fact that $D = D_0 \cup D_{\infty} \cup D_r$ supports an ample divisor with positive coefficients follows from Lemma \ref{lem:Kahler}, recalling Lemma \ref{Lem:H2equality}.
\item The natural map $\Hilb^{[k]}(Z) \rightarrow \Sym^k(Z)$ is crepant, and Chern zero spheres lie in its fibres, which underlies Lemma \ref{lem:chern_0_positive_D_r}.
\item In Hypothesis \ref{Hyp:regular} we define $B_r$ to be the locus of all points in $\Mdbar$ which lie on some stable Chern one sphere which meets $D_r$.  That this is codimension two and then yields the conclusion of Hypothesis \ref{Hyp:regular} is proved in Lemma \ref{Lem:BrTechnical} below.
\item The two parts of Hypothesis \ref{Hyp:GWvanishes}, on vanishing of the Gromov-Witten invariant, are Lemma \ref{Lem:homologous_to_distant} and  Corollary \ref{Cor:GWZeroHilb}.  Hypothesis  \ref{hyp:no_Maslov_0_disc_Maslov_1} is discussed at the end of Section \ref{Sec:Hypotheses}; the key input is Lemma \ref{Lem:bound_none}.
\end{itemize}

\subsection{Establishing the hypotheses} \label{Sec:Hypotheses}

We begin by establishing the remaining parts of Hypothesis \ref{Hyp:Main}.

\begin{Lemma}
Every irreducible component of the divisor $D_0 \cap \Mbar$  moves in its linear system, and the base locus Bs$|D_0| = \emptyset$. 
\end{Lemma}

\begin{proof}
That $D_0$ moves follows from the fact that both components  $s_0, s_{\infty} \subset \bar{A}_{2k-1} \subset Z$ have holomorphically trivial normal bundle and themselves move on $\bar{A}_{2k-1}$.  Taking deformations of the $s_i$ with appropriately chosen support shows that there are divisors linearly equivalent to $D_0 \cap \Mbar$ disjoint from any given zero-dimensional subscheme of $\Mbar$, hence the base locus is trivial. \end{proof}

\begin{Lemma}
$D_{\infty}$ is nef.
\end{Lemma}

\begin{proof}
The divisor $F_{\infty} \subset Z$ is nef, hence lies on the boundary of the ample cone.  Any ample line bundle $L \rightarrow Z$ on a projective surface induces an ample line bundle on $\Sym^k(Z)$ (descended from the exterior direct product $L \boxtimes \cdots \boxtimes L \rightarrow Z^k$), and hence by pullback under the Hilbert-Chow morphism a nef line bundle $L^{[k]}$ on $\Hilb^{[k]}(Z)$. Running this construction for a sequence of ample bundles converging to $F_{\infty}$ on $Z$, one sees that the line bundle $\mathcal{O}(F_{\infty})^{[k]} \rightarrow \Hilb^{[k]}(Z)$ canonically associated to $F_{\infty} \subset Z$ is in the closure of the nef cone of the Hilbert scheme, hence is nef. \end{proof}

\begin{Lemma}
There is a holomorphic volume form on $\Hilb^{[k]}(Z)$ with poles contained in $D_0 \cup D_{\infty}$ and simple poles on $D_0$.
\end{Lemma}

\begin{proof}[Sketch] This is a variant on Fogarty's theorem \cite{Fogarty} that the Hilbert scheme of a holomorphic symplectic surface is holomorphic symplectic.  The holomorphic volume form $\eta$ on $Z$ induces a product form on $Z^k$ and hence an algebraic volume form on $\Sym^k(Z)$ which lifts to a smooth form on the Hilbert scheme.  The resolution is crepant, so the only zeroes and poles are those on the symmetric product. 
\end{proof}

At this point, we have verified all parts of Hypothesis \ref{Hyp:Main}.

En route to computing the Gromov-Witten invariant counting Chern one spheres, our next task is to consider Hypothesis \ref{Hyp:regular}. 
The Hilbert scheme contains stable Chern one spheres which contain multiply covered Chern zero components, at which one cannot achieve transversality without virtual perturbations.  As in Section \ref{Sec:Generalities}, we avoid these issues by seeing that the relevant curves evaluate into a sufficiently high codimension subset of the total space.

Recall that $\overline{\scrY}_k = \Hilb^{[k]}(\bar{A}_{2k-1})$ is the Hilbert scheme of the quasi-projective surface $\bar{A}_{2k-1} \rightarrow \bC$.  We define the locus $B_r \subset \overline{\scrY}_k$ to be the points of $\overline{\scrY}_k$ which lie on the image of a stable Chern one sphere $u$ for which image$(u)$ meets the relative Hilbert scheme $D_r$. There is a natural map 
\begin{equation} \label{Eqn:Collapse}
\Hilb^{[k]}(\bar{A}_{2k-1}) \rightarrow \Sym^k(\bar{A}_{2k-1}) \rightarrow \Sym^k(\bC) =  \bC^k
\end{equation}
where the final arrow is induced from the Lefschetz fibration $\pi: \bar{A}_{2k-1} \rightarrow \bC$. Chern zero curves in $\overline{\scrY}_k$ must project to constants under \eqref{Eqn:Collapse}. Indeed, since non-constant rational curves in $\Sym^k(\bar{A}_{2k-1})$ have strictly positive Chern number,  the only Chern zero curves in $\overline{\scrY}_k$ lie inside the exceptional locus of the Hilbert-Chow morphism.  Such curves necessarily meet the relative Hilbert scheme, so it is certainly true that the locus
\[
ev_1^{-1}(\overline{\scrY}_k \backslash B_r) \ \subset \ \overline{\scrM}_{1}(\overline{\scrY}_k | 1)
\]
of Chern one stable maps with one marked point which evaluate into the complement of $B_r$ can be made regular by generic perturbation away from $D_r$. 

\begin{Lemma} \label{Lem:BrTechnical}
$B_r \subset \overline{\scrY}_k$ has complex codimension two, and maps under \eqref{Eqn:Collapse} into the big diagonal in $\Sym^k(\bC)$.
\end{Lemma}

\begin{proof}
Any stable rational curve $u$ (which is in particular a connected tree of rational curves) maps under \eqref{Eqn:Collapse} to a point, by the maximum principle in the base.  The relative Hilbert scheme $D_r$ maps into the big diagonal of $\Sym^k(\bar{A}_{2k-1})$, so the locus $B_r$ lies in the preimage of the big diagonal.  Since the map from the Hilbert scheme to the symmetric product is crepant, the image of $u$ in $\Sym^k(\bar{A}_{2k-1})$ must have Chern one. Rational curves in the symmetric product arise from surfaces  in $\bar{A}_{2k-1}$ by the tautological correspondence of \eqref{eqn:cylindrical-view}.  The only Chern number one connected closed curves in $\bar{A}_{2k-1}$ are the spheres which form components of  critical fibres of the projection $\pi$.  It follows that $B_r$ is contained in the codimension two subset which is the preimage in the Hilbert scheme of the locus in the symmetric product of tuples which both lie over the big diagonal of $\bC^k$ and lie in the divisor defined by having support meeting the critical values of $\pi$.  
\end{proof}

Lemma \ref{Lem:BrTechnical} and the paragraph preceding its statement together verify Hypothesis \ref{Hyp:regular}.

\begin{Lemma} \label{Lem:bound_none}
In fibred position, the  submanifolds $L_{\wp}$ are disjoint from $B_r$.  Moreover, for an almost complex structure making \eqref{Eqn:Collapse} holomorphic and product-like away from the big diagonal, these Lagrangians bound no non-constant holomorphic discs of vanishing Maslov index.
\end{Lemma}

\begin{proof}
 By definition, in fibred position  $L_{\wp}$ maps down via the map $\pi^k: \Hilb^{[k]}(\bar{A}_{2k-1}) \rightarrow \bC^k$ of  \eqref{Eqn:Collapse} to a product of arcs in $\bC^k$, defining a totally real cube disjoint from the big diagonal.  The first statement then follows from Lemma \ref{Lem:BrTechnical}.
 
  Recall from Section \ref{Sec:Real_Spheres} that the divisor $s_0 \cup s_{\infty} \subset \bar{A}_{2k-1}$ has components $D_{\pm}$ which are sections of $\bar{A}_{2k-1} \rightarrow \bC$.   Since $L_{\wp}$ lies far from the diagonal in the Hilbert scheme, by Lemma \ref{lem:Kahler} we can assume that the symplectic form is product like near $L_{\wp}$. Working with complex structures which make the projection $\pi^k$ holomorphic, the maximum principle implies that any holomorphic disc with boundary in $L_{\wp}$ lies in a fibre of $\pi^k$.  Such a disc arises from a tuple of discs in $\bar{A}_{2k-1}$ which lie in fibres of the projection $\pi$. The only non-constant such discs have positive transverse intersections with the sections  $D_{\pm}\subset \bar{A}_{2k-1}$ at infinity, along which the holomorphic volume form has a simple pole. Hence no non-constant disc has Maslov index zero.
\end{proof}

\begin{Lemma} \label{Lem:maslov_two_bounded}
In fibred position, and for an almost complex structure as in Lemma \ref{Lem:bound_none}, the only Maslov index two discs with boundary in $L_{\wp}$ are products of Maslov index two discs in the Milnor fibre with constant discs.
\end{Lemma}

\begin{proof}
Take a crossingless matching $\wp$.  We want to understand Maslov 2 discs in $\bar{\scrY}_k$ with boundary on $L_{\wp} = (S^2)^k \subset \Hilb^{[k]}(\bar{A}_{2k-1})$.  As in the proof of Lemma \ref{Lem:bound_none}, working with an almost complex structure making $\pi^k$ holomorphic, we see that all discs lie in fibres of $\pi^k$.  Since $\pi^k(L_{\wp})$ is disjoint from the diagonal in $\bC^k$, and the diagonal in the symmetric product of $\bar{A}_{2k-1}$ maps via \eqref{Eqn:Collapse} to the diagonal in $\bC^k$, we conclude that all discs with boundary on $L_{\wp}$ are disjoint from the diagonal in the Hilbert scheme.  

Over each interior point of the image cube $\pi^k(L_{\wp})$, $L_{\wp}$ comprises a product of equators in the fibre $\bP^1\times\cdots\times\bP^1$ of $\pi^k$.  Assuming the almost complex structure is product-like away from the diagonal,  Maslov index  two discs are non-constant in at most one factor.  We accordingly get a toric picture of every regular fibre, with $2^k$ Maslov 2 discs with boundary through a given generic point of $L_{\wp}$. All discs therefore come from discs in the $k$-fold product $ \bar{A}_{2k-1}^k $.  
\end{proof}

The first statement in Lemma \ref{Lem:bound_none} is exactly the first part of Hypothesis \ref{hyp:no_Maslov_0_disc_Maslov_1}.  Since discs with vanishing intersection number with $D_0$ also have vanishing Maslov index, the second statement in Lemma \ref{Lem:bound_none} establishes the 3rd part of Hypothesis \ref{hyp:no_Maslov_0_disc_Maslov_1}. 

Recall that the second part of Hypothesis \ref{hyp:no_Maslov_0_disc_Maslov_1} asserts the transversality of certain fibre products  $  \Mod{(0,1)}{1}(L_{\wp})$,    $ \scrR_{1}^{1}(\overline{\scrY}_k; (1,0)| L_{\wp}  ) \times_{\overline{\scrY}_k} B_{0}$ and $\scrM_{1}(\overline{\scrY}_k | 1) \times_{\overline{\scrY}_k} L_{\wp}$ under evaluation of spaces of Maslov index two discs with boundary on $L_{\wp}$, in the first two cases, or of Chern one spheres in $\overline{\scrY}_k$, in the last case, with appropriate cycles in $\overline{\scrY}_k$. The explicit description of the Maslov index 2 discs with boundary on $L_{\wp}$ obtained in Lemma \ref{Lem:maslov_two_bounded} shows that $  \Mod{(0,1)}{1}(L_{\wp})$ is transverse.   By definition, discs in $ \scrR_{1}^{1}(\overline{\scrY}_k; (1,0)| L_{\wp}  ) $ also have Maslov index 2, and have vanishing intersection number with $D_r$ so can contain no Chern zero components. For a push-off $D_0'$ of $D_0$ coming from a push-off $s_0' \cup s_{\infty}'$ of $s_0 \cup s_{\infty} \subset \bar{A}_{2k-1}$, the fibre product $ \scrR_{1}^{1}(\overline{\scrY}_k; (1,0)| L_{\wp}  ) \times_{\overline{\scrY}_k} B_{0}$ is also transverse, simply because it is empty: all Maslov index two discs lie in a subset where the divisors $D_0$ and $D_0'$ are disjoint, and $B_0$ is empty. 

Finally, 
we have established Hypothesis \ref{Hyp:regular} in our setting, so $L_{\wp}$ is disjoint from $B_r$ and  $\scrM_{1}(\overline{\scrY}_k | 1)$ is regular and of the correct dimension along $ev_1^{-1}(L_{\wp})$. No stable Chern one spheres in this locus can contain (any, in particular any multiply covered) Chern zero components.  The irreducible Chern one spheres meeting $L_{\wp}$ are again obtained from curves visible in the partially compactified Milnor fibre $\bar{A}_{2k-1}$, which yields transversality of  $\scrM_{1}(\overline{\scrY}_k | 1) \times_{\overline{\scrY}_k} L_{\wp}$.   This completes the verification of    Hypothesis \ref{hyp:no_Maslov_0_disc_Maslov_1}.

\begin{Lemma} \label{Lem:homologous_to_distant}
The intersection $B_0 = D_0 \cap D_0'$ is homologous (in the Borel-Moore homology of $D_0 $) to a locally finite cycle supported on $D_0^{sing} \cup (D_0 \cap D_r)$.
\end{Lemma}

\begin{proof}
We have a divisor $s_0\cup s_{\infty} \subset \bar{A}_{2k-1}$, and its push-off $s_0' \cup s_{\infty}'$.  We consider a one-parameter family of divisors $s_0^t \cup s_{\infty}^t$ which at $t=0$ agree with $s_0 \cup s_{\infty}$ and at $t=1$ with $s_0' \cup s_{\infty}'$.  One obtains divisors $(D_0')^t$ of $\Hilb^{[k]}(Z)$ by considering subschemes which meet $(s_0)^t \cup (s_{\infty})^t \subset Z$, for $t\in (0,1]$.  Consider a sequence of points $\xi_n$ in $D_0 \cap (D_0')^{t_n}$, for a sequence of values $t_n \to 0$.  The subschemes $\xi_n$ limit to a subscheme which meets $s_0 \cup s_{\infty}$ with multiplicity two, and hence either lies in the singular locus of $D_0$, or which lies in the relative Hilbert scheme.  

We now consider the chain in $D_0$ defined by the intersection of $D_0$ with the closure of the set of subschemes in $Z$ which meet both  $D_0$ and $(D_0')^t$ for some $t\in (0,1]$. At $t=1$ this chain has boundary $B_0 = D_0 \cap D_0'$, whilst the previous paragraph implies that all other boundary components are contained in $D_0^{sing} \cup (D_0 \cap D_r)$. This completes the proof.
\end{proof}

 Lemma \ref{Lem:homologous_to_distant} establishes the second part of Hypothesis \ref{Hyp:GWvanishes}. The first part, which is the vanishing of the Gromov-Witten invariant $GW_1$, is established in the next section.

\subsection{Computing the Gromov-Witten invariant, II} \label{Sec:GWII}

We know that $\bar{A}_{2k-1}$ contains spheres of Chern number 1, which are components $C_j$ of the singular fibres of $\bar{A}_{2k-1} \rightarrow \bC$.  It follows that $\Hilb^{[k]}(\bar{A}_{2k-1})$ also contains Chern 1 spheres, by taking the image in the Hilbert scheme of 
\begin{equation} \label{Eqn:productChern1sphere}
C_j \times \{q_2\} \times \cdots \times \{q_{k-1}\} \ \subset \bar{A}_{2k-1}^{k},
\end{equation}
where the $\{q_j\} \subset \bar{A}_{2k-1} \backslash C_j$ are distinct points, meaning the cycle of \eqref{Eqn:productChern1sphere} lies away from the diagonals. We will refer to such Chern 1 spheres as being of product type.

\begin{Lemma} \label{Cor:GWZeroHilb}
$GW_{1} |_{\scrY_k} = 0 \in H^2(\scrY_k)$.
\end{Lemma}

\begin{proof}
Recall the description of the second cohomology of $\Hilb^{[k]}(A_{2k-1})$ given in Lemma \ref{lem:H2}. Pick a collection $W$ of $k-1$ distinct points in $A_{2k-1}$ disjoint from a neighbourhood $U$ of its compact core of real matching spheres, and lying in different fibres of the Lefschetz fibration $A_{2k-1} \rightarrow \bC$. One obtains an embedding $U \to \Hilb^{[k]}(A_{2k-1})$, taking a point $z$ to the subscheme $z\amalg W$, which has image inside $\scrY_k$ and for which inclusion $U \to \scrY_k$ induces an isomorphism on $H^2$.  This gives an identification $H^2(\scrY_k) \cong H^2(A_{2k-1})$, from which one sees that the natural restriction map
\[
H^2(\scrY_k) \longrightarrow \oplus_{\wp} H^2(L_{\wp})
\]
(where we sum over upper half-plane matchings) is injective.   Alternatively, this follows from the  description of $H^*(\scrY_k)$ obtained in \cite{Khovanov:Cohomology}. It therefore suffices to show that $GW_1 |_{L_{\wp}} = 0 \in H^2(L_{\wp})$ for a Lagrangian fibred over a matching $\wp$. Any such Lagrangian is far from the diagonal and the relative Hilbert scheme, and meets only product Chern one spheres.  The result now follows from the corresponding vanishing theorem on the Milnor fibre. \end{proof}

All of $\scrY_k$, $\bar{\scrY}_k$ and $\Hilb^{[k]}(\bar{A}_{2k-1})$ have trivial first cohomology. Therefore, choices of nullhomology of $GW_{1} = 0 \in H^2(\scrY_k)$ are essentially unique. 

Lemma \ref{Cor:GWZeroHilb} establishes the first part of Hypothesis \ref{Hyp:GWvanishes}. 
At this point, we have established all of the hypotheses from Section \ref{Sec:Generalities}.  By Proposition \ref{prop:ConstructDilation}, a choice of bounding cochain for $GW_{1}$ now defines an \nc-vector field on the category $\scrF(\scrY_k)$ whose objects are the Lagrangian submanifolds $L_{\wp}$, which we consider as lying in fibred position.

\subsection{Purity for product Lagrangians}

Fix a bounding cochain $gw_{1} \in C^1({\scrY}_k)$ with
\begin{equation}
\partial (gw_{1}) = [ev_1 (\scrM_1(\Mbar | 1))]
\end{equation}
which exists since the cohomology class $[ev_1(\scrM_1(\Mbar|1))] = GW_{1} =0 \in H^2(\scrY_k)$ from Section \ref{Sec:GWII}. Since $GW_{1} = 0 \in H^2(\scrY_k)$ and the Lagrangians $L_{\wp}$ are simply-connected, they admit equivariant structures.

\begin{Lemma} \label{lem:SinglePure}
$L_{\wp}$ is pure, so weight is proportional to grading on $HF^*(L_{\wp},L_{\wp})$.
\end{Lemma}

\begin{proof}
Recall that weights on self-Floer-cohomology are independent of the choice of equivariant structure and  of almost complex structure.  For definiteness, if we work with a Morse model for $\scrF(\scrY_k)$ and equip the $L_{\wp}$ with perfect Morse functions, then $CF^1(L_{\wp}, L_\wp) = 0$ and hence the equivariant structure $c_{L_{\wp}}$ can be chosen to vanish identically for every $\wp$.

We work in the setting of Lemma \ref{Lem:maslov_two_bounded}, so all discs which are counted by the \nc-vector field are products, and are non-constant in at most one factor.   The signs of the non-constant discs are then governed by the results from Section \ref{Sec:Orientations} and Lemma \ref{lem:Solomon}. Since the construction is equivariant with respect to the symmetric group action, the choices of orientations on the divisor $s_0\cup s_{\infty}$ in each factor $\bar{A}_{2k-1}$ of $\bar{A}_{2k-1}^k$ which lead to positive signs for the product \nc-vector field, as in Lemma \ref{lem:WeightForOneLag}, descend to a choice of orientation on $D_0 \subset \Hilb^{[k]}(\bar{A}_{2k-1})$ so that the equivariant structure on $HF^2(L_{\wp},  L_{\wp})$ is given by multiplication by $+2$. The result follows.
\end{proof}

For each upper half-plane matching, we now choose the gradings as in Proposition \ref{Prop:GradeIt} so  $HF^*(L_{\wp}, L_{\wp'})$ is symmetrically graded in the sense that the groups $HF^*(L_{\wp}, L_{\wp'})$ and $HF^*(L_{\wp'}, L_{\wp})$ live in the same range of degrees  $k-c(\wp,\wp') \leq \ast \leq k+c(\wp,\wp')$.  Recall the plait matching $\wp_{\bullet}$ from  Figure \ref{Fig:PlaitMixed}.

\begin{Lemma} \label{lem:choose_equivariant_1}
Fix the equivariant structure $c_{L_{\wpb}} = 0$.  There is a unique choice of equivariant structure $c_{L_{\wp}}$ on $L_{\wp}$ for each other upper half-plane matching $\wp$ so that the endomorphism $b^1$ defined by the \nc-vector field agrees with the Euler vector field on the groups $HF^*(L_\wp, L_{\wpb})$ and $HF^*(L_{\wpb}, L_{\wp})$.
\end{Lemma}

\begin{proof}
Let $k-m_{\wp}$ denote the minimal degree of an element of $ HF^*(L_\wp, L_{\wpb})$. There is a unique equivariant structure on $ c_{L_{\wp}}$ so that the weight on $ HF^{k-m_{\wp}}(L_\wp, L_{\wpb})$ is $k-m_{\wp}$. We shall show that this choice satisfies the remainder of the desired properties.

First, cyclicity of $ HF^*(L_\wp, L_{\wpb}) $ as a module over $ HF^*(L_{\wpb}, L_{\wpb}) $, Corollary  \ref{Lem:TopClassFor2}, together with additivity of weights (as well as cohomological degrees) under multiplication, Equation \eqref{Eqn:WeightFloerProduct}, implies that the \nc-vector field agrees with the Euler vector field on $HF^*(L_\wp, L_{\wpb})$. In particular, the class of largest cohomological degree $k+m_{\wp}$ for $HF^*(L_\wp, L_{\wpb})$ has weight $k+m_{\wp}$.

Next, Poincar\'e duality implies that the top class $\eta_{\wp} \in H^{2k}(L_{\wp})$ lies in the image of multiplication
\begin{equation}
HF^{k-m_{\wp} }(L_{\wpb}, L_{\wp}) \otimes HF^{k+m_{\wp} }(L_{\wp},L_{\wpb}) \rightarrow HF^{2k}(L_{\wp},L_{\wp}) = H^{2k}(L_{\wp}).
\end{equation} 
The weight of $\eta_{\wp}$ is fixed by Lemma \ref{lem:SinglePure}, so the weight on $HF^{k-m_{\wp} }(L_{\wpb},L_{\wp}) $ must be $k-m_{\wp}$. Again using the cyclicity of the module $HF^*(L_{\wpb}, L_{\wp})$  over $HF^*(L_{\wp},L_{\wp})$, we conclude the desired result.  
\end{proof}

\begin{Proposition}
For the equivariant structures given by Lemma \ref{lem:choose_equivariant_1},  the endomorphism $b^1$ defined by the \nc-vector field agrees with the Euler vector field on the group $HF^*(L_\wp, L_{\wp'})$ for any pair $\wp$ and $\wp'$ of crossingless matchings.
\end{Proposition}

\begin{proof}
Consider the product 
\begin{equation}
HF^*(L_{\wpb}, L_{\wp'}) \otimes HF^*(L_{\wp}, L_{\wpb}) \longrightarrow HF^*(L_{\wp}, L_{\wp'}).
\end{equation}  Weight is  equal to grading for both groups in the domain of the multiplication, and the subspace of largest cohomological degree is in the image, by Corollary \ref{cor:hittopclass}. The eigenvalue of $b^{1}$ on this subspace is equal to its grading by  Lemma \ref{lem:WeightGradings} and  Equation \eqref{Eqn:WeightFloerProduct}. 

Next, we again appeal to 
 Corollary \ref{Lem:TopClassFor2}, namely that $HF^*(L_{\wp}, L_{\wp'})$ is a cyclic module over $H^*(L_{\wp})$, generated by any minimal degree element $\alpha$.  We may therefore write an  element of largest cohomological degree in $HF^*(L_{\wp}, L_{\wp'})$ as $\alpha \cdot v$, for $v \in H^*(L_{\wp})$.  Purity of $H^*(L_{\wp})$ was established in Lemma \ref{lem:SinglePure}, and hence the weight grading of $v$ agrees with its cohomological degree.  The additivity of weights, Lemma \ref{lem:WeightGradings} and \eqref{Eqn:WeightFloerProduct}, now fixes the weight of the minimal degree generator $\alpha$ to agree with its grading, since the weights of $v$ and $\alpha\cdot v$ are known to have this property.  Finally, cyclicity of the module structure shows that weight co-incides with grading for the whole of $HF^*(L_{\wp},L_{\wp'})$. This  establishes the desired purity, and concludes the argument.
 \end{proof}

\begin{Corollary} \label{cor:Arc-Pure}
If the characteristic of $\bk$ is zero, the symplectic arc algebra admits a pure structure, hence is formal as an $A_{\infty}$-algebra.
\end{Corollary}

 This completes the proof of Theorem \ref{thm:Formal} from the Introduction. The bigrading on the symplectic arc algebra arising from the weight decomposition is not  interesting, precisely because of purity.  However, for general elements $\beta \in \Br_{2k}$ of the braid group, the weight grading on $HF^*(L_{\wp}, \beta(L_{\wp}))$ afforded by the eigenspace decomposition for the linear part of the \nc-vector field $b^1$ does not \emph{a priori} reduce to information held by the cohomological grading.  This is obviously relevant to the relation between Khovanov and symplectic Khovanov cohomologies.

\bibliographystyle{alpha}

\end{document}